\documentclass[a4paper,10pt,twoside]{article}

\usepackage[utf8]{inputenc}
\usepackage[colorinlistoftodos,prependcaption,color=yellow,textsize=tiny]{todonotes}

\usepackage{amsmath}  
\usepackage{mathrsfs}  
\usepackage{amsthm}
\usepackage{amsfonts}
\usepackage{amssymb, latexsym}
\usepackage{mhequ}
\usepackage{verbatim}

\usepackage{mathptmx}
\DeclareMathAlphabet{\mathcal}{OMS}{cmsy}{m}{n} 

\usepackage{graphicx}
\usepackage{color}
\usepackage{xcolor}
\usepackage{tikz}
\usetikzlibrary{calc,intersections,through,backgrounds}
\usetikzlibrary{decorations.pathreplacing}
\usetikzlibrary{shapes}

\usepackage{float} 
\usepackage{enumerate}
\usepackage[margin=1.5in, marginparwidth=1.2in]{geometry}

\usepackage[pdfauthor={B. Gess, P. Tsatsoulis}, pdftitle={Lyapunov exponents and synchronisation by noise for systems of SPDEs}, pdftex]{hyperref}
\hypersetup{
  colorlinks=false,
  pdfborder={0 0 0}
  }
\usepackage{url}

\usepackage{tocloft}
\usepackage{fancyhdr}

\usepackage[toc,page]{appendix}

\newcommand{\eps}{\varepsilon}

\newcommand{\ee}{\mathrm{e}}
\newcommand{\ii}{\mathrm{i}}
\newcommand{\dd}{\mathrm{d}}

\newcommand{\dist}{\mathrm{dist}}
\newcommand{\Id}{\mathrm{Id}}
\newcommand{\supp}{\mathrm{supp}}
\newcommand{\error}{\mathcal{E}}

\newcommand{\RR}{\mathbb{R}}
\newcommand{\TT}{\mathbb{T}}

\newcommand{\V}{\mathbf{V}}

\newcommand{\C}{\mathcal{C}}
\newcommand{\Z}{\mathcal{Z}}
\newcommand{\N}{\mathcal{N}}
\newcommand{\M}{\mathcal{M}}
\newcommand{\T}{\mathcal{T}}
\newcommand{\Oo}{\mathcal{O}}
\newcommand{\Ss}{\mathcal{S}}

\newcommand{\Prob}{\mathbf{P}}
\newcommand{\Exp}{\mathbf{E}}

\newcommand{\lng}{\langle}
\newcommand{\rng}{\rangle}

\definecolor{darkpink}{rgb}{1,0,.6}

\theoremstyle{plain}
\newtheorem{theorem}{Theorem}[section]
\newtheorem*{theorem*}{Theorem}
\newtheorem{proposition}[theorem]{Proposition}
\newtheorem{corollary}[theorem]{Corollary}
\newtheorem{lemma}[theorem]{Lemma}

\theoremstyle{definition}
\newtheorem{definition}[theorem]{Definition}
\newtheorem{assumption}[theorem]{Assumption}

\theoremstyle{remark}
\newtheorem{remark}[theorem]{Remark}
\newtheorem{example}[theorem]{Example}

\numberwithin{equation}{section}

\title{Lyapunov exponents and synchronisation by noise for systems of SPDEs}
\author{Benjamin Gess, Pavlos Tsatsoulis}
\date{}

\fancypagestyle{plain}{

  \fancyhf{}
  \fancyfoot[L]{\footnotesize \today \, }
  \fancyfoot[C]{\thepage}
}

\renewenvironment{abstract}
{
\begin{center}
\begin{minipage}{.9\textwidth}\small\textbf{Abstract}\noindent
}
{
\end{minipage}
\end{center}
}

\newenvironment{keywords}
{
\begin{center}
\begin{minipage}{.9\textwidth}\small\textbf{Keywords}:\noindent
}
{
\end{minipage}
\end{center}
}

\newenvironment{msc}
{
\begin{center}
\begin{minipage}{.9\textwidth}\small\textbf{MSC 2020}:\noindent
}
{
\end{minipage}
\end{center}
}
\begin{document}

\maketitle

\begin{abstract}
Quantitative estimates for the top Lyapunov exponents for systems of stochastic reaction-diffusion equations are proven. The treatment includes reaction potentials with degenerate minima. The proof relies on an asymptotic expansion of the invariant measure, with careful control on the resulting error terms. As a consequence of these estimates, synchronisation by noise is deduced for systems of stochastic reaction-diffusion equations for the first time.  
\end{abstract}


\begin{msc} 60H15, 37H15, 37L30, 35K57
\end{msc}

\tableofcontents 

\section{Introduction}  

We study quantitative estimates on Lyapunov exponents and synchronisation by noise for vector-valued stochastic reaction-diffusion equations
of gradient type  
\begin{equs} \label{eq:rds_intro}
  \begin{cases}
    & (\partial_t- \kappa\Delta) u = -\nabla_u V(u) + \sqrt{2\eps} \xi \quad \text{on} \ \RR_{>0}\times \TT, 
    \\
    & u|_{t=0} = f,
  \end{cases}
\end{equs}
where $\TT$ denotes the $1$-dimensional torus of unit size, $\kappa$ is a positive parameter, $V:\RR^n\to \RR_{\geq0}$ is a non-negative potential, $\eps$ is a small positive 
parameter, $\xi$ is space-time white noise, and $f$ is a continuous initial condition. 

The problem of synchronisation by noise, for example, arises in Markov chain Monte Carlo sampling from the invariant measure formally associated to the SPDE \eqref{eq:rds_intro}. Indeed, Richardson-Romberg extrapolation used in the construction of higher order sampling algorithms leads to (discretisations) of the two-point motion. The analysis of its invariant measure(s), and, thus 
synchronisation by noise controls the possible increase of variance caused by these extrapolation algorithms, see Section~\ref{s:applications} below.

Systems of SPDEs \eqref{eq:rds_intro} and the task to sample from their invariant measures, appear in a series of applications, including Euclidean 
quantum field theories (e.g., Higgs standard model, $\sigma(N)$-model), scaling limits in continuous spin Ising models (e.g., $n$-vector model), stochastic multi-phase field models, and pattern formation (e.g., Swift Hohenberg equation). In fact, one of the motivations to introduce the quantisation of quantum field theory \eqref{eq:rds_intro} by Parisi and Wu in \cite{PW81} was the link to Monte Carlo sampling. We refer to Section~\ref{s:applications} below for more details on these applications.

Synchronisation by noise refers to the existence of a (weak) pull-back random attractor $\mathcal{A}(\omega)$ for \eqref{eq:rds_intro}, which consists of a single random point $\mathcal{A}(\omega) = \{a(\omega)\}$, see Definition~\ref{def:pull_back_attr} below. In this case, the long time dynamics are asymptotically globally stable in the sense that for every compact set $K\subset L^2(\TT;\RR^n)$ we have that\footnote{We write $u_{-t}(\cdot;f,\omega)$ for the solution to \eqref{eq:rds_intro} started at time $-t$, see Section~\ref{sec:ass_framework}.}
\begin{equs}
 \lim_{t\nearrow +\infty} \sup_{f\in K} \dist_{\|\cdot\|_2} \big(u_{-t}(0;f,\omega), \mathcal{A}(\omega)\big) = 0 \quad \text{in probability}.
 \label{eq:weak_attr_intro}
\end{equs}
In particular, for every compact set $K\subset L^2(\TT;\RR^n)$, $ \sup_{f,g\in K} \|u_{-t}(0;f) - u_{-t}(0;g)\|_{2} \to 0$ in probability
for $ t\nearrow +\infty$. 
We are particularly interested in cases for which the deterministic dynamics, that is, $\eps=0$ in \eqref{eq:rds_intro} are not asymptotically stable. A model example is given by the sombrero potential\footnote{also called Mexican hat potential, Coleman–Weinberg potential, or in Higgs potential in quantum field theory} $V(u) = \frac{1}{4} (|u|^2-1)^2$. The results of this work prove that for $\eps>0$ small enough, the stochastic dynamics to \eqref{eq:rds_intro} become asymptotically stable.

While synchronising effects of noise for scalar-valued SPDEs have been intensively investigated in the literature, the problem of 
synchronisation by noise for systems like \eqref{eq:rds_intro} has been open for a long time, and to the best of our knowledge there 
are no results available in the existing literature. This is due to the fact that the analysis for 
scalar-valued SPDEs is grounded on comparison principles, a tool not available for systems of equations. 

Therefore, a different line of argument is required. The present work builds upon the general theory for synchronisation by noise 
developed in \cite{FGS14}. This approach combines an analysis of local asymptotic stability in terms of the negativity of Lyapunov 
exponents, with asymptotic contraction (``contraction on large sets") and swift transitivity to deduce global asymptotic stability, 
namely synchronisation by noise. A major challenge in order to apply this theory to S(P)DEs is the proof of the negativity of the top 
Lyapunov exponent. 

The purpose of the present work is to prove quantitative estimates on the top Lyapunov exponent 
$\lambda_{\mathrm{top}}^{(\eps)}$ of \eqref{eq:rds_intro} defined in \eqref{eq:lyap} below, implying its negativity, in the 
regime of small noise intensity $\eps$, and a careful analysis of the required lower bound on the viscosity $\kappa$. The results obtained here address two dynamically quite distinct situations, expressed in 
terms of different assumptions on the potentials $V$, accompanied with the general Assumption~\ref{ass:V_general} on $V$:
\begin{enumerate}[a.]
 \item Non-degenerate minima (see Assumption~\ref{ass:V}\ref{it:fmm} below): There exist  $m\geq 1$ and $w_1, \ldots,w_m\in 
\RR^n$ such that $V(w_i) = 0$ and the Hessian $\nabla_u^2 V(w_i)$  is non-degenerate, that is, there exists 
$\lambda_{\min}(w_i) >0$ such that
 \begin{equs}
  \nabla_u^2 V(w_i) h \cdot h \geq \lambda_{\min}(w_i) |h|^2, \quad \text{for all } h \in \RR^d,\ i=1,\dots,m. 
 \end{equs} 
 \item Degenerate minima (Assumption~\ref{ass:V}\ref{it:mmp} below): There exists a function $g:\RR_{\geq 0}\to\RR_{\geq 0}$,   
 $m\geq 1$ and $r_1, \ldots,r_m\in \RR$ such that $V(u)=g(|u|^2)$, $g(r_i^2)=0$ and $g''(r_i^2)>0$. 
\end{enumerate}
In the case of degenerate minima, the set of minimum points of $V$ is given by the union of spheres of radii $r_i$ 
centred at $0$. A model example is the sombrero potential $V(u) = \frac{1}{4} (|u|^2-1)^2$. For a 
motivation from quantum field theories with $O(N)$-symmetries we refer the reader to Section~\ref{s:qft} below.

In the simpler case of non-degenerate minima, the following quantitative estimate on the top Lyapunov exponent is shown in this work. 
\begin{theorem*}[see Theorem~\ref{thm:main_1} below] Under  Assumptions~\ref{ass:V_general}  and \ref{ass:V}\ref{it:fmm} for every $n\geq 1$, $\kappa_0>0$ and $\eta>0$ there exists $\eps_0\in(0,1]$ such that for every
$\eps\leq \eps_0$ and $\kappa\geq \kappa_0$ we have that
\begin{equs}
 \lambda_{\mathrm{top}}^{(\eps)} 
 \leq - \sum_{i=1}^{m} \lambda_{\min}(w_{i}) \, \mathrm{p}^{(\kappa)}(w_{i}) + \eta, \label{eq:top_lyap_bound}
\end{equs}
where $\mathrm{p}^{(\kappa)}(w_{i}) = \frac{\theta^{(\kappa)}(w_{i})}{\sum_{i=1}^{m} \theta^{(\kappa)}(w_{i})}$
for $\theta^{(\kappa)}(w)$ defined in \eqref{eq:theta_1}. 
\end{theorem*}

The factor $\mathrm{p}^{(\kappa)}(w_{i})$ appearing in \eqref{eq:top_lyap_bound} captures the asymptotic mass of the invariant 
measure of \eqref{eq:rds_intro} for $\eps\searrow 0$ in a sufficiently small neighbourhood of the minimiser $w_i$. We emphasise 
that the negativity of the top Lyapunov exponent follows due to the assumption of non-degeneracy on the Hessian at the 
minima $w_i$, which appears often in the literature of systems of S(P)DEs, see for example \cite{MS88, MOS89, MSS94, Te08, BG13}. 

The main focus of the present work is the case of degenerate minima, that is, Assumption~\ref{ass:V}\ref{it:mmp}. In this case, the 
set of minima $\mathcal{M}$ becomes a finite union of manifolds and the Hessian $\nabla_u^2 V(w)$ is degenerate in the tangential direction for 
each $w\in \M$. As a consequence, the potential is not locally strictly convex at the minima anymore, and, therefore, the previous 
strategy fails. In fact, the degeneracy of the Hessian changes the dynamical picture. We are now confronted with a fast-slow system; there is a fast relaxation of deviations from the set of minima, and a slow component capturing the random 
motion on the manifold of minima. In contrast to the case of finitely many minima, synchronization is not caused by rare transitions between potential wells anymore, but by the positive curvature of the manifold of minima. 

Due to the degeneracy of the Hessian, one expects that $\lambda_{\mathrm{top}}^{(\eps)} \to 0$ for $\varepsilon\searrow 0$ and, thus, a higher order expansion in $\varepsilon$ is needed. Precisely, 
we rescale time $t\mapsto t/\eps$, which corresponds to rescaling the
top Lyapunov exponent by $\lambda_{\mathrm{top}}^{(\eps)}/\eps$, see Section~\ref{sec:ass_framework} below. The main 
result in this case reads as follows.

\begin{theorem*}[see Theorem~\ref{thm:main_2} below] Under Assumptions~\ref{ass:V_general} and \ref{ass:V}\ref{it:mmp},
there exists $\kappa_0>0$ depending only 
on $g''(r_i^2)$ and $r_i^2$ for $i=1,\ldots,m$, such that for every $n\geq 2$ and $\eta>0$ 
there exists $\eps_0\in(0,1]$ such that for every $\eps\leq \eps_0$ and $\kappa \geq \kappa_0$ we have that  
\begin{equs}
 \frac{1}{\eps} \lambda_{\mathrm{top}}^{(\eps)}
 & < -( n-\frac{5}{2}) \sum_{i=1}^m \frac{1}{ r_i^2 }\mathrm{p}^{(\kappa)}(r_{i}) +\eta, \label{eq:top_lyap_bound_rescaled}
\end{equs}
where $\mathrm{p}^{(\kappa)}(w_{i}) = \frac{|\Ss^{d-1}_{r_i}| \theta^{(\kappa)}(r_i e_1)}{ \sum_{i=1}^m  |\Ss^{d-1}_{r_i}|}$
for $\theta^{(\kappa)}(w)$ defined in \eqref{eq:theta_2}.
\end{theorem*}

The factor $ \mathrm{p}^{(\kappa)}(r_{i})$ in \eqref{eq:top_lyap_bound_rescaled} captures again the asymptotic mass of the 
invariant 
measure of \eqref{eq:rds_intro} for $\eps\searrow 0$ in a sufficiently small neighbourhood of $\Ss^{n-1}_{r_i}$. 

By rescaling, the assumption $\kappa \ge \kappa_0$ can be interpreted as a restriction to sufficiently small tori of length $\kappa_0^{-1/2}$. 
Indeed, for $x\in [0,1/\sqrt{\kappa})$ and 
$\tilde u(t,x) := u(t,\sqrt{\kappa}x)$, we have 
\begin{equs}
 \partial_t \tilde u (t,x) & = \kappa \Delta u(t,\sqrt{\kappa}x) + \nabla_u V(\tilde u(t,x)) + \sqrt{2\eps} \, \xi(t,\sqrt{\kappa} x)
 \\
 & = \Delta \tilde u(t,x) + \nabla_u V(\tilde u(t,x)) + \sqrt{\frac{2\eps}{\kappa}} \, \tilde \xi(t,x), \label{eq:rescaling_kappa}
\end{equs}
where $\tilde\xi(t,x) := \sqrt{k} \xi(t,\sqrt{\kappa}x)$ is $1/\sqrt{\kappa}$-periodic space-time white noise. Therefore, a careful control of the size of $\kappa_0$, and its uniformity in $\varepsilon$ is crucial. In particular, asymptotic estimates for $\kappa \nearrow\infty$ would not be satisfactory since this would correspond to infinitely small domains. We emphasise that the quantitative bounds derived in this work allow to give an explicit bound on $\kappa_0$, depending only on $g$ and on a Sobolev embedding constant, see Theorem~\ref{thm:main_2}. 

An interesting observation is that the estimate \eqref{eq:top_lyap_bound_rescaled} depends on $\Ss^{n-1}_{r_i}$ but not on the explicit form of $g$ for $\kappa\geq\kappa_0$. Intuitively, this is due to the fact, that in the case of degenerate minima, asymptotically the negativity of the top Lyapunov exponent corresponds to the negativity of the top Lyapunov exponent of the slow dynamics on the set of minima $\mathcal{M}$. The latter relies only on the convexity properties of $\mathcal{M}$ and not on further properties of $g$. 

Let us point out that in both Theorem~\ref{thm:main_1} and Theorem~\ref{thm:main_2} the derivation of the bounds is done 
carefully in order to trace the dependency of $\varepsilon_0$ on $\kappa$. In particular, we prove that $\varepsilon_0$ can be 
chosen uniformly for $\kappa_0$ fixed. 

Under additional assumptions on the potential $V$ (see Assumption~\ref{ass:V_quartic}), in Section~\ref{sec:synchr} we prove synchronisation based on the general theory of \cite{FGS14}, including the model example of the sombrero potential 
$V(u) = \frac{1}{4} (|u|^2-1)^2$.  

\begin{theorem}[Synchronisation by noise]\label{thm:synchr}\hfill 
\begin{enumerate}[i.]
 \item Under Assumptions~\ref{ass:V_general}, \ref{ass:V}\ref{it:fmm} and \ref{ass:V_quartic}, $n\geq 1$ and $\kappa_0>0$, there exists $\eps_0\in(0,1]$ such that for all 
 $\eps\leq \eps_0$ and $\kappa\geq \kappa_0$ we have  synchronisation by noise for \eqref{eq:rds_intro}.
 \item Under Assumptions~\ref{ass:V_general}, \ref{ass:V}\ref{it:mmp} and \ref{ass:V_quartic}, $n\geq 3$ there exists
 $\kappa_0>0$ and $\eps_0\in(0,1]$ such that for every  
 $\eps\leq \eps_0$ and $\kappa\geq \kappa_0$ we have synchronisation by noise for \eqref{eq:rds_intro}. 
\end{enumerate}
\end{theorem}

\subsection{Motivation and applications}\label{s:applications}

\subsubsection{Sampling and Richardson--Romberg extrapolation} \label{s:RR_extrapolation}

The problem of synchronisation by noise for systems like \eqref{eq:rds_intro} arises in the Monte-Carlo sampling of Gibbs' measures of the form
\begin{equs}
 \nu(\dd u) = \frac{1}{\Z} \exp\left\{-\frac{1}{\eps} \left( \kappa \int \dd x \, |\nabla u(x)|^2 + \int \dd x \, V(u) \right) \right\} \, \dd u, \label{eq:gibbs_measure}
\end{equs}
where $\Z$ is a normalisation constant.
By ergodicity, samples from these measures can be obtained by sampling the empirical measure
$\frac{1}{T}\int_0^T \dd t \, \delta_{u(t)}$ of the corresponding over-damped Langevin dynamics
\begin{equs}\label{eq:SPDE-applications}
 \dd u(t) =\big(\kappa \Delta u(t) -\nabla V(u(t))\big) \dd t + \sqrt{2\varepsilon} \, \dd W(t).
\end{equs}
A numerical implementation of this approach relies on discretising the Langevin dynamics, for example, by means of an explicit Euler discretisation $u_n$ of step size $t_{n}$, with $t_n\searrow 0$ and 
$T_N:=\sum_{n=1}^N t_n \nearrow +\infty$. This leads to the time-discrete empirical measures
\begin{equs}
 \nu_{N}(\omega)= \frac{1}{T_N} \sum_{n=1}^N t_n \delta_{u_{n}(\omega)}.
\end{equs}
Then, informally, $\nu_{N}(\omega) \to \nu$ weakly a.s., and, assuming that $t_n\searrow 0$  fast enough, the central limit type estimate
\begin{equs}
 \sqrt{T_N} (\Exp_{\nu_{N}(\omega)} F - \Exp_\nu F) \to \mathcal{N}(0, \sigma_F^2),
\end{equs}
is expected to be satisfied, where the variance $\sigma_F^2$ depends on $F$ and $\nu$, see \cite{LPP15} for the SDE case. This gives a quantitative bound on the rate of convergence. The assumption $t_n\searrow 0$ fast enough is crucial in order to control the error caused by the Euler discretisation. However, it limits the rate of convergence, since it implies a slower divergence of the total time $T_N:=\sum_{n=1}^N t_n \nearrow +\infty$, and thereby of the convergence $T_N^{-1/2}$.  
As a result, the order of convergence is limited. 

A remedy of this restriction is the use of higher order discretisation schemes, replacing the plain explicit Euler discretisation, which can, for example, be obtained by Richardson--Romberg extrapolation introduced in \cite{TT90}. This corresponds to running a second realisation of the Euler scheme, but with half step size $\frac{t_n}{2}$ and to recombine the two simulations in order to remove the afore-mentioned bias. 

While this method allows to increase the discretisation step size $t_n$, and, therefore, the order of convergence, it may lead to an increase of the variance of the estimator. Since this increase may compensate the gain in the order of convergence, it is important to control its size. The asymptotic variance $\tilde \sigma_F^2$ of the extrapolated scheme can be expressed in terms of $F$ and the invariant measure $\nu^{(2)}$ of the two-point motion of \eqref{eq:SPDE-applications}. Here, the two point-motion corresponds to the continuous time limit of the duplicated Euler scheme used in the Richardson--Romberg extrapolation. Following the arguments introduced in \cite[Section 4]{LPP15}, the resulting variance can be seen to be minimal if $\nu^{(2)}$ is concentrated on the diagonal, that is, if weak synchronisation occurs. 

The above method is explained in detail in \cite{LPP15} for finite-dimensional SDEs. On an informal level, the same arguments apply to the case of SPDEs, but a rigorous justification would require significant work. 

\subsubsection{Quantum field theory} \label{s:qft}

The fundamental objects of constructive quantum field theory are probability measures of the form \eqref{eq:gibbs_measure} and their correlation functions (see for example \cite{GJ87}, \cite{Ja00}, \cite{Ja08}). Many of these models, like Higgs field models \cite{AHHK89}, \cite{PS21}, \cite[Example 1.1]{CCHS20} and $O(N)$ models
\cite{CJP75}, \cite{SSZZ22} correspond to vector-valued potentials $V$. The invariance of the $O(N)$ models under the action of the $O(N)$ group implies rotational invariance of the potential $V$, which motivates the Assumption~\ref{ass:V}\ref{it:mmp} of the present work. 

Computing correlation functions in constructive quantum field theory thus requires sampling from measures of the form \eqref{eq:gibbs_measure}, see \cite{CJR83}, \cite{CJR79}. 
Along the lines of Section~\ref{s:RR_extrapolation} such samples can be obtained by stochastic quantisation  
of Euclidean quantum field theories introduced by Parisi and Wu in \cite{PW81}. In fact, this was one of the motivations to introduce stochastic quantisation in  \cite{PW81}. This precisely corresponds to considering the reversible Langevin dynamics \eqref{eq:SPDE-applications}. The numerical analysis of the resulting Monte-Carlo sampling again leads to the need to use  Richardson--Romberg extrapolation, and, thereby to the analysis of the two-point motion 
and its long-time behaviour in terms of synchronisation by noise. 

In this setting of stochastic quantisation of quantum field theories, two essential difficulties arise. First, in general, the resulting dynamics are vector-valued, and the sets of minima degenerate. Second, renormalisation 
is required in spatial dimensions larger than one. In the current work we concentrate on the dynamical aspects caused by the first of these challenges. As indicated above, the vector-valued case is dynamically fundamentally different from the scalar-valued case, since synchronisation is not caused by large deviations transitions between (finitely many) locally stable states anymore, but by concentration on a slow, synchronising motion on the manifold of minima. We therefore restrict our attention to one spatial dimension where no renormalisation is required, but consider the case of vector-valued potentials. 

In the authors' previous work \cite{GT20}, following the same motivation, the case of scalar-valued  Euclidean quantum filed theories in higher spatial dimension has been considered, and uniform synchronisation by noise in space dimensions $2$ and $3$ has been shown. The proof of \cite{GT20} fundamentally relies on comparison principles, and, cannot be used in the vector-valued case.

\subsubsection{Stochastic (multi)-phase field models}

Phase field models describe the evolution of microstructures where different states (or phases) of the material co-exist. The states of the material are described by order parameter fields which evolve
according to (stochastic) dynamics, like the (stochastic) Ginzburg--Landau model, see for example \cite[Chapter III]{GSS83}.
For two-phase field models, the evolution can be described by one order parameter leading to scalar-valued equations, whereas in multi-phase field models one needs to introduce a set of order parameter fields leading to
 vector-valued equations like \eqref{eq:rds_intro}.  A concrete example arising in grain growth was introduced in \cite{CY94}, where the 
order parameter fields are non-conserved and their evolution is given in terms of a vector-valued stochastic Allen--Cahn equation, that is, \eqref{eq:rds_intro} with sombrero potential $V(u) = \frac{1}{4} (|u|^2-1)^2$.
In the same work, the authors employ Monte Carlo simulations of the Langevin dynamics in order to 
simulate grain growth kinetics. This in particular fits the framework discussed in Section~\ref{s:RR_extrapolation}. 
For a modern review on phase-field models we refer the reader to \cite{Chen02}.

\subsubsection{Scaling limits of continuous spin models close to criticality}

Stochastic dynamics like \eqref{eq:rds_intro} describe the critical behaviour of spin models from statistical mechanics with 
long range interactions evolving according to the spin-flip Glauber dynamics, see \cite[Section~6.2]{GLP99}. These models are given in 
terms of a Gibbs measure tuned by a non-negative parameter $\beta$, the so-called inverse temperature, see \cite[Section~6.1]{GLP99}. 
One example is the Ising--Kac model where particles interact in a large neighbourhood of size $1/\gamma$, and spin configurations take values $\pm1$. It is known 
that for $\beta$ close to the critical value of the mean field model, the average magnetisation field of the spin-flip Glauber dynamics
converges to the solution of the scalar-valued stochastic Allen--Cahn equation. See \cite{BPRS93, FR95} for one space dimension, and
\cite{MW17} for two space dimensions in the context of singular SPDEs. 

A second example is the $n$-vector model with long range interactions, which is a generalisation of the Ising--Kac model,
where spin configurations take values on the $n-1$-dimensional sphere $\Ss^{n-1}$.
In this case, close to critical temperature  the average magnetisation field converges to the solution of the vector-valued stochastic Allen--Cahn equation, that is, \eqref{eq:rds_intro} with sombrero potential  $V(u) = \frac{1}{4} (|u|^2-1)^2$.
A variant of this result was verified for the two dimensional $n$-vector model in \cite[Corollary 2.21]{Ib17} in the context of singular SPDEs, generalising the result of \cite{MW17}. 
Notably, in \cite{Ib17} the author further introduced a family of statistical models for which the average magnetisation field converges to a 
non-linear SPDE with arbitrary polynomial non-linearity of odd degree and negative leading coefficient.

\subsubsection{Stochastic complex Ginzburg--Landau equation / Swift--Hohenberg equation}

The stochastic Swift-Hohenberg equation \begin{equs}
  \partial_t v = -(1+\Delta)^2 v + \theta^2  v - v^3 +\theta^{\frac{3}{2}} \sqrt{\frac{8\eps}{\kappa}} \xi, \quad \text{on} \ \RR_{>0} \times [0, 2L/\theta), 
\end{equs}
with $\theta$ a non-negative parameter, is a toy model for the description of the onset of instabilities in Rayleigh-B\'enard convection
, see for example \cite{HS77}.
As it was shown in \cite[Theorem 1.1]{BHP05}, for small values of $\theta$, the solution $v$ can be approximated by an amplitude wave
\begin{equs}
 A(t,x) = 2\theta \mathrm{Re}\left(u(\theta^2 t,\theta x) \ee^{ix}\right),
\end{equs}
where $u$ is a solution to the stochastic complex-valued Ginzburg--Landau equation
\begin{equs}
 \partial_t u = 4 \Delta u - 3 |u|^2 u + u + \sqrt{\frac{8\eps}{\kappa}} \eta \quad \text{on} \ \RR_{>0} \times [0,2L)
\end{equs}
and $\eta$ denotes complex space-time white noise. Likewise \eqref{eq:rescaling_kappa}, choosing $L=1/\sqrt{\kappa}$
this corresponds to \eqref{eq:rds_intro} with $n=2$ after rescaling.  

\subsection{Literature}

\paragraph{Literature on synchronisation by noise} 
 
The theory of order-preserving random dynamical systems (RDS) has been initiated  in \cite{AC98,Ch02,CS04}. These methods were later extended in \cite{Ge13i,FGS17} with applications to stochastic porous-media equations, and synchronisation by noise. Recently, in \cite{BS20}, coupling techniques for order-preserving Markov semigroups were developed, which allowed to prove synchronisation by noise for stochastic reaction-diffusion equations with highly degenerate noise. The framework of order-preserving Markov semigroups was 
extended in \cite{GT20} to negative order Besov spaces, with applications to synchronisation by noise for singular SPDEs, in particular including the $\Phi^4$-model in spatial dimensions $2$ 
and $3$. Synchronisation by noise for the KPZ equation based on monotonicity arguments has been shown in \cite{Ro19}.

In the case of systems of S(P)DEs, order-preservation is only available under restrictive assumptions on the coefficients, see \cite[Section 2]{Ch02}. All available results for synchronisation for systems of S(P)DE are restricted to finite dimensions. For the case of SDEs on compact manifolds we refer the reader to \cite{Ba91}. Stabilisation by multiplicative noise for linear SDE has been shown in \cite{ACW83}. The case of SDEs with potentials with finitely many, strictly convex minima and small noise, has been studied in \cite{Te08}. A general framework for synchronisation by noise was developed in \cite{FGS14,CGS16}, in particular treating gradient systems in finite dimension in large generality. For example, synchronisation by noise for SDEs with sombrero potential was shown for the first time in \cite{FGS14}. A second general approach to synchronisation by noise was developed in \cite{Ne18}. Synchronisation forward in time was analysed in \cite{Newman20}. The case of SDEs with degenerate noise has been considered in \cite{Vo18}, and precise asymptotics on the approaching time towards the attractor for small noise can be found in \cite{Vo20}. The relation of the negativity of the top Lyapunov exponent to synchronisation has been analysed in \cite{SV18}. 

A proof of synchronisation for scalar-valued SPDE that does not rely on monotonicity arguments, and, thus, allows to consider higher order differential operators, can be found in
\cite{BBY16}. However, to the best of our understanding, these methods cannot be applied to systems of SPDEs.  

Regularisation by noise, also implying synchronisation for stochastic scalar conservation laws has been considered in \cite{GG19}. Motivated from turbulent convection, in the recent works \cite{BBP19i, BBP19ii}, synchronisation for passive scalars transported by stochastic fluids, that is, for a class of random PDE, has been shown. In particular, path-wise exponential decay in negative Sobolev spaces is obtained, which can be interpreted as a quantitative synchronisation by noise result. For the Kraichnan model of turbulent convection, a conservative SPDE, a similar result  has been obtained in \cite{GY21}. More precisely, in this work the authors prove path-wise exponential decay for a class of SPDEs with non-asymptotically stable deterministic counterpart.  

\paragraph{Literature on dynamics of stochastic reaction-diffusion equations} 

There is a vast body of works on the dynamics of stochastic reaction-diffusion equations, and giving a complete account would go much beyond the scope of this work. Therefore, we will restrict to those most relevant to the present work, and the references therein.

Ergodicity for stochastic reaction-diffusion equations with non-degenerate noise has been studied extensively in the classical literature for SPDEs, see for example \cite{DPZ96}. For degenerate noise, we refer the reader to the seminal works \cite{HM08, HMS11} and the references therein. 

Vanishing noise, path-space large deviations for the $(1+1)$-dimensional Allen-Cahn equation have been obtained in \cite{FJL82}. Metastability for the stochastic Allen-Cahn equation with small noise has been analysed in \cite{BG13} and \cite{BDGW17,TW19}, where an Eyring--Kramers law for the transition times between the stable minima
was proved. These results are based on the ``exponential loss of memory" property, which was first introduced
in \cite{MS88, MOS89, MSS94}, and which was extended to the singular case in \cite{TW19}. This property implies that solutions starting in the same potential-well contract exponentially fast with overwhelming probability.
In addition, in \cite[Proposition 3.4]{MOS89} it has been shown that solutions starting in different potential wells contract exponentially fast after a sufficiently large waiting time, proportional to the time needed to jump in-between the potential wells. This result can be interpreted as synchronisation by noise. The same result is expected to hold in the singular case, as it was explained in \cite[Remark 3.4]{TW19}. 

Recently, in \cite{BEN21} the authors studied bifurcations for the scalar stochastic Allen-Cahn equation with additive noise. In particular, they showed that the bifurcation stays visible in the stochastic case if one considers finite time Lyapunov exponents, in contrast to the usual ``infinite time" Lyapunov exponents. The problem of estimating the ``infinite time" Lyapunov exponents answered in the present work was posed as an open question  in \cite{BEN21}.

\subsection{Notation}

In the sequel $\xi = (\xi_1,\xi_2,\ldots,\xi_n)$ denotes space-time white noise on $\RR\times \TT$ defined on a probability space 
$(\Omega, \mathcal{F}, \Prob)$. We identify $\Omega$ with the space of Schwartz distributions $\mathscr{S}'(\RR\times\TT;\RR^n)$, namely, the topological dual of the space of Schwartz functions on $\RR\times \TT$
with values in $\RR^n$ denoted by $\mathscr{S}(\RR\times\TT;\RR^n)$. Therefore, for $\omega\in \Omega$ we have $
\big(\xi(\omega),\psi\big) = \big(\omega, \psi \big)$ for every $\psi\in \mathscr{S}(\RR\times\TT;\RR^n)$, 
where $\big(\cdot,\cdot\big)$ denotes the dual pairing of a Schwartz distribution and a Schwartz function. In this context the $\sigma$-algebra $\mathcal{F}$ is the Borel $\sigma$-algebra on the space of Schwartz distributions. 
Given a compact subset $A\subset \RR\times \TT$ and a Banach space $\mathcal{B}$ we write 
$\C^0([0,T];\mathcal{B})$ for the space of continuous functions on $A$ taking values in $\mathcal{B}$. 
The space $L^p(A;\mathcal{B})$ for $p\in[1,\infty]$ is defined analogously. We also write $\C^k(\mathcal{B};\RR^m)$ for 
the space of $k$-times continuously differentiable functions on $\mathcal{B}$ taking values in $\RR^m$. 

\subsection{Assumptions and general framework}\label{sec:ass_framework}

Our general assumption on $V$ reads as follows,  

\begin{assumption}\label{ass:V_general} $V\in \C^4(\RR^n;\RR)$  
and there exists $p\geq 1$ such that
\begin{enumerate}[i.]
 \item For every multi-index $\alpha$ with $|\alpha|=4$ there exists $C<\infty$ such that for every $u\in \RR^n$  
  \begin{equs}
   |\partial^\alpha V(u)|\leq C (1+|u|^{2p-4}). \label{eq:pol_growth}
  \end{equs}
 \item There exists $C<\infty$ such that such that for every $u\in \RR^n$   
  \begin{equs}
   -\nabla V(u) \cdot u \leq - |u|^{2p} + C. \label{eq:coerc_grad_V}
  \end{equs}
  \item For every $R>0$ there exists $M\geq 1$ such that for every $|u|\geq R$ and $h\in \RR^n$  
   \begin{equs}
    \nabla^2 V(u) h \cdot h \geq M |h|^2. \label{eq:hess_V_bnd}
   \end{equs}
\end{enumerate}
\end{assumption}

\begin{remark} Note that \eqref{eq:coerc_grad_V} implies that there exists $C>0$ such that for every $u\in \RR^n$  
\begin{equs}
 & -V(u) +\frac{1}{2p} |u|^{2p} \leq C. \label{eq:coerc_V} 
\end{equs}
\end{remark}

As we already discussed, in the sequel we consider the following two different types of potentials,

\begin{assumption} \label{ass:V}\hfill
\begin{enumerate}[a.]
 \item \label{it:fmm} \textbf{Non-degenerate minima}: There exist $w_i\in \RR^n$, $i=1,\ldots, m$ such that $V(w_i) =0$,
 $\nabla V(w_i) =0$, and $\nabla^2 V(w_i)$ is positive definite, that is, there exist $\lambda_{\min}(w_i)>0$ such that
 \begin{equs}
  \nabla^2V(w_i) h \cdot h \geq \lambda_{\min}(w_i) |h|^2, \quad \text{for all } h\in \RR^n. 
  \end{equs}
 \item \label{it:mmp} \textbf{Degenerate minima} (rotationally-invariant $V$): $V$ is rotationally-invariant and,
 in particular, $V(u) = g(|u|^2)$, where $g: \RR \to \RR_{\geq0}$ is in $\C^3(\RR;\RR)$ and there exist $r_i$, $i=1,\ldots,m$ such that $g(r_i^2) =0$, $g'(r_i^2) = 0$ and $g''(r_i^2) >0$. 
\end{enumerate}
\end{assumption}

\begin{remark} Note that under Assumption~\ref{ass:V}\ref{it:mmp} the potential $V$ assumes its global minima at $ \cup_{i=1}^m \Ss_{r_i}^{n-1}$. In particular, if $\T_w$ denotes the tangent space and $\T_w^{\perp}$ the normal space at $w\in \cup_{i=1}^m \Ss_{r_i}^{n-1}$,  Assumption~\ref{ass:V}\ref{it:mmp} implies that $V(w) =0$, $\nabla V(w) = 0$ and $\nabla^2 V(w)$ is positive semi-definite, satisfying the non-degeneracy condition in the normal direction
 \begin{equs}
  \nabla^2V(w) h \cdot h = 4g''(r_i^2) |h|^2 \quad \text{for all } h\in \T_w^{\perp}.
 \end{equs}
\end{remark}

We let 
\begin{equs}
 \mathbf{V}(u) := \int_\TT \dd x \, V(u(x)), \quad b(u):=-\nabla_u V(u), 
\end{equs}
and re-write \eqref{eq:rds_intro} in the following form,  
\begin{equs} \label{eq:reaction-diffusion_system}
  \begin{cases}
    & (\partial_t- \kappa\Delta) u = b(u) + \sqrt{2\eps} \xi \quad \text{on} \ \RR_{>0}\times \TT,
    \\
    & u|_{t=0} = f.
  \end{cases}
\end{equs}

\paragraph{White noise RDS} For $s\leq t$, $x\in \TT$ and $\omega\in \Omega$ we consider the stochastic convolution 
\begin{equs}
w_s(t,x;\omega) := \xi\big(\mathbf{1}_{(s,t]} H_{t-\cdot}^{(\kappa)}(x-\cdot)\big)(\omega),
\end{equs}
where $H_{r}^{(\kappa)}$, $r>0$, stands for the periodic heat kernel associated to $-\kappa \Delta$. By Kolmogorov's 
continuity criterion, see for example \cite[Theorem~3.4]{DPZ92}, we know that $w_{s}\in \C^0\left([s,T];\C^\alpha(\TT;\RR^n)\right)$
for every $\alpha\in[0,\frac{1}{2})$, $T>s$. In the sequel we fix a continuous version of $\{w_s\}_{s\in \RR}$, namely, 
we assume that the mapping $t \mapsto w_s(t,\cdot;\omega)
\in\C^\alpha(\TT;\RR^n)$ is continuous for every $s\in \RR$, $\omega\in \Omega$ and $\alpha\in[0,\frac{1}{2})$. Note that $w_s$ solves the stochastic heat equation
\begin{equs} \label{eq:stoch_conv}
 \begin{cases}
  & (\partial_t - \kappa\Delta) w_s = \sqrt{2\eps} \xi \quad \text{on} \ [s,\infty)\times \TT, \\
  & w_s|_{t=s} = 0, 
 \end{cases}
\end{equs}
and satisfies for $t\geq r\geq s$ the relation
\begin{equs}
w_s(t) = H_{t-r}^{(\kappa)}*w_s(r) + w_r(t). 
\end{equs}

We define $u_s:=w_s +v_s$, where 
\begin{equs}\label{eq:remainder}
 \begin{cases}
  & (\partial_t - \kappa \Delta) v_s = b(w_s+v_s) \quad \text{on} \ [s,\infty)\times \TT, \\
  & v_s|_{t=s} = f. 
 \end{cases}
\end{equs}
We sometimes write $u_s(t;f)$ and $v_s(t;f)$, $u_s(t;f,w_s)$ and $v_s(t;f,w_s)$ or $u_s(t;f,\omega)$ and $v_s(t;f,\omega)$ 
to specify the dependence on the initial condition $f$ and the stochastic convolution $w_s$ or the parameter $\omega$. 
In the case $s=0$ we omit the subscript $s$ and simply write $u$, $w$ and $v$. 
Note that $u=w+v$ solves \eqref{eq:reaction-diffusion_system}, although the equality should be understood as the definition
of the solution.

As we argue in Appendix~\ref{app:RDS}, $(t,\omega,f)\mapsto u(t;f,w(\cdot;\omega))$ gives rise to a white noise RDS on
$\C^0(\TT;\RR^n)$ which is extended to $L^2(\TT;\RR^n)$. 

\paragraph{Invariant measure} In Appendix~\ref{app:RDS} we argue that the white noise (RDS) gives rise to a strong Markov process. The corresponding Markov process is ergodic with invariant measure $\nu_\eps^{(\kappa)}$ given by
\begin{equs}
 \nu_\eps^{(\kappa)}(\dd u) 
 := \frac{1}{\Z_\eps^{(\kappa)}}  \exp\left\{-\frac{1}{\eps} \mathbf{V}(u) 
 + \frac{1}{2\eps} \|u\|_2^2  \right\} \mu_\eps^{(\kappa)}(\dd u), \label{eq:gibbs}
\end{equs}
where $\mu_\eps^{(\kappa)}$ is the Gaussian measure $\mathcal{N}\left(0, \eps (-\kappa\Delta + 1)^{-1}\right)$ on the space of 
Schwartz distributions. We also write $\mu^{(\kappa)}$ for the Gaussian measure 
$\mathcal{N}\left(0, (-\kappa\Delta + 1)^{-1}\right)$. Note that $\supp\mu_\eps^{(\kappa)}\subset \C^0(\TT;\RR^n)$. 

\begin{remark} Here we use a massive Gaussian measure as the reference measure for $\nu_\eps^{(\kappa)}$ 
since we work in the periodic setting. Note that $\nu_\eps^{(\kappa)}(\dd u)$ is well-defined, since the
partition function $\Z_\eps^{(\kappa)}$ is finite. Indeed, by \eqref{eq:coerc_V} implies that 
\begin{equs}
 \exp\left\{-\frac{1}{\eps} \mathbf{V}(u) + \frac{1}{2\eps} \|u\|_2^2  \right\}  
 \leq \exp\left\{ \frac{C}{\eps} \right\},
\end{equs}
yielding the bound $\Z_\eps^{(\kappa)} \leq \exp\left\{C/\eps \right\}$. 
\end{remark} 

\paragraph{Rescaled system} In the non-degenerate case (Assumption~\ref{ass:V}\ref{it:fmm}) we study \eqref{eq:reaction-diffusion_system} directly.
In the case of degenerate minima (Assumption~\ref{ass:V}\ref{it:mmp}) we work with the rescaled system 
\begin{equs} \label{eq:reaction-diffusion_system_rescaled}
  \begin{cases}
    & (\partial_t- \frac{\kappa}{\eps}\Delta) \tilde u = \frac{1}{\eps}b(\tilde u) + \sqrt{2} \xi,
    \\
    & \tilde u|_{t=0} = f.  
  \end{cases}
\end{equs}
The connection between \eqref{eq:reaction-diffusion_system} and \eqref{eq:reaction-diffusion_system_rescaled} is 
given by the time-rescaling $t \mapsto t/\eps$. More precisely, 
it is easy to see that $(t,x)\mapsto u(\eps^{-1}t,x;f)$ solves \eqref{eq:reaction-diffusion_system_rescaled} with
$\xi$ replaced by $\tilde \xi(t,x):=\sqrt{\eps}^{-1}\xi(\eps^{-1}t,x)$. Since $\xi$ and $\tilde \xi$ are equal in law, the same 
holds for $\tilde u(\cdot;f)$ and $u(\eps^{-1}\cdot;f)$.

\paragraph{Linearisation} In the sequel, we repeatedly work with the linearisation with respect to the initial data $f$ of $f\mapsto u_s(\cdot;f)$
given by
\begin{equs} \label{eq:linearised_system}
  \begin{cases}
    & (\partial_t - \kappa\Delta) Du_s(t;f)h = \nabla_u b(u_s(t;f)) Du_s(t;f)h, 
    \\
    &  Du_s(t;f)h|_{t=s} = h,
  \end{cases}
\end{equs}
and, analogously for \eqref{eq:reaction-diffusion_system_rescaled}, 
\begin{equs}  \label{eq:linearised_system_rescaled}
  \begin{cases}
    & (\partial_t - \frac{\kappa}{\eps}\Delta) D\tilde u_s(t;f)h = \frac{1}{\eps} \nabla_{u} b(\tilde u_s(t;f)) D\tilde u_s(t;f)h, 
    \\
    &  D\tilde u_s(t;f)h|_{t=0} = h.
  \end{cases}
\end{equs}
Since $u_s\in \C^0([s,T];\C^0(\TT;\RR^n))$ for every $T>0$ and the above equations are linear, existence and uniqueness of solutions in $\C^0\left([0,\infty);\C^0(\TT;\RR^n)\right)$ for every $h\in \C^0(\TT;\RR^n)$ and $\omega\in \Omega$ is immediate. 
As we argue in Appendix~\ref{app:RDS} solutions can be extended to $\C^0\left([0,\infty);L^2(\TT;\RR^n)\right)$ for every $h\in 
L^2(\TT;\RR^n)$.

\paragraph{Top Lyapunov exponent} For $\eps>0$, $\omega\in \Omega$ and $f\in \supp \mu_\eps^{(\kappa)}$ we write\footnote{omitting the dependence on
$\kappa$ to ease the notation} $\lambda_{\mathrm{top}}^{(\eps)}$ and $\tilde \lambda_{\mathrm{top}}^{(\eps)}$ for the top Lyapunov exponents
\begin{equs}
 {} & \lambda_{\mathrm{top}}^{(\eps)}:= \lim_{t\nearrow\infty} \frac{1}{t}\log \sup_{\|h\|_2\leq 1} \|Du(t;f,\omega)h\|_2, \label{eq:lyap}
 \\
 & \tilde \lambda_{\mathrm{top}}^{(\eps)} := \lim_{t\nearrow\infty} \frac{1}{t}\log \sup_{\|h\|_2\leq 1} \|D\tilde u(t;f,\omega)h\|_2. \label{eq:lyap_rescaled}
\end{equs}
Existence of the two limits in \eqref{eq:lyap} and \eqref{eq:lyap_rescaled} for $\nu_\eps^{(\kappa)}\times \Prob$-almost every $
(f,\omega)$ and the fact that they do not depend on $(f,\omega)$ is shown in Proposition~\ref{prop:lyap_exp_exist}. Since $\tilde u(\cdot;f) = u(\eps^{-1}\cdot;f)$ in law, it is easy to see that
\begin{equs}
 \tilde \lambda_{\mathrm{top}}^{(\eps)} = \frac{1}{\eps} \lambda_{\mathrm{top}}^{(\eps)}. \label{eq:equal_law}
\end{equs}

\subsection{Structure of the proof}

The proofs of Theorem~\ref{thm:main_1} and Theorem~\ref{thm:main_2} are based on the asymptotic expansions in the noise intensity $\varepsilon$ of Gaussian integrals, which arise after suitably estimating the top Lyapunov exponent. 
There are two main challenges: Firstly, the top Lyapunov exponent needs to be estimated by sufficiently regular quantities in order to allow for an asymptotic expansion. Secondly, this estimate needs to retain enough information in order to lead to quantitative estimates proving the negativity of the top Lyapunov exponent. These two challenges are particularly hard in the case of degenerate minima, where spatial variations need to be controlled by the viscosity coefficient $\kappa$, with careful control on its required size. 

We start this exposition of the structure of the proof by heuristically motivating a general upper bound on the top Lyapunov exponent \eqref{eq:lyap}. The details are given in Proposition~\ref{prop:lyap_exp_formulae} below. From \eqref{eq:linearised_system} we obtain that 
\begin{equs}
\frac{1}{2} \partial_t \|Du(t;f) h\|_2^2 = \lng \left(\kappa \Delta + \nabla_u b(u(t;f))\right) Du(t;f) h, Du(t;f) h \rng, \label{eq:test}
\end{equs}
which implies that
\begin{equs}
 \|Du(t;f) h\|_2 = \exp\left\{\int_0^t \dd s \, \left\lng \left(\kappa \Delta + \nabla_u b(u(s;f))\right) \frac{Du(s;f)h}{\|Du(s;f)h\|_2}, \frac{Du(s;f)h}{\|Du(s;f)h\|_2} \right\rng\right\} \|h\|_2. 
\end{equs}
Therefore, we obtain the bound 
\begin{equs}
 \log \|Du(t;f)h\|_2 
 =  \int_0^t \dd s \, \left\lng 
 \left(\kappa \Delta + \nabla_u b(u(s;f))\right) \frac{Du(s;f)h}{\|Du(s;f)h\|_2}, \frac{Du(s;f)}{\|Du(s;f) h\|_2} \right\rng.
 \label{eq:log_linear_bnd}
\end{equs}
Pretending for the moment that the coupled flow $\{(u(t;f),  Du(s;f)/\|Du(s;f) h\|_2)\}_{t\geq0}$ is ergodic with invariant measure
$\rho^{(\kappa)}_{\eps}$, we arrive at the following formula,
\begin{equs}
 \lim_{t\nearrow \infty} \frac{1}{t} \log \|Du(t;f)h\|_2 = \iint \rho^{(\kappa)}_{\eps}(\dd u, \dd v) \, 
 \left\lng 
 \left(\kappa \Delta + \nabla_u b(u)\right) v, v \right\rng, \quad 
 \text{for}\ \nu_\eps^{(\kappa)}\times \Prob\text{-a.e.}\ (f,\omega).
\end{equs}
By \cite[Corollary 2.2]{Ru82}, there exists a subspace $U_\omega$ of $L^2(\TT;\RR^n)$ such that for every $h\in U_\omega$ 
,
\begin{equs}
 \lim_{t\nearrow \infty} \frac{1}{t} \log \|Du(t;f)h\|_2 = \lambda_{\mathrm{top}}^{(\eps)}, \quad
 \text{for}\ \nu_\eps^{(\kappa)}\times \Prob\text{-a.e.}\ (f,\omega),
\end{equs}
yielding the exact identity
\begin{equs}
 \lambda_{\mathrm{top}}^{(\eps)} = \iint \rho^{(\kappa)}_{\eps}(\dd u, \dd v) \, 
 \left\lng 
 \left(\kappa \Delta + \nabla_u b(u)\right) v, v \right\rng, \quad 
 \text{for}\ \nu_\eps^{(\kappa)}\times \Prob\text{-a.e.}\ (f,\omega).
 \label{eq:lyap_sharp_est} 
\end{equs}
However, the analysis of the measure $\rho^{(\kappa)}_{\eps}$ and of its concentration behavior for $\varepsilon$ small are typically difficult, even if ergodicity of $\{(u(t;f),  Du(s;f)/\|Du(s;f) h\|_2)\}_{t\geq0}$ is known. Instead, we work with the following weaker 
estimate, which immediately follows from \eqref{eq:lyap_sharp_est}\footnote{In the rigorous proof we obtain \eqref{eq:fk_bound} directly from 
\eqref{eq:log_linear_bnd}, therefore we do not need to consider the problem of ergodicity for the coupled flow $\{(u(t;f),  Du(s;f)/\|Du(s;f) h\|_2)\}_{t\geq0}$.}, 
\begin{equs} 
 \lambda_{\mathrm{top}}^{(\eps)} \leq \int \nu_\eps^{(\kappa)}(\dd u) \, \lambda_+(u), \quad 
 \text{for}\ \nu_\eps^{(\kappa)}\times \Prob\text{-a.e.}\ (f,\omega), \label{eq:fk_bound}
\end{equs}
where 
\begin{equs}
  \lambda_+(u) := \sup_{\|h\|_2=1}
  \left \lng\big(\kappa\Delta + \nabla_u b(u)\big) h, h \right\rng, \quad \text{for}\ u\in \supp \mu_\eps^{(\kappa)}. \label{eq:lambda_+}
\end{equs}
The advantage of \eqref{eq:fk_bound} over \eqref{eq:lyap_sharp_est} is that the former only involves 
the invariant measure $\nu_\eps^{(\kappa)}$, which has an explicit Gibbs structure, see \eqref{eq:gibbs}.   

As a consequence, we aim to derive quantitative upper bounds on the right hand side of \eqref{eq:fk_bound} proving, in particular, its negativity. One of the key ideas to obtain such bounds is to analyse asymptotic expansions with respect to the noise intensity $\varepsilon$ in the form\footnote{Due to parity, only even powers of $\sqrt{\eps}$ contribute to the expansion.} 
\begin{equs}\label{eq:intro-expansion}
 \int \nu_\eps^{(\kappa)}(\dd u) \lambda_+(u) \approx a_0 + a_1 \eps+ \Oo(\eps^2)
\end{equs}
and estimate the coefficients $a_i$, $i=1,2$.

At this point, the treatment of non-degenerate and degenerate minima deviate. Let us first consider the simpler case of non-degenerate minima. In the limit $\eps\searrow 0$ the invariant measure $\nu_\eps^{(\kappa)}$ 
is expected to concentrate on the minima $w_i$ with masses  $\mathrm{p}^{(\kappa)}(w_i)\in[0,1]$. Hence, combined with \eqref{eq:intro-expansion} we obtain that
\begin{equs}
 \int \nu_\eps^{(\kappa)}(\dd u) \, \lambda_+(u) \approx \sum_{i=1}^m \lambda_+(w_i) \mathrm{p}^{(\kappa)}(w_i) + \Oo(\eps). \label{eq:ER_exp_1}
\end{equs}
Due to the strict convexity of $V$, see Assumption~\ref{ass:V}\ref{it:fmm}, we have
\begin{equs}
 \lambda_+(w_i) 
 = - \inf_{\|h\|_2 =1}  \big(\kappa \|\nabla h\|_2^2  + \left \lng \nabla_u^2 V(w_i) h, h \right\rng\big)
 \leq -\lambda_{\min}(w_i) <0,  
\end{equs}
uniformly in $\kappa\geq0$. Hence, from \eqref{eq:ER_exp_1} we obtain
\begin{equs}
 \int \nu_\eps^{(\kappa)}(\dd u) \, \lambda_+(u) \approx- \sum_{i=1}^m \lambda_{\min}(w_i)\mathrm{p}^{(\kappa)}(w_i) + \Oo(\eps),
\end{equs}
which together with \eqref{eq:fk_bound} yields a quantitative upper bound on the top Lyapunov exponent and, in particular, its negativity for $\eps$ small enough. 

In the case of degenerate minima the above heuristic fails. In this case, in \eqref{eq:intro-expansion}, the concentration of $\nu_\eps^{(\kappa)}$ on the manifold of minimum points $\M$ implies that, for some probability measure $\mathrm{p}^{(\kappa)}$,
\begin{equs}
 a_0 =  \int_\M \mathrm{p}^{(\kappa)}(\dd w) \, \lambda_+(w) =0,
 \label{eq:ER_exp_2}
\end{equs}
where we used that $\lambda_+(w) = 0$ for all $w\in \M$, due to the non-strict convexity of $V$ on $\M$. Therefore, the zero order expansion in \eqref{eq:intro-expansion} does not imply a non-trivial bound on the top Lyapunov exponent anymore. 

To resolve this issue, we go beyond zero order in the asymptotic expansion \eqref{eq:intro-expansion}. Correspondingly, we rescale the system in time to obtain \eqref{eq:reaction-diffusion_system_rescaled}, with the rescaled top Lyapunov exponent  \eqref{eq:lyap_rescaled} satisfying the following analogue of \eqref{eq:fk_bound}, 
\begin{equs} 
 \tilde\lambda_{\mathrm{top}}^{(\eps)} \leq \frac{1}{\eps} \int \nu_\eps^{(\kappa)}(\dd u) \, \lambda_+(u) \approx a_1  + \Oo(\eps).   
 \label{eq:fk_bound_rescaled}
\end{equs}
As a consequence, the aim is to prove quantitative estimates and the negativity of the now leading order term $a_1$. This is a challenging task and relies on the derivation of good estimates, exploiting the particular form of $a_1$. We will next describe some of the arising difficulties and the main ideas on how they are resolved. 

Asymptotic expansions of the form \eqref{eq:intro-expansion} rely on the regularity of the involved coefficients, that is, of $\lambda_+(u)$. Based on \eqref{eq:lambda_+} this is unclear. Therefore, we need to find a carefully chosen, sufficiently regular upper bound $\chi_+(u)$ for $\lambda_+(u)$. At the same time it is important to not loose too much information when passing from $\lambda_+(u)$ to $\chi_+(u)$, since the latter will eventually determine the quantitative bounds and, thus, the negativity of the top Lyapunov exponent. In the case of non-degenerate minima, the resulting upper bound is simple, see Proposition~\ref{prop:h-red_simple}, and can be estimated directly by adapting the asymptotic expansions for Gibbs measures in \cite[Theorem~4.6]{ER82I}, see Lemma~\ref{lem:ellis--rosen}. In the case of degenerate minima, relying on the rotation-invariance $V(u) = g(|u|^2)$, in Proposition~\ref{prop:h-red} we show that $\lambda_+(u)$ can be bounded by 
\begin{equs}
 \chi_+(u) &= - \left( 2\int \dd x \, g'(|u(x)|^2) 
 - \frac{4 C_*^2 \|g'(|u|^2)\|_2^2}{\kappa} \right) 
 +
 \mathbf{1}_{\left\{\dist_{\|\cdot\|_\infty}(\M,u)\geq \delta\right\}} \mathsf{Error}(u).
\end{equs}
The main contribution in the estimates comes from the first term on the right hand side. Indeed, the error term vanishes exponentially in $\eps$, see Lemma~\ref{lem:varadhan}. This leads to the task of justifying and estimating the following
asymptotic expansion replacing \eqref{eq:fk_bound_rescaled},
\begin{equs}
 -\frac{1}{\eps} \int \nu_\eps^{(\kappa)}(\dd u) \,  \left( 2\int \dd x \, g'(|u(x)|^2) - \frac{4 C_*^2 \|g'(|u|^2)\|_2^2}{\kappa} \right)
 \approx a_1 + \Oo(\eps)
\end{equs}
and, thus, to estimating the leading order term $a_1$. This is done in Corollary~\ref{cor:asympt} by adapting the asymptotic 
expansions for Gibbs measures in \cite[Theorem~5]{ER82II}, see Lemma~\ref{lem:ellis--rosen} and 
Lemma~\ref{lem:ellis--rosen_rot_inv}, and by carefully estimating the resulting
expression for $a_1$. In particular, we prove that this expression consists of two parts; a main part, arising from the expansion of 
the first term in the parenthesis, which is negative for $n\geq 3$, and an error which becomes small for $\kappa$ sufficiently large. 

In the above heuristics we have neglected the dependency of $\Oo(\eps)$ on $\kappa$, which for fixed 
$\kappa_0>0$ is uniform in $\kappa\geq \kappa_0$. However, a quantification of the error in the 
asymptotic expansions is not immediate by the results in \cite{ER82I, ER82II} and requires the uniform 
estimates of Section~\ref{sec:uniform_kappa}. 

\subsection{Outline}

In Section~\ref{sec:uniform_kappa} we prove quantified 
error estimates, which allow to keep track of the dependency in $\kappa$.
In Section~\ref{app:ellis--rosen} we adapt the asymptotic expansions from \cite{ER82I, ER82II}
to the current setting. In Section~\ref{sec:lyap_est} we obtain suitable bounds on the top Lyapunov exponents, which we then 
use in Section~\ref{sec:lyap_est_0} and Section~\ref{sec:lyap_est_1} to prove the two main results, Theorem~\ref{thm:main_1} 
and Theorem~\ref{thm:main_2}. Lastly, in Section~\ref{sec:synchr}, building on the framework of \cite{FGS14} we prove
Theorem~\ref{thm:synchr}.   

\section*{Acknowledgements}  

The authors were supported by the Deutsche Forschungsgemeinschaft (DFG, German Research Foundation) – SFB 1283/2 2021 – 317210226.  BG acknowledges support by the Max--Planck--Society through the Research Group ``Stochastic Analysis in the Sciences (SAiS)". PT thanks the Max--Planck--Institute for Mathematics in the Sciences for its warm hospitality. 

\section{Asymptotic analysis} 

In this section we consider $V\in \C^{k+1}(\RR^n;\RR_{>0})$, for some $k\in \mathbb{N}$, which satisfies
the coercivity estimate \eqref{eq:coerc_V} for $p=2$. We further assume that the set of minimum points $\M$
of $V$ is either a finite discrete set or it forms a union of smooth compact submanifolds of $\RR^n$ and that the
Hessian $\nabla^2 V(w)$, $w\in \M$, is non-negative. We will refer to the two cases as $\dim \M=0$
and $\dim \M\neq 0$ respectively. 

For $w\in \M$ we write $\T_w$ for the tangent space to $\M$ at $w$ and $\T_w^\perp$ for its normal space, with the 
understanding that $\T_w=\emptyset$ and $\T_w^\perp = \RR^n$ for $\dim \M=0$.  We further assume that the Hessian 
$\nabla^2 V(w)$, $w\in \M$, is positive definite on $\T_w^{\perp}$, that is, there exists $b_*>0$ such that 
\begin{equs}
 \inf_{w\in\M}\inf_{v\in \T_w^\perp, \, |v|=1} \nabla^2 V(w)v \cdot v\geq b_* >0. \label{eq:non_degeneracy}
\end{equs}

For $\kappa>0$, we define the action $E^{(\kappa)} : u\mapsto E^{(\kappa)}(u)$ by 
\begin{equs}
 E^{(\kappa)}(u) := \frac{\kappa}{2} \|\nabla u\|_2^2 + \V(u),  
\end{equs}
where we recall that $\V(u)= \int_\TT \dd x \, V(u(x))$. Note that due to \eqref{eq:coerc_V}, $\V$ is coercive, that is, 
\begin{equs}
 -\V(u) + \frac{1}{2} \|u\|_2^2 \leq C. \label{eq:coercivity}
\end{equs}

The set of minimum points of $E^{(\kappa)}$ is given by $\M$, seen as a finite discrete subset, respectively compact
submanifold, of $L^2(\TT;\RR^n)$. In the latter case, we let 
\begin{equs}
 \N_w := \T_w^\perp \oplus \left\{f\in L^2(\TT;\RR^n): \, \int_\TT\dd x \, f(x) = 0\right\}. 
\end{equs}
Note that \eqref{eq:non_degeneracy} implies that for every $w\in \M$ and $v\in \N_w$ 
\begin{equs}
 D^{(2)} E^{(\kappa)}(w)(v,v) = D^{(2)}\V(w)(v,v) = \int_\TT \dd x \, \nabla^2 V(w)v(x) \cdot v(x)  \geq b_* \|v\|_2^2 .  
\end{equs}

Last, for $w\in \M$ we let $\mu_w^{(\kappa)}(\dd v)$ be the Gaussian measure on $L^2(\TT;\RR^n)$ with covariance 
$C_w$ given by
\begin{equs}
C_w = 
 \begin{cases}
  (-\kappa\Delta \, \mathrm{Id}_{n\times n} + \nabla^{(2)}V(w))^{-1}, \quad & \dim \M=0, 
  \\
  (-\kappa\Delta \, \mathrm{Id}_{n\times n} + P_{\N_w}\nabla^{(2)}V(w) P_{\N_w})^{-1}, \quad & \dim \M\neq 0, 
 \end{cases}
\end{equs}
where $P_{\N_w}$ stands for the projection onto $\N_w$.

\subsection{Quantified error estimates} \label{sec:uniform_kappa}

The purpose of this subsection is to provide a careful analysis of the error estimates used in the proof of Lemma~\ref{lem:ellis--rosen}  
below, in order to ensure that the error term in $\eps$ is uniform in $\kappa$ as long as $\kappa\geq \kappa_0>0$. To simplify the notation in the following statements we let $\kappa_0=1$ but the proofs trivially extend to arbitrary
$\kappa_0>0$.   

The next two lemmata will be used in the proof of Lemma~\ref{lem:varadhan} below, where we derive a quantified 
version of Varadhan's lemma. In what follows we decompose the Gaussian measure $\mu^{(\kappa)}(\dd u)$ to the product
$\bar\mu(\dd \bar u)$ and $\mu_\perp^{(\kappa)}(\dd u^{\perp})$ where $\bar u = \int \dd x \, u(x)$ and $u^\perp = u - \bar u$,  
which allows to eliminate the dependence on $\kappa$ by rescaling $u^\perp$ after a simple change of measure and to obtain an
estimate which is uniform in $\kappa\geq 1$.

\begin{lemma}\label{lem:uniform_M} For every $M\geq 1$ and $\eps\in(0,1]$ we have that
\begin{equs}
 \sup_{\kappa\geq1}\int_{\left\{\|\sqrt{\eps}u\|_\infty\leq M\right\}} \mu^{(\kappa)}(\dd u) \, \exp\left\{\frac{1}{2} \|u\|_2^2 \right\}
 \lesssim \left(\frac{M}{\sqrt{\eps}}\right)^{n}.
\end{equs}
\end{lemma}

\begin{proof} We write $\mu^{(\kappa)}(\dd u)$ as a product measure of $\bar\mu(\dd \bar u)$ and $\mu_\perp^{(\kappa)}(\dd u^{\perp})$. This gives
\begin{equs}
 {} & \int_{\left\{\|\sqrt{\eps}u\|_\infty\leq M\right\}} \mu^{(\kappa)}(\dd u) \, \exp\left\{\frac{1}{2} \|u\|_2^2 \right\} 
 \\
 & \quad \leq \int_{\left\{|\sqrt{\eps}\bar u|_\infty\leq M\right\}} \bar\mu(\dd \bar u) \, \exp\left\{\frac{1}{2} |\bar u|_2^2 \right\} 
 \int \mu_\perp^{(\kappa)}(\dd u^{\perp}) \, \exp\left\{\frac{1}{2} \|u^\perp\|_2^2 \right\}. \label{eq:product_decomp}
\end{equs}
The first integral on the right hand side of \eqref{eq:product_decomp} can be estimated as
\begin{equs}
 \int_{\left\{|\sqrt{\eps}\bar u|_\infty\leq M\right\}} \bar\mu(\dd \bar u) \, \exp\left\{\frac{1}{2} |\bar u|_2^2 \right\} 
 = \frac{1}{(2\pi)^{\frac{n}{2}}} \int_{\left\{|\sqrt{\eps}\bar u|_\infty\leq M\right\}} \dd \bar u 
 \lesssim \left(\frac{M}{\sqrt{\eps}}\right)^{n}. \label{eq:average_M_bound}
\end{equs}
Decomposing $\mu_\perp^{(\kappa)}(\dd u^{\perp})$ as the product $\prod_{m\neq 0}\mu_m^{(\kappa)}(\dd \hat u(m))$ 
on the Fourier modes $\{\hat u(m)\}_{m\neq 0}$ of $u^\perp$\footnote{namely for every $m\neq 0$, $\mu_m^{(\kappa)}$ is 
a centred normal distribution with covariance matrix $\sigma= \sqrt{1+\kappa (2\pi m)^2} \mathrm{Id}_{n\times n}$},  
the second integral on the right hand side of \eqref{eq:product_decomp} can be estimated uniformly in $\kappa \geq 1$ as
\begin{equs}
 \int \mu_\perp^{(\kappa)}(\dd u^{\perp}) \, \exp\left\{\frac{1}{2} \|u^\perp\|_2^2 \right\} & = \prod_{m\neq 0} \int \hat \mu_m^{(\kappa)}(\dd \hat u(m)) 
 \, \exp\left\{\frac{1}{2} |\hat u(m)|^2 \right\} 
 \\
 & = \left(\prod_{m\neq 0} \frac{1+\kappa (2\pi m)^2}{2\pi}\right)^{\frac{n}{2}} \int \dd \hat u(m)
 \, \exp\left\{-\frac{\kappa (2\pi m)^2}{2} |\hat u(m)|^2 \right\}
 \\
 & = \left(\prod_{m\neq 0} \frac{1+\kappa (2\pi m)^2}{\kappa (2\pi m)^2}\right)^{\frac{n}{2}} 
 = \exp\left\{ \frac{n}{2} \sum_{m\neq 0} \log\left(1+\frac{1}{\kappa (2\pi m)^2}\right) \right\} 
 \\
 & \leq \exp\left\{\frac{n}{2} \sum_{m\neq 0} \frac{1}{(2\pi m)^2}\right\}. \label{eq:change_of_reference_measure}
\end{equs}
\end{proof}

\begin{lemma}\label{lem:fernique_u_perp} There exists $b>0$ such that for every $M\geq 1$, $c>0$, and $\eps\in(0,1]$ we have that
\begin{equs}
 \sup_{\kappa\geq 1}\int_{\left\{\|\sqrt{\eps}u^{\perp}\|_\infty \geq c; \, \|\sqrt{\eps}u\|_\infty\leq M\right\}} \mu^{(\kappa)}(\dd u) \, \exp\left\{\frac{1}{2} \|u\|_2^2 \right\}
 \lesssim \left(\frac{M}{\sqrt{\eps}}\right)^{n} \exp\left\{-\frac{b c^2}{2\eps}\right\}.
\end{equs}
\end{lemma}

\begin{proof} We follow the proof of Lemma~\ref{lem:uniform_M}. As in \eqref{eq:product_decomp} combined with \eqref{eq:average_M_bound} we have that
\begin{equs}
 {} & \int_{\left\{\|\sqrt{\eps}u^{\perp}\|_\infty \geq c; \, \|\sqrt{\eps}u\|_\infty\leq M\right\}} \mu^{(\kappa)}(\dd u) \, \exp\left\{\frac{1}{2} \|u\|_2^2 \right\}
 \\
 & \quad \leq\int_{\left\{|\sqrt{\eps}\bar u|_\infty\leq M\right\}} \bar\mu(\dd \bar u) \, \exp\left\{\frac{1}{2} |\bar u|_2^2 \right\} 
 \int_{\left\{\|\sqrt{\eps}u^{\perp}\|_\infty \geq c\right\}} \mu_\perp^{(\kappa)}(\dd u^{\perp}) \, \exp\left\{\frac{1}{2} \|u^\perp\|_2^2 \right\}
 \\
 & \quad \lesssim \left(\frac{M}{\sqrt{\eps}}\right)^{n} 
 \int_{\left\{\|\sqrt{\eps}u^{\perp}\|_\infty \geq c\right\}} \mu_\perp^{(\kappa)}(\dd u^{\perp}) \, \exp\left\{\frac{1}{2} \|u^\perp\|_2^2 \right\}. \label{eq:product_decomp_fernique}
\end{equs}
Decomposing $\mu_\perp^{(\kappa)}$ in Fourier space as in the proof of Lemma~\ref{lem:uniform_M}, we first note that 
\begin{equs}
\mu_\perp^{(\kappa)}(\dd u^{\perp}) \, \exp\left\{\frac{1}{2} \|u^\perp\|_2^2\right\} & = 
\left(\prod_{m\neq 0}  \frac{1+\kappa (2\pi m)^2}{2\pi}\right)^{\frac{n}{2}} \dd \hat u(m)
 \, \exp\left\{-\frac{\kappa (2\pi m)^2}{2} |\hat u(m)|^2 \right\}
 \\
 &=  \left(\prod_{m\neq 0}\frac{1+\kappa (2\pi m)^2}{\kappa (2\pi m)^2}\right)^{\frac{n}{2}}
 \left(\prod_{m\neq 0} \frac{\kappa (2\pi m)^2}{2\pi}\right)^{\frac{n}{2}} \dd \hat u(m)
 \\
 & \quad \times  \exp\left\{-\frac{\kappa (2\pi m)^2}{2} |\hat u(m)|^2 \right\}.
 \\
& = \left(\prod_{m\neq 0} \frac{1+\kappa (2\pi m)^2}{\kappa (2\pi m)^2}\right)^{\frac{n}{2}}  \tilde\mu_\perp^{(\kappa)}(\dd u^{\perp}),
\label{eq:fourier_decomp} 
\end{equs}
where  $\tilde\mu_\perp^{(\kappa)}(\dd u^{\perp})$ is the centred Gaussian measure on $u^\perp$ with covariance $(-\kappa\Delta)^{-1} \mathrm{Id}_{n\times n}$. Using that the law of $\sqrt{\kappa}^{-1} u^\perp$ under 
$\tilde\mu_\perp^{(1)}$ is given by $\tilde\mu_\perp^{(\kappa)}$, we can rewrite the integral on the right hand side of 
\eqref{eq:product_decomp_fernique} as
\begin{equs}
\int_{\left\{\|\sqrt{\eps}u^{\perp}\|_\infty \geq c\right\}} \mu_\perp^{(\kappa)}(\dd u^{\perp}) \, \exp\left\{\frac{1}{2} \|u^\perp\|_2^2 \right\} 
   = \left(\prod_{m\neq 0} \frac{1+\kappa (2\pi m)^2}{\kappa (2\pi m)^2}\right)^{\frac{n}{2}} 
 \int_{\left\{\|\sqrt{\frac{\eps}{\kappa}}u^{\perp}\|_\infty \geq c\right\}} \tilde\mu_\perp^{(1)}(\dd u^{\perp}). 
\end{equs}
Therefore, using the above representation, for $b>0$ sufficiently small which we fix below, we have
\begin{equs}
 {} & \int_{\left\{\|\sqrt{\eps}u^{\perp}\|_\infty \geq c\right\}} \mu_\perp^{(\kappa)}(\dd u^{\perp}) \, \exp\left\{\frac{1}{2} \|u^\perp\|_2^2 \right\} 
 \\
 & \quad \leq \exp\left\{ \frac{n}{2} \sum_{m\neq 0} \log\left(1+\frac{1}{\kappa (2\pi m)^2}\right) \right\} 
 \exp\left\{-\frac{\kappa b c^2}{\eps}\right\} 
 \int \tilde\mu_\perp^{(1)}(\dd u^{\perp}) \exp\left\{\frac{b}{2}\|u^\perp\|_\infty^2\right\}
 \\
 & \quad \leq \exp\left\{\frac{n}{2} \sum_{m\neq 0} \frac{1}{(2\pi m)^2}\right\} 
 \exp\left\{-\frac{b c^2}{\eps}\right\} 
 \int \tilde\mu_\perp^{(1)}(\dd u^{\perp}) \exp\left\{\frac{b}{2}\|u^\perp\|_\infty^2\right\}, 
\end{equs}
where the final bound is uniform in $\kappa\geq 1$. Fernique's theorem 
\cite[Theorem~2.8.5]{Bo98} ensures the existence of $b>0$ 
such that
\begin{equs}
 \int \tilde\mu_\perp^{(1)}(\dd u^{\perp}) \exp\left\{\frac{b}{2}\|u^\perp\|_\infty^2\right\}< \infty
\end{equs}
which completes the proof. 
\end{proof}

From now on we let $\|u\|_\alpha :=|\bar u| + |u^\perp|_\alpha$ where $|u^\perp|_\alpha$ is the
usual H\"older semi-norm of order $\alpha\in(0,1)$. The next lemma is a finer version of Fernique's theorem 
for $\mu^{(\kappa)}(\dd u)$ with uniform control on $\kappa\geq 1$, which will be used in the proofs of 
Lemma~\ref{lem:varadhan} and Lemma~\ref{lem:ellis--rosen}. 

\begin{lemma}\label{lem:fernique_u} Let $\alpha\in(0,\frac{1}{2})$. There exists $b>0$ such that for every $M\geq 1$  and $\eps\in(0,\frac{b}{4}]$ we have that
\begin{equs}
 \sup_{\kappa\geq 1}\int_{\{\|\sqrt{\eps}u\|_\alpha\geq M\}} \mu^{(\kappa)}(\dd u) \, \exp\left\{\|\sqrt{\eps} u\|_\alpha^2\right\}
 \lesssim \exp\left\{-\frac{b M^2}{16\eps}\right\}.
\end{equs}
For $\alpha=0$ the statement holds if we replace $\|\cdot\|_\alpha$ by $\|\cdot\|_\infty$. 
\end{lemma}

\begin{proof} By the triangle inequality we know that
\begin{equs}
 {} & \int_{\{\|\sqrt{\eps}u\|_\alpha\geq M\}} \mu^{(\kappa)}(\dd u) \, \exp\left\{\|\sqrt{\eps} u\|_\alpha^2\right\} 
 \\
 & \quad \leq \int_{\{|\sqrt{\eps}\bar u|\geq \frac{M}{2}\}} \mu^{(\kappa)}(\dd u) \, \exp\left\{\|\sqrt{\eps} u\|_\alpha^2\right\} 
 + \int_{\{|\sqrt{\eps}u^{\perp}|_\alpha\geq M\}} \mu^{(\kappa)}(\dd u) \, \exp\left\{\|\sqrt{\eps} u\|_\alpha^2\right\}. 
 \label{eq:fernique_u_triagle}
\end{equs}
The first term on the right hand side of \eqref{eq:fernique_u_triagle} can be estimated similarly to \eqref{eq:product_decomp} as
\begin{equs}
 {} & \int_{\{|\sqrt{\eps}\bar u|\geq \frac{M}{2}\}} \mu^{(\kappa)}(\dd u) \, \exp\left\{\|\sqrt{\eps} u\|_\alpha^2\right\}  
 \\
 & \quad \leq \int_{\{|\sqrt{\eps}\bar u|\geq \frac{M}{2}\}} \bar\mu(\dd \bar u) \, \exp\left\{|\sqrt{\eps} \bar u|^2\right\}  
 \int \mu_\perp^{(\kappa)}(\dd u^{\perp}) \, \exp\left\{|\sqrt{\eps} u^\perp|_\alpha^2\right\}. \label{eq:product_decomp_fernique_u}
\end{equs}
For $b>0$ sufficiently small which we fix below and $\eps\in(0,\frac{b}{4}]$, the first integral on the right hand side of \eqref{eq:product_decomp_fernique_u} can be estimated as 
\begin{equs}
 \int_{\{|\sqrt{\eps}\bar u|\geq \frac{M}{2}\}} \bar\mu(\dd \bar u) \, \exp\left\{|\sqrt{\eps} \bar u|^2\right\} 
 & \leq \exp\left\{-\frac{bM^2}{16\eps}\right\} \int \bar\mu(\dd \bar u) \, \exp\left\{\frac{b}{2}|\bar u|^2\right\}.
\end{equs}
Using \eqref{eq:fourier_decomp} and the fact that the law of $\sqrt{\kappa}^{-1} u^\perp$ under 
$\tilde\mu_\perp^{(1)}$ is given by $\tilde\mu_\perp^{(\kappa)}$, the second integral on the right hand side of 
\eqref{eq:product_decomp_fernique_u} can be rewritten as
\begin{equs}
{} & \int \mu_\perp^{(\kappa)}(\dd u^{\perp}) \, \exp\left\{|\sqrt{\eps} u^\perp|_\alpha^2\right\} 
 \\
 & \quad  = \left(\prod_{m\neq 0} \frac{1+\kappa (2\pi m)^2}{\kappa (2\pi m)^2}\right)^{\frac{n}{2}} 
 \int \tilde\mu_\perp^{(1)} \, \exp\left\{-\frac{1}{2}\|\sqrt{\frac{1}{\kappa}} u^\perp\|_2^2\right\} \exp\left\{|\sqrt{\frac{\eps}{\kappa}} u^\perp|_\alpha^2\right\}
\end{equs}
and, therefore, it is estimated by
\begin{equs}
 {} & \int \mu_\perp^{(\kappa)}(\dd u^{\perp}) \, \exp\left\{|\sqrt{\eps} u^\perp|_\alpha^2\right\} 
 \\
 & \quad \leq \exp\left\{ \frac{n}{2} \sum_{m\neq 0} \log\left(1+\frac{1}{\kappa (2\pi m)^2}\right) \right\} 
 \int \tilde\mu_\perp^{(1)}(\dd u^{\perp}) \, \exp\left\{\frac{b}{4}|u^\perp|_\alpha^2\right\} 
 \\
 & \quad \leq \exp\left\{ \frac{n}{2} \sum_{m\neq 0} \frac{1}{(2\pi m)^2} \right\} 
 \int \tilde\mu_\perp^{(1)}(\dd u^{\perp}) \, \exp\left\{\frac{b}{4}|u^\perp|_\alpha^2\right\},
\end{equs}
uniformly in $\kappa\geq1$. Choosing $b>0$ sufficiently small by Fernique's theorem 
\cite[Theorem~2.8.5]{Bo98} we know that
\begin{equs}
 \int \bar\mu(\dd \bar u) \, \exp\left\{\frac{b}{2}|\bar u|^2\right\} < \infty, \quad 
 \int \tilde\mu_\perp^{(1)}(\dd u^{\perp}) \, \exp\left\{\frac{b}{4}|u^\perp|_\alpha^2\right\} < \infty, 
\end{equs}
which combined with \eqref{eq:product_decomp_fernique_u} implies that 
\begin{equs}
 \sup_{\kappa\geq 1}\int_{\{|\sqrt{\eps}\bar u|\geq \frac{M}{2}\}} \mu^{(\kappa)}(\dd u) \, \exp\left\{\|\sqrt{\eps} u\|_\alpha^2\right\}
 \lesssim \exp\left\{-\frac{bM^2}{16\eps}\right\}. 
\end{equs}
The same argument can be used to estimate the second term on the right hand side of \eqref{eq:fernique_u_triagle}.
\end{proof}

Before, we proceed to the proof of Lemma~\ref{lem:varadhan} we also need the following statement for the level sets of $u\mapsto \V(u) + \|u^\perp\|_\infty$. 

\begin{lemma}\label{lem:level_sets} Let $\dim \M=0$ or $\dim \M\neq 0$. For any $\delta>0$ and
$M\geq 1$ there exists $c\equiv c(\delta, M)>0$ such that whenever $\|u\|_\infty\leq M$ and
$\dist_{\|\cdot\|_\infty}(u, \M) > \delta$ we have that $\V(u) + \|u^\perp\|_\infty \geq c$, where $u^\perp=u-\int \dd x \, u(x)$. 
\end{lemma}

\begin{proof} Assume that $\V(u) + \|u^\perp\|_\infty \leq \lambda$ for some $\lambda \in(0,1]$ which we fix below. Then $\|u^\perp\|_\infty\leq \lambda$ and
by the mean value theorem and the boundedness of $\bar u$ and $u$ we get 
\begin{equs}
 |\V(u) - \V(\bar u)| = \left|\int \dd x \int_0^1 \dd s \, \nabla V\left(\bar u + s u^\perp(x)\right) \cdot u^\perp(x)\right| \lesssim_M \|u^\perp\|_\infty \leq \lambda.
\end{equs}
Since $\V(u)\leq \lambda$ we also have that $\V(\bar u) = V(\bar u) \lesssim_M \lambda$ and by choosing $\lambda\equiv \lambda(\delta, M)<\frac{\delta}{2}$
sufficiently small we can ensure that $\dist(\M, \bar u)\leq \frac{\delta}{2}$ and $\|u^\perp\|_\infty\leq \frac{\delta}{2}$, which in turn imply that
$\dist_{\|\cdot\|_\infty}(\M, u) \leq \delta$. 
\end{proof}

We are now ready to prove a quantified version of Varadhan's lemma which is uniform in $\kappa\geq 1$ and will be used in the 
proof of Lemma~\ref{lem:ellis--rosen} below to control some error terms.

\begin{lemma}\label{lem:varadhan} Let $\dim \M=0$ or $\dim \M\neq 0$, $F:\C^0(\TT;\RR^n) \to \RR$ with at most polynomial growth and $b>0$ as in the statement of Lemma~\ref{lem:fernique_u}. For every $\delta>0$ there exists $c\equiv c(b,\delta)>0$
such that for every $\eps\in(0,\frac{b}{4}]$ we have that
\begin{equs}
\sup_{\kappa\geq 1}\int_{\{\dist_{\|\cdot\|_\infty}(\sqrt{\eps}u,\M)\geq \delta \}} \mu^{(\kappa)}(\dd u) \, F(\sqrt{\eps}u) \exp\left\{-\frac{1}{\eps} \V(\sqrt{\eps}u) + \frac{1}{2} \|u\|_2^2 \right\}
\lesssim \exp\left\{-\frac{c}{2\eps}\right\}.
\end{equs}
\end{lemma}

\begin{proof} We let $M\geq 1$ (which we fix below) and split the integral into
\begin{equs}
 {} & \int_{\{\dist_{\|\cdot\|_\infty}(\sqrt{\eps}u,\M)\geq \delta\}} \mu^{(\kappa)}(\dd u) F(\sqrt{\eps}u) \exp\left\{-\frac{1}{\eps} \V(\sqrt{\eps}u) + \frac{1}{2} \|u\|_2^2 \right\}
 \\
 & \quad = \int_{\{\dist_{\|\cdot\|_\infty}(\sqrt{\eps}u,\M)\geq \delta\}\cap\{\|\sqrt{\eps}u\|_\infty \leq M\}} \mu^{(\kappa)}(\dd u) \, F(\sqrt{\eps}u) \exp\left\{-\frac{1}{\eps} \V(\sqrt{\eps}u) + \frac{1}{2} \|u\|_2^2 \right\}
 \\
 & \quad \quad + \int_{\{\dist_{\|\cdot\|_\infty}(\sqrt{\eps}u,\M)\geq \delta\}\cap\{\|\sqrt{\eps}u\|_\infty\geq M\}} \mu^{(\kappa)}(\dd u) \, F(\sqrt{\eps}u) \exp\left\{-\frac{1}{\eps} \V(\sqrt{\eps}u) + \frac{1}{2} \|u\|_2^2 \right\}.
 \\
 \label{eq:cut-off_decomp_M}
\end{equs}

For the first term on the right hand side of \eqref{eq:cut-off_decomp_M}, by Lemma~\ref{lem:level_sets} we find  $c\equiv c(\delta, M)>0$ such that $\V(\sqrt{\eps}u) + \|\sqrt{\eps}u^{\perp}\|_\infty \geq c$ . 
Hence, we estimate
\begin{equs}
 {} & \left|\int_{\{\|\sqrt{\eps}u\|_\infty\leq M\}} \mu^{(\kappa)}(\dd u) \, F(\sqrt{\eps}u) \exp\left\{-\frac{1}{\eps} \V(\sqrt{\eps}u) + \frac{1}{2} \|u\|_2^2 \right\}\right|
 \\
 & \quad \leq \sup_{\|u\|_\infty\leq M} |F(u)| \, 
 \int_{\left\{\V(\sqrt{\eps}u) + \|\sqrt{\eps}u^{\perp}\|_\infty \geq c; \, \|\sqrt{\eps}u\|_\infty\leq M\right\}} \mu^{(\kappa)}(\dd u) \, \exp\left\{-\frac{1}{\eps} \V(\sqrt{\eps}u) + \frac{1}{2} \|u\|_2^2 \right\}.
 \\
 \label{eq:2nd_term_est}
\end{equs}
The integral on the right hand side of \eqref{eq:2nd_term_est} can be bounded by
\begin{equs}
 {} & \int_{\left\{\V(\sqrt{\eps}u) + \|\sqrt{\eps}u^{\perp}\|_\infty \geq c; \, \|\sqrt{\eps}u\|_\infty\leq M\right\}} \mu^{(\kappa)}(\dd u) \, \exp\left\{-\frac{1}{\eps} \V(\sqrt{\eps}u) + \frac{1}{2} \|u\|_2^2 \right\}
 \\
 & \quad \leq \int_{\left\{\V(\sqrt{\eps}u) \geq \frac{c}{2}; \, \|\sqrt{\eps}u\|_\infty\leq M\right\}} \mu^{(\kappa)}(\dd u) \, \exp\left\{-\frac{1}{\eps} \V(\sqrt{\eps}u) + \frac{1}{2} \|u\|_2^2 \right\}
 \\
 & \quad \quad + \int_{\left\{\|\sqrt{\eps}u^{\perp}\|_\infty \geq \frac{c}{2}; \, \|\sqrt{\eps}u\|_\infty\leq M\right\}} \mu^{(\kappa)}(\dd u) \, \exp\left\{-\frac{1}{\eps} \V(\sqrt{\eps}u) + \frac{1}{2} \|u\|_2^2 \right\}.
 \label{eq:2nd_term_decomp_further}
\end{equs}
The first term on the right hand side of \eqref{eq:2nd_term_decomp_further} can be estimated uniformly in
$\kappa\geq 1$ using Lemma~\ref{lem:uniform_M} by
\begin{equs} 
 {} & \int_{\left\{\V(\sqrt{\eps}u) \geq \frac{c}{2}; \, \|\sqrt{\eps}u\|_\infty\leq M\right\}} \mu^{(\kappa)}(\dd u) \, \exp\left\{-\frac{1}{\eps} \V(\sqrt{\eps}u) + \frac{1}{2} \|u\|_2^2 \right\}
  \\
 & \quad \leq \exp\left\{-\frac{c}{2\eps}\right\} \int_{\left\{\|\sqrt{\eps}u\|_\infty\leq M\right\}} \mu^{(\kappa)}(\dd u) \, \exp\left\{\frac{1}{2} \|u\|_2^2 \right\}
 \lesssim \left(\frac{M}{\sqrt{\eps}}\right)^{n} \exp\left\{-\frac{c}{2\eps}\right\}.
 \label{eq:2nd_term_decomp_further_1}
\end{equs}
For second term on the right hand side of \eqref{eq:2nd_term_decomp_further}, using the non-negativity of $\V$ and Lemma~\ref{lem:fernique_u_perp} we know that there exists $b>0$ 
such that
\begin{equs}
 {} & \int_{\left\{\|\sqrt{\eps}u^{\perp}\|_\infty \geq \frac{c}{2}; \, \|\sqrt{\eps}u\|_\infty\leq M\right\}} \mu^{(\kappa)}(\dd u) \, \exp\left\{-\frac{1}{\eps} \V(\sqrt{\eps}u) + \frac{1}{2} \|u\|_2^2 \right\}
 \\
 & \quad \leq \int_{\left\{\|\sqrt{\eps}u^{\perp}\|_\infty \geq \frac{c}{2}; \, \|\sqrt{\eps}u\|_\infty\leq M\right\}} \mu^{(\kappa)}(\dd u) \, \exp\left\{\frac{1}{2} \|u\|_2^2 \right\}
 \lesssim \left(\frac{M}{\sqrt{\eps}}\right)^{n} \exp\left\{-\frac{b c^2}{8\eps}\right\} 
 \label{eq:2nd_term_decomp_further_2}
\end{equs}
uniformly in $\kappa\geq 1$. Combining \eqref{eq:2nd_term_est} with the polynomial growth of $F$, \eqref{eq:2nd_term_decomp_further}, \eqref{eq:2nd_term_decomp_further_1}, and
\eqref{eq:2nd_term_decomp_further_2}, we find $\eta>0$ such that
\begin{equs}
 {} & \left|\int_{\{\dist_{\|\cdot\|_\infty}(\sqrt{\eps}u,\M)\geq \delta\}\cap\{\|\sqrt{\eps}u\|_\infty\leq M\}} \mu^{(\kappa)}(\dd u) F(\sqrt{\eps}u) \exp\left\{-\frac{1}{\eps} \V(\sqrt{\eps}u) + \frac{1}{2} \|u\|_2^2 \right\}\right|
 \\
 & \quad \lesssim_M \exp\left\{-\frac{\eta}{2\eps}\right\}, 
 \label{eq:cut-off_decomp_M_final_1}
\end{equs}
uniformly in $\kappa\geq 1$.

To estimate the second term on the right hand side of \eqref{eq:cut-off_decomp_M}, we first use the coercivity of $\V$ \eqref{eq:coercivity}
in the form $-\V(u) + \frac{1}{2}\|u\|_2^2 \leq C$ which implies the estimate
\begin{equs}
 {} & \int_{\{\dist_{\|\cdot\|_\infty}(\sqrt{\eps}u,\M)\geq \delta\}\cap\{\|\sqrt{\eps}u\|_\infty\geq M\}} \mu^{(\kappa)}(\dd u) \, F(\sqrt{\eps}u) \exp\left\{-\frac{1}{\eps} \V(\sqrt{\eps}u) + \frac{1}{2} \|u\|_2^2 \right\}
 \\
 & \quad \lesssim \exp\left\{\frac{C}{\eps}\right\} \int_{\{\|\sqrt{\eps}u\|_\infty\geq M\}} \mu^{(\kappa)}(\dd u) \, F(\sqrt{\eps}u).
\end{equs}
Using the fact that $|F(u)| \lesssim 1 + \|u\|_\infty^{2p}$, for some $p\geq0$, and Lemma \ref{lem:fernique_u}
there exists $b>0$ such that for any $\eps\leq \frac{b}{4}$  
\begin{equs}
 {} & \int_{\{\|\sqrt{\eps}u\|_\infty\geq M\}} \mu^{(\kappa)}(\dd u) \, F(\sqrt{\eps}u)
 \lesssim_p \int_{\{\|\sqrt{\eps}u\|_\infty\geq M\}} \mu^{(\kappa)}(\dd u) \, \exp\left\{\|\sqrt{\eps} u\|_\infty^2\right\}
 \lesssim \exp\left\{-\frac{b M^2}{16\eps}\right\},
\end{equs}
uniformly in $\kappa\geq 1$. Choosing $M\geq\sqrt{\frac{3C}{2b}}$ we have proved that
\begin{equs}
 \int_{\{\|\sqrt{\eps}u\|_\infty\geq M\}} \mu^{(\kappa)}(\dd u) \, F(\sqrt{\eps}u) \exp\left\{-\frac{1}{\eps} \V(\sqrt{\eps}u) + \frac{1}{2} \|u\|_2^2 \right\}
 \lesssim \exp\left\{-\frac{C}{2\eps}\right\},
 \label{eq:cut-off_decomp_M_final_2} 
\end{equs}
uniformly in $\kappa\geq 1$ for any $\eps\leq \frac{b}{4}$. 

Finally, by \eqref{eq:cut-off_decomp_M_final_1} and \eqref{eq:cut-off_decomp_M_final_2}, for any $\eps\leq \frac{b}{4}$
we get that  
\begin{equs}
 \int \mu^{(\kappa)}(\dd u) \, F(\sqrt{\eps}u) \exp\left\{-\frac{1}{\eps} \V(\sqrt{\eps}u) + \frac{1}{2} \|u\|_2^2 \right\}
 \lesssim  \exp\left\{-\frac{\eta\wedge C}{2\eps}\right\},
\end{equs}
uniformly in $\kappa\geq 1$. 
\end{proof}

For a trace-class operator $L$ on a Hilbert space $H$ with eigenvalues $(\lambda_j)_{j\in J}$, where $J$ is a countable set, 
the Fredholm determinant $\det(\mathrm{Id} + L)$ is defined via
\begin{equs}
\det(\mathrm{Id}+ L) := \prod_{j\in J} (1+\lambda_j). 
\end{equs}
In Lemma~\ref{lem:ellis--rosen} the Fredholm determinants of
$(-\kappa \Delta + 1)^{-1}(\nabla^{(2)}V(w) - \mathrm{Id}_{n\times n})$ for $\dim M=0$ and $\big(P_{\N_w}(-\kappa \Delta + 1) \mathrm{Id}_{n\times n} P_{\N_w}\big)^{-1} P_{\N_w}(\nabla^{(2)}V(w) - \mathrm{Id}_{n\times n}) P_{\N_w}$ for $\dim \M\neq 0$ appear due to a transformation of Gaussian measures. Note that the specific operators are trace-class since $(-\kappa \Delta + 1)^{-1}$ is trace-class. In the next lemma we compute these determinants and prove an estimate which is uniform in $\kappa\geq 1$. 

\begin{lemma}\label{lem:det_uniform} In the case $\dim \M =0$, for every $w\in \M$ we have that 
\begin{equs}
 {} & \det\left(\mathrm{Id} + (-\kappa \Delta + 1)^{-1} (\nabla^{(2)}V(w) - \mathrm{Id}_{n\times n}) \right)
 \\
 & \quad = \det\left(\nabla^{(2)}V(w)\right)
 \prod_{m\neq 0}
 \frac{\det\left(\kappa (2\pi m)^2 \mathrm{Id}_{n\times n}+\nabla^2 V(w)\right)}{\det\big((\kappa (2\pi m)^2+1)\mathrm{Id}_{n\times n}\big)}. \label{eq:det_explicit_0}
\end{equs}
In the case $\dim \M\neq 0$, for every $w\in \M$ we have that 
\begin{equs}
 {} & \det\left(\mathrm{Id}_{\N_w} + \big(P_{\N_w}(-\kappa \Delta + 1) \mathrm{Id}_{n\times n} P_{\N_w}\big)^{-1} P_{\N_w}(\nabla^{(2)}V(w) - \mathrm{Id}_{n\times n})P_{\N_w} \right)
 \\
 & \quad = \det\left(P_{\N_w}\nabla^{(2)}V(w)P_{\N_w}\right)
 \prod_{m\neq 0}
 \frac{\det\left(\kappa (2\pi m)^2 \mathrm{Id}_{n\times n}+\nabla^2 V(w)\right)}{\det\big((\kappa (2\pi m)^2+1)\mathrm{Id}_{n\times n}\big)}. \label{eq:det_explicit}
\end{equs}
In addition, the following uniform estimates hold respectively,  
\begin{equs}
 0 & < \inf_{w\in\M}\det\left(\nabla^{(2)}V(w)\right)  \leq \sup_{w\in\M}\det\left(\nabla^{(2)}V(w)\right) <\infty \label{eq:det_uniform_0}
\end{equs}
and

\begin{equs}
 0 & < \inf_{w\in\M}\det\left(P_{\N_w}\nabla^{(2)}V(w)P_{\N_w}\right)  \leq \sup_{w\in\M}\det\left(P_{\N_w}\nabla^{(2)}V(w)P_{\N_w}\right) <\infty, \label{eq:det_uniform_1}
\end{equs}
as well as the uniform estimate
\begin{equs}
 0 & < \inf_{\kappa\geq 1}\inf_{w\in\M} \prod_{m\neq 0}
 \frac{\det\left(\kappa (2\pi m)^2 \mathrm{Id}_{n\times n}+\nabla^2 V(w)\right)}{\det\big((\kappa (2\pi m)^2+1)\mathrm{Id}_{n\times n}\big)}
 \\
 & \leq
 \sup_{\kappa\geq 1}\sup_{w\in\M} \prod_{m\neq 0}
 \frac{\det\left(\kappa (2\pi m)^2 \mathrm{Id}_{n\times n}+\nabla^2 V(w)\right)}{\det\big((\kappa (2\pi m)^2+1)\mathrm{Id}_{n\times n}\big)}
 \\
 & \leq 
 \sup_{\kappa\geq 1}\sup_{w\in\M} \prod_{m\neq 0}
 \frac{\det\left(\kappa (2\pi m)^2 \mathrm{Id}_{n\times n}+\nabla^2 V(w)\right)}{\det\big(\kappa (2\pi m)^2\mathrm{Id}_{n\times n}\big)}
 < \infty. \label{eq:det_uniform_2}
\end{equs}
\end{lemma}

\begin{proof} For simplicity we only prove \eqref{eq:det_explicit}. The proof of \eqref{eq:det_explicit_0} follows similarly. For fixed $w\in \M$, let $\{\eta_i(w)\}_{1\leq i \leq \dim \T_w^\perp}$ be a basis of eigenvectors of $\nabla^2 V(w)$
on $\T_w^\perp$ which extends to a basis $\{\eta_i(w)\}_{1\leq i\leq n}$ of eigenvectors of $\nabla^2 V(w)$ on $\RR^n$ 
with eigenvalues $\{\lambda_i(w)\}_{1\leq i\leq n}$. The union $\{\eta_i(w)\}_{1\leq i\leq\dim \T_w^\perp} \cup 
\{\mathrm{e}^{2\pi \mathrm{i} m \cdot}\eta_i(w) \}_{1\leq i\leq n;m\neq 0}$ forms a basis of $\N_w$ and $P_{\N_w}$ projects onto 
the span of $\{\eta_i(w)\}_{1\leq i\leq\dim \T_w^\perp}$. It is easy to check that for $i=1,\ldots, \dim \T_w^\perp$
\begin{equs}
{} &\bigg(\big(P_{\N_w}(-\kappa \Delta + 1) \mathrm{Id}_{n\times n} P_{\N_w}\big)^{-1} P_{\N_w}\big(\nabla^{(2)}V(w) - \mathrm{Id}_{n\times n}\big)P_{\N_w}\bigg)\eta_i(w) = (\lambda_i(w) - 1) \eta_i(w),
\end{equs}
while for $i=1,\ldots, n$ and $m\neq 0$
\begin{equs}
{} &\bigg(\big(P_{\N_w}(-\kappa \Delta + 1) \mathrm{Id}_{n\times n} P_{\N_w}\big)^{-1} 
P_{\N_w}\nabla^{(2)}\big(V(w) - \mathrm{Id}_{n\times n}\big)P_{\N_w}\bigg) \mathrm{e}^{2\pi \mathrm{i} m \cdot} \eta_i(w) 
\\ & \quad = (-\kappa \Delta + 1)^{-1}\mathrm{e}^{2\pi \mathrm{i} m \cdot} \, \big(\nabla^{(2)}V(w) - \mathrm{Id}_{n\times n}\big)\eta_i(w)
= \big(\kappa(2\pi m)^2 + 1\big)^{-1} (\lambda_i(w) - 1) \mathrm{e}^{2\pi \mathrm{i} m \cdot} \eta_i(w).
\end{equs}
Therefore, by definition the Fredholm determinant is given by 
\begin{equs}
{} & \det\left(\mathrm{Id}_{\N_w} + \big(P_{\N_w}(-\kappa \Delta + 1) \mathrm{Id}_{n\times n} P_{\N_w}\big)^{-1} P_{\N_w}\nabla^{(2)}\big(V(w) - \mathrm{Id}_{n\times n}\big)P_{\N_w} \right) 
\\
& \quad = \Big[\prod_{i=1, \ldots,\dim \T_w^{\perp}} \big(1 + (\lambda_i(w)-1) \big)\Big] 
\Big[\prod_{\substack{m\neq 0 \\ i=1, \ldots,n}} 
\Big(1 + \big(\kappa(2\pi m)^2 + 1\big)^{-1} (\lambda_i(w) - 1)\Big)\Big]
\\
& \quad = \Big[\prod_{i=1, \ldots,\dim \T_w^{\perp}} \lambda_i(w)\Big]  \Big[\prod_{\substack{m\neq 0 \\ i=1, \ldots,n}} 
\frac{\kappa(2\pi m)^2 + \lambda_i(w)}{\kappa(2\pi m)^2 + 1}\Big], 
\end{equs}
which justifies \eqref{eq:det_explicit}.  

Due to the non-negativity of $\nabla^2 V(w)$ we know that $\lambda_i(w)\geq0$ and by the non-degeneracy condition  
\eqref{eq:non_degeneracy} on $\T_w^\perp$ we 
now that $\lambda_i(w)>0$ for every $i=1,\ldots, \dim \T_w^\perp$. 
Therefore, we obtain that $\det\left(P_{\N_w}\nabla^{(2)}V(w)P_{\N_w}\right)>0$ and
\eqref{eq:det_uniform_1} follows by compactness of $\M$ and continuity of $\nabla^{(2)}V(w)$ in $w$. The proof of 
\eqref{eq:det_uniform_0} follows similarly. 

In order to prove the lower bound in \eqref{eq:det_uniform_2} we notice that for every $\kappa \geq 1$ 
\begin{equs}
0<\prod_{\substack{m\neq 0 \\ i=1, \ldots,n}} \frac{(2\pi m)^2}{(2\pi m)^2 + 1}
\leq \prod_{\substack{m\neq 0 \\ i=1, \ldots,n}} 
\frac{\kappa(2\pi m)^2 + \lambda_i(w)}{\kappa(2\pi m)^2 + 1}
 = \prod_{m\neq 0} \frac{\det\left(\kappa (2\pi m)^2 \mathrm{Id}_{n\times n}+\nabla^2 V(w)\right)}{\det\big((\kappa (2\pi m)^2+1)\mathrm{Id}_{n\times n}\big)}. 
\end{equs}
Last, the upper bound in \eqref{eq:det_uniform_2} follows from the estimate
\begin{equs}
 \prod_{\substack{m\neq 0 \\ i=1, \ldots,n}} 
\frac{\kappa(2\pi m)^2 + \lambda_i(w)}{\kappa(2\pi m)^2 + 1}
 &
 \leq \prod_{m\neq 0} \Big(1+\frac{\sup_{1\leq i\leq n}\lambda_i(w)}{\kappa (2\pi m)^2}\Big)^{n}
 \exp\Big\{n \sum_{m\neq 0} \frac{\sup_{1\leq i\leq n}\lambda_i(w)}{(2\pi m)^2}\Big\},   
\end{equs}
and the fact that $\sup_{w\in \M}\sup_{1\leq i\leq n}\lambda_i(w)<\infty$. 
\end{proof}

In order to control the error term in Lemma~\ref{lem:ellis--rosen} below, we prove a version of 
Fernique's theorem for the measure $\mu^{(\kappa)}_w(\dd v)$ which is uniform in $\kappa\geq 1$ and $w\in \M$. As in the 
case of $\mu^{(\kappa)}_\eps(\dd u)$,
to eliminate the dependence on $\kappa$ we write $\mu^{(\kappa)}_w(\dd v)$ as the product of 
$\bar \mu^{(\kappa)}_w(\dd \bar v)$ and $\mu^{(\kappa)}_{\perp,w}(\dd v^\perp)$, where 
$\bar v =\int \dd x \, v(x)$ and $v^\perp = v - \bar v$.   

\begin{lemma}\label{lem:fernique_normal} Let $\dim \M=0$ or $\dim \M\neq 0$. There exists $b>0$ such that
\begin{equs}
 \sup_{\kappa\geq 1}\sup_{w\in \M}\int \mu^{(\kappa)}_w(\dd v) \, \exp\left\{\frac{b}{2}\|v\|_\infty^2\right\} <\infty. 
\end{equs}
\end{lemma}

\begin{proof} We prove the estimate in the case $\dim \M\neq 0$. The proof of the estimate 
in the case $\dim\M=0$ follows along the same lines. We decompose $\mu^{(\kappa)}_w(\dd v)$ into the product of $\bar \mu^{(\kappa)}_w(\dd \bar v)$ and
$\mu^{(\kappa)}_{\perp,w}(\dd v^\perp)$. Therefore
we have the estimate 
\begin{equs}
 {} & \int_{\N_w} \mu^{(\kappa)}_w(\dd v) \, \exp\left\{\frac{b}{2}\|v\|_\infty^2\right\} 
 \\
 & \quad \leq \int_{\T_w^\perp} \bar \mu^{(\kappa)}_w(\dd \bar v) \, \exp\left\{\frac{b}{2}|\bar v|^2\right\}
 \int \mu^{(\kappa)}_{\perp,w}(\dd v^\perp) \, \exp\left\{\frac{b}{2}\|v^\perp\|_\infty^2\right\}. \label{eq:fernique_normal_decomp}
\end{equs}
Using \eqref{eq:det_uniform_1}, \eqref{eq:non_degeneracy}, and choosing $b<b_*$, the first term on the right hand side of \eqref{eq:fernique_normal_decomp} can be estimated uniformly in $w\in\M$ as
\begin{equs}
 \int_{\T_w^\perp} \bar \mu^{(\kappa)}_w(\dd \bar v) \, \exp\left\{\frac{b}{2}|\bar v|^2\right\} 
 & = \frac{\det\left(P_{\N_w}\nabla^2V(w)P_{\N_w}\right)^{\frac{1}{2}}}{(2\pi)^{\frac{\dim \T_w^\perp}{2}}}
 \int \dd \bar v \, \exp\left\{-\frac{1}{2} \nabla^2V(w) \bar v \cdot \bar v\right\}  \exp\left\{\frac{b}{2}|\bar v|^2\right\} 
 \\
 & \lesssim \int \dd \bar v \, \exp\left\{-\frac{1}{2} b_* |\bar v|^2\right\} 
 \exp\left\{\frac{b}{2}|\bar v|^2\right\} \lesssim 1.
\end{equs}
We rewrite the second term on the right hand side of \eqref{eq:fernique_normal_decomp} as
\begin{equs}
 {} & \int \mu^{(\kappa)}_{\perp,w}(\dd v^\perp) \, \exp\left\{\frac{b}{2}\|v^\perp\|_\infty^2\right\} 
 \\
 & \quad = \left(\prod_{k\neq 0} \frac{\det\left(\kappa (2\pi k)^2\mathrm{Id}_{n\times n}+\nabla^2 V(w)\right)}{\det\big(\kappa (2\pi k)^2 \mathrm{Id}_{n\times n}\big)}\right)^{\frac{1}{2}}
 \\
 & \quad \quad \times\int \mu_{\perp}(\dd v^\perp) \, \exp\left\{-\frac{1}{2\kappa}\int \dd x \, v^\perp(x) \nabla^2 V(w) v^\perp(x)\right\} 
 \exp\left\{\frac{b}{2\kappa}\|v^\perp\|_\infty^2\right\}, 
 %
\end{equs}
where $\mu^{\perp}$ is the centred Gaussian measure on $v^\perp$ with covariance $(-\Delta)^{-1} \mathrm{Id}_{n\times n}$. 
Using \eqref{eq:det_uniform_2} and the fact that $\nabla^2 V(w)$ is positive definite for every $w\in \M$ this term is 
estimated uniformly in $\kappa\geq 1$ and $w\in \M$ by
\begin{equs}
\int \mu^{(\kappa)}_{\perp,w}(\dd v^\perp) \, \exp\left\{\frac{b}{2}\|v^\perp\|_\infty^2\right\} 
\lesssim \int \mu_{\perp}(\dd v^\perp) \exp\left\{\frac{b}{2}\|v^\perp\|_\infty^2\right\}. 
\end{equs}
By Fernique's theorem \cite[Theorem~2.8.5]{Bo98} we can choose $b>0$ sufficiently small such that
\begin{equs}
 \int \mu_{\perp}(\dd v^\perp) \exp\left\{\frac{b}{2}\|v^\perp\|_\infty^2\right\} < \infty,
\end{equs}
which completes the proof. 
\end{proof}

\subsection{Asymptotic expansion of Gaussian integrals}\label{app:ellis--rosen} 

In this subsection we provide an asymptotic expansion 
which allow to estimate the top Lyapunov exponent under Assumption~\ref{ass:V}\ref{it:fmm} and Assumption~\ref{ass:V}\ref{it:mmp} respectively in Section~\ref{sec:lyap_est} below. The proof is an
adaptation of the proofs of \cite[Theorem~4.6]{ER82I} and \cite[Theorem~5]{ER82II}. In the former work, the
authors consider unbounded functionals on a Hilbert space\footnote{which in the current framework is given by $L^2(\TT;\RR^n)$} and in the latter bounded functionals on a space of continuous functions, whereas in order to estimate the top Lyapunov exponent we need to work with functionals on $\C^0(\TT;\RR^n)$ of polynomial growth. An additional difficulty arises when tracing the 
dependency of $\eps_0$ on $\kappa$, where the proof is adapted suitably using the quantified estimates from Section~\ref{sec:uniform_kappa}. For the reader's 
convenience we give a sketch of the proof highlighting the differences from \cite{ER82I} and \cite{ER82II}.

Before we proceed to the main statements of this subsection, we introduce some extra notation. For $w\in\M$ and 
$v\in \C^0(\TT;\RR^n)$ we write
\begin{equs}
 H(w,v) & : = \V(w+v) - \V(w) - D^{(1)}\V(w)(v) - \frac{D^{(2)}\V(w)(v)}{2}.
\end{equs}
For $F:\C^{k+1}(\C^0(\TT;\RR^n); \RR)$ with derivatives of at most polynomial growth we let $G:\M\times \C^0(\TT;\RR^n)\to \RR$ defined through\footnote{For simplicity we omit the dependence of $G$ on $F$, $V$ and $W$ since it will be clear from the context.} 
\begin{equs}
G(w,v): = 
 \begin{cases}
  F(w+v) \exp\left\{-\frac{1}{\eps} H(w,v)\right\}, & \quad \dim \M=0, 
  \\
  F(w+v) \det(\mathrm{Id}_{\T_w} - W_{w,v}) \exp\left\{-\frac{1}{\eps} H(w,v)\right\}, & \quad \dim \M\neq 0, 
 \end{cases}
 \label{eq:G_fmm}
\end{equs}
where, as in \cite[Eq.~(1.8)]{ER82II}, $W$ denotes the Weingarten map, see Appendix~\ref{app:weingarten}. Due to the smoothness of $\M$ and the regularity assumptions on $F$ and $V$ we can expand $G$ to order $k$ in $v$ to obtain 
\begin{equs}
 {} & G(w,\sqrt{\eps}v)
 = \sum_{j=0}^k \frac{\sqrt{\eps}^j}{j!} Q_j(w,v) + R_{k+1}(w,\sqrt{\eps}v), \label{eq:Q_j}
\end{equs}
where $Q_j(w,v) = D^{(j)} G(w, 0)(v,\ldots,v)$, $D^{(j)}$ being the $j$-th Fr\'echet derivative of $G(w, v)$ in $v$, and 
$R_{k+1}(w,\sqrt{\eps}v)$ the Taylor remainder. Due to the polynomial growth of $F$, $V$, and their derivatives,
and the compactness of $\M$ there exist $r>0$ and uniform constants $C,c>0$ in $w\in \M$ such that for every $j=1,\ldots, k$,
\begin{equs}
|Q_j(w,v)| \leq C (1+\|v\|_\infty^r) \label{eq:Q_j_bound}
\end{equs}
and
\begin{equs}
 |R_{k+1}(w, \sqrt{\eps} v)| & \leq C \eps^{\frac{k+1}{2}} \|v\|_\infty^{k+1} \exp\left\{\frac{1}{\eps}|H(w,\sqrt{\eps}v)|\right\}
 \\
 & \leq C \eps^{\frac{k+1}{2}} \|v\|_\infty^{k+1} \exp\left\{\frac{c}{\eps} \|\sqrt{\eps} v\|_\infty^3\right\}. \label{eq:rem}
\end{equs}
Below, for $w\in \M$ we let
\begin{equs}\label{eq:theta_1}
 \theta^{(\kappa)}(w) & := 
 \det\left(\mathrm{Id} + (-\kappa \Delta + 1)^{-1} \big(\nabla^{(2)}V(w) -  \mathrm{Id}_{n\times n}\big)\right)^{-\frac{1}{2}}
\end{equs}
in the case $\dim \M=0$
\begin{equs}
 \theta^{(\kappa)}(w) & := 
 \det\left(2\pi \, P_{\T_w}\big((-\kappa \Delta + 1)\mathrm{Id}_{n\times n}\big)^{-1} P_{\T_w}\right)^{-\frac{1}{2}} 
 \\
 & \quad \times\det\left(\mathrm{Id}|_{\N_w} + \big(P_{\N_w}(-\kappa \Delta + 1) \mathrm{Id}_{n\times n} P_{\N_w}\big)^{-1} P_{\N_w}\big(\nabla^{(2)}V(w) -  \mathrm{Id}_{n\times n}\big)P_{\N_w} \right)^{-\frac{1}{2}}
 \\
 \label{eq:theta_2}
\end{equs}
in the case $\dim \M\neq 0$. Note that $\sup_{k\geq \kappa_0}\sup_{w\in\M}\theta^{(\kappa)}(w)<\infty$, which follows from \eqref{eq:det_explicit_0}, \eqref{eq:det_uniform_0} and \eqref{eq:det_uniform_2} in the case $\dim\M=0$ and \eqref{eq:det_explicit}, \eqref{eq:det_uniform_1}, \eqref{eq:det_uniform_2} and the fact that $P_{\T_w}\big((-\kappa \Delta + 1)\mathrm{Id}_{n\times n}\big)^{-1} P_{\T_w} = \mathrm{Id}|_{\T_w}$ in the case
$\dim \M\neq 0$.

In what follows, $\int_\M \dd w$ stands for the surface integral over $\M$, which in the case $\dim \M = 0$, namely $\M=\{w_i\}_{i=1}^m$, 
is given by $\sum_{i=1}^m \delta_{w_i}$. 

\begin{lemma}\label{lem:ellis--rosen} Assume that $V\in \C^{k+1}(\RR^n; \RR_{\geq 0})$ satisfies
the coercivity estimate \eqref{eq:coerc_V} for $p=2$, $\dim \M=0$ or $\dim \M\neq 0$ and \eqref{eq:non_degeneracy} holds. For every
$F\in \C^{k+1}(\C^0(\TT;\RR^n);\RR)$ with derivatives of at most polynomial growth there exists $\eps_0\in(0,1]$ such that for every 
$\eps\leq \eps_0$ and $\kappa>0$ 
\begin{equs}
{} & \eps^{\frac{\dim \M}{2}} \int  \mu_{\eps}^{(\kappa)}(\dd u) \, F(u) \exp\left\{-\frac{1}{\eps} \V(u) + \frac{1}{2} \|u\|_2^2\right\}
\\
& \quad = \int_\M \dd w \, \theta^{(\kappa)}(w) \int
 \mu_{w}^{(\kappa)}(\dd v) \, \sum_{j=0}^k \frac{\eps^{\frac{j}{2}}}{j!} Q_j(w,v)
 + \error^{(\kappa)}_\eps, \label{eq:ellis--rosen}
%
\end{equs}
where $\error^{(\kappa)}_\eps\in \RR$ with $|\error^{(\kappa)}_\eps|\leq C \big(\int_\M \dd w \, \theta^{(\kappa)}(w) +1\big) 
\eps^{\frac{k+1}{2}}$ for a constant $C$ which depends on $F$ and $V$ but is uniform in $\kappa\geq \kappa_0$ for every $\kappa_0>0$ and 
$Q_j(w,v) = D^{(j)} G(w, 0)(v,\ldots,v)$ for $G$ as in \eqref{eq:G_fmm}.
\end{lemma}

\begin{proof}[Proof in the case $\dim \M=0$.] First, we note that for every $\eps>0$,
\begin{equs}
 {} & \int \mu_\eps^{(\kappa)}(\dd u) \, F(u) \exp\left\{-\frac{1}{\eps} \V(u) + \frac{1}{2\eps} \|u\|_2^2 \right\} 
 \\
 & \quad = \int \mu^{(\kappa)}(\dd u) \, F(\sqrt{\eps}u) \exp\left\{-\frac{1}{\eps} \V(\sqrt{\eps}u) + \frac{1}{2} \|u\|_2^2 \right\}, 
 \label{eq:rescaling}
\end{equs}
where $\mu^{(\kappa)}$ denotes the centred Gaussian measure with covariance $(-\kappa\Delta+1)^{-1} \Id_{n\times n}$. Indeed, if $u$ denotes 
a Gaussian random variable with law $\mu^{(\kappa)}$, then the linear transformation $u\to \sqrt{\eps} u$ induces a Gaussian random
variable with law $\mu_\eps^{(\kappa)}$, which is easy to verify through its characteristic function. 

For $\delta>0$ sufficiently small such that the sets $\{u: \, \|u-w_i\|_\infty<\delta\}$ are disjoint we consider a continuous function
cut-off function $0\leq\varphi\leq 1$ on $\C^0(\TT;\RR^n)$ such that
\begin{equs}
 \varphi(u)  := 
 \begin{cases}
   1, & \quad \dist_{\|\cdot\|_{\infty}}(u, \M)\leq \frac{\delta}{2}, \\
   0, & \quad \dist_{\|\cdot\|_{\infty}}(u, \M)\geq \delta.
 \end{cases}
\end{equs}
Note that such a $\varphi$ exist by Uryshon's  lemma, since $\{\dist_{\|\cdot\|_{\infty}}(u, \M)\leq \frac{\delta}{2}\}$ and 
$\{\dist_{\|\cdot\|_{\infty}}(u, \M)\geq \delta\}$ are closed in $\C^0(\TT;\RR^n)$. Using $\varphi$, we write 
\begin{equs}
{} & \int  \mu_{\eps}^{(\kappa)}(\dd u) \, F(\sqrt{\eps}u) \exp\left\{-\frac{1}{\eps} \V(\sqrt{\eps}u) + \frac{1}{2} \|\sqrt{\eps}u\|_2^2\right\}
\\
& \quad =
 \int \mu_{\eps}^{(\kappa)}(\dd u) \, \varphi(\sqrt{\eps}u) F(\sqrt{\eps}u) \exp\left\{-\frac{1}{\eps} \V(\sqrt{\eps}u) + \frac{1}{2} \|\sqrt{\eps}u\|_2^2\right\}
\\
& \quad \quad + \int  \mu_{\eps}^{(\kappa)}(\dd u) \, \big(1-\varphi(\sqrt{\eps}u)\big) F(\sqrt{\eps}u) \exp\left\{-\frac{1}{\eps} \V(\sqrt{\eps}u) + \frac{1}{2} \|\sqrt{\eps}u\|_2^2\right\}. \label{eq:fmm_decomp}
\end{equs}
By Lemma~\ref{lem:varadhan}, using that $1-\varphi(\sqrt{\eps}u)$ is bounded by $1$ and vanishes unless $\dist_{\|\cdot\|_{\infty}}(\sqrt{\eps}u, \M)\geq \frac{\delta}{2}$, we find $c\equiv c(\delta)>0$ and $b>0$ such that for every
$\eps\in(0,\frac{b}{4}]$ the second term on the right hand side of \eqref{eq:fmm_decomp} is estimated as
\begin{equs}
  \Big|\int  \mu_{\eps}^{(\kappa)}(\dd u) \, \big(1-\varphi(\sqrt{\eps}u)\big) F(\sqrt{\eps}u) \exp\left\{-\frac{1}{\eps} \V(\sqrt{\eps}u) + \frac{1}{2} \|\sqrt{\eps}u\|_2^2\right\}\Big|
 \lesssim \exp\left\{-\frac{c}{2\eps}\right\}, \quad \label{eq:fmm_6}
\end{equs}
uniformly in $\kappa\geq \kappa_0$. To treat the first term on the right hand side of \eqref{eq:fmm_decomp} we proceed as in the proof of \cite[Lemma~4.2]{ER82I} translating $u \mapsto u+\sqrt{\eps}^{-1} w$, which implies via change
of the reference measure, c.f. \cite[Lemma~4.5]{ER82I},   
\begin{equs}
{} & \int  \mu_{\eps}^{(\kappa)}(\dd u) \, \varphi(\sqrt{\eps}u) F(\sqrt{\eps}u) \exp\left\{-\frac{1}{\eps} \V(\sqrt{\eps}u) + \frac{1}{2} \|\sqrt{\eps}u\|_2^2\right\} 
\\
& \quad = \eps^{-\frac{\dim \M}{2}} \int_\M \dd w \, \theta^{(\kappa)}(w) \int \mu_{w}^{(\kappa)}(\dd v) \, \varphi(w+\sqrt{\eps} v) F(w+\sqrt{\eps} v) \exp\left\{-\frac{1}{\eps} H(w,\sqrt{\eps}v) \right\},
\\
\label{eq:fmm_1} 
\end{equs}
where $\mu_{w}^{(\kappa)}$ is the centred Gaussian measure with covariance 
$(-\kappa\Delta \, \mathrm{Id}_{n\times n} +\nabla^{(2)}V(w))^{-1}$. Temporally omitting the first integral and plugging
\eqref{eq:Q_j} in on the right hand side of \eqref{eq:fmm_1} we obtain 
\begin{equs}
{} & \int \mu_{w}^{(\kappa)}(\dd v) \, \varphi(w+\sqrt{\eps} v) G(w, \sqrt{\eps} v)
\\
& \quad = 
\int \mu_{w}^{(\kappa)}(\dd v) \, \varphi(w+\sqrt{\eps} v) \sum_{j=0}^k \frac{1}{j!} Q_j(w_i,v) 
+
\int \mu_{w}^{(\kappa)}(\dd v) \, \varphi(w+\sqrt{\eps} v) R_{k+1}(w, \sqrt{\eps}v)
\\
& \quad =  
\int \mu_{w}^{(\kappa)}(\dd v) \, \sum_{j=0}^k \frac{\eps^{\frac{j}{2}}}{j!} Q_j(w,v) +  
\int \mu_{w}^{(\kappa)}(\dd v) \, \big(\varphi(w+\sqrt{\eps} v)-1\big)  \sum_{j=0}^k \frac{1}{j!} Q_j(w,v) 
\\
&\quad \quad 
+ \int \mu_{w}^{(\kappa)}(\dd v) \, \varphi(w+\sqrt{\eps} v) R_{k+1}(w, \sqrt{\eps}v).  \label{eq:fmm_2}
\end{equs}
For the second term on the right hand side of \eqref{eq:fmm_2} we use the fact that $\varphi(w+\sqrt{\eps} v)-1$ vanishes 
unless $\|\sqrt{\eps} v\|_\infty \geq\frac{\delta}{2}$ together with Lemma~\ref{lem:fernique_normal}
and \eqref{eq:Q_j_bound}, therefore we find $b>0$ such that 
\begin{equs}
{} & \big| \int
 \mu_w^{(\kappa)}(\dd v) \, (\varphi(w+\sqrt{\eps}v)-1) \sum_{j=0}^k \frac{\eps^{\frac{j}{2}}}{j!} Q_j(w,v)\big|
 \lesssim 
 \int_{\{\|\sqrt{\eps}v\|_\infty > \frac{\delta}{2}\}} 
 \mu_w^{(\kappa)}(\dd v) \, (1+\|v\|_\infty^r) 
 \\
 & \quad \lesssim \exp\left\{-\frac{b\delta^2}{16\eps}\right\}  
 \int_{\{\|\sqrt{\eps}v\|_\infty > \frac{\delta}{2}\}} 
 \mu_w^{(\kappa)}(\dd v) \, (1+\|v\|_\infty^r) \exp\left\{\frac{b}{4} \|v\|_\infty^2\right\} 
 \\
 & \quad \lesssim_{b,r} \exp\left\{-\frac{b\delta^2}{16\eps}\right\} 
 \int
 \mu_w^{(\kappa)}(\dd v) \, \exp\left\{\frac{b}{2}\|v\|_\infty^2\right\}
 \lesssim \exp\left\{-\frac{b\delta^2}{16\eps}\right\}
 \label{eq:fmm_3}
\end{equs}
uniformly in $\kappa\geq \kappa_0$ and $w\in \M$. The third term on the right hand side of \eqref{eq:fmm_2} can be estimated 
using \eqref{eq:rem} as
\begin{equs}
 {} & \big| \int
 \mu_w^{(\kappa)}(\dd v) \, \varphi(w+\sqrt{\eps}v) R_{k+1}(w, \sqrt{\eps} v)\big|
 \\
 & \quad \stackrel{\eqref{eq:rem}}{\lesssim} \eps^{\frac{k+1}{2}}  
 \int \mu_w^{(\kappa)}(\dd v) \|v\|_\infty^{k+1} \exp\left\{c\delta \|v\|_\infty^2\right\} 
 \\
 & \quad \lesssim_{b,n}  \eps^{\frac{k+1}{2}}  
 \int \mu_w^{(\kappa)}(\dd v) \exp\left\{\left(\frac{b}{4} + c\delta\right) \|v\|_\infty^2\right\} 
 \lesssim \eps^{\frac{k+1}{2}}
 \label{eq:fmm_4}
\end{equs}
uniformly in $\kappa\geq \kappa_0$ and $w\in \M$, where in the last inequality we use again
Lemma~\ref{lem:fernique_normal}, provided 
that we choose $\delta<\frac{b}{4c}$. Combining \eqref{eq:fmm_1}, \eqref{eq:fmm_2}, 
\eqref{eq:fmm_3}, and \eqref{eq:fmm_4} we have proved that the first term on 
the right hand side of \eqref{eq:fmm_decomp} satisfies 
\begin{equs}
 {} & \big|\eps^{\frac{\dim \M}{2}} \int  \mu_{\eps}^{(\kappa)}(\dd u) \, \varphi(\sqrt{\eps}u) F(\sqrt{\eps}u) \exp\left\{-\frac{1}{\eps} \V(\sqrt{\eps}u) + \frac{1}{2} \|\sqrt{\eps}u\|_2^2\right\}
 \\
 & \quad -\int_\M \dd w \, \theta^{(\kappa)}(w)\int \mu_{w}^{(\kappa)}(\dd v) \, \sum_{j=0}^k \frac{\eps^{\frac{j}{2}}}{j!} Q_j(w,v)   \big|
 \lesssim \int_\M \dd w \, \theta^{(\kappa)}(w) \eps^{\frac{k+1}{2}} 
 \label{eq:fmm_5}
\end{equs}
uniformly in $\kappa\geq \kappa_0$. Combining \eqref{eq:fmm_decomp}, \eqref{eq:fmm_6} and \eqref{eq:fmm_5} yields \eqref{eq:ellis--rosen}.
\end{proof}

\begin{proof}[Proof in the case $\dim \M\neq 0$.] As in the case $\dim \M= 0$, we work with the right hand side of 
\eqref{eq:rescaling}. Following \cite[proof of Theorem 5]{ER82II}, we fix $\delta>0$ sufficiently small such that
$\mathcal{M}^{\delta} := \{u: \, \dist_{\|\cdot\|_2}(u,\M)<\delta\}$ is a uniformly tabular neighbourhood of $\M$ and we
introduce a continuous cut-off function $0\leq \varphi\leq 1$ on $\C^0(\TT;\RR^n)$ such that
\begin{equs}
 \varphi(u)  := 
 \begin{cases}
   1, & \quad u\in \{w+v: \, w \in \M, \, v\in \N_w \text{ with } \|v\|_\infty\leq \frac{\delta}{2}\}, \\
   0, & \quad u \in \{w+v: \, w \in \M, \, v\in \N_w \text{ with } \|v\|_\infty\geq \delta, \, \|v\|_2\leq \delta\} 
   \cup (\mathcal{M}^\delta)^{\mathrm{c}}.
 \end{cases}
\end{equs}
As in \cite[proof of Theorem 5]{ER82II} one can prove that the sets $\{w+v: \, w \in \M, \, v\in \N_w \text{ with } \|v\|_\infty\leq \frac{\delta}{2}\}$ and $\{w+v: \, w \in \M, \, v\in \N_w \text{ with } \|v\|_\infty\geq \delta, \, \|v\|_2\leq \delta\}$ are closed 
in $\C^0(\TT;\RR^n)$, therefore the existence of $\varphi$ follows by Uryshon's lemma. We then proceed as in the case 
$\dim \M=0$ using the decomposition \eqref{eq:fmm_decomp}. The main difference is the treatment of the second term on 
the right hand side of \eqref{eq:fmm_decomp} since the cut-off function $\varphi$ is slightly more involved here. 
In particular, we need to use that $1-\varphi(\sqrt{\eps}u)$ vanishes on the complement of
\begin{equs}
S_\delta:=\{w+v: \, w \in \M, \, v\in \N_w \text{ with } \|v\|_\infty\geq \tfrac{\delta}{2}, \, \|v\|_2\leq \tfrac{\delta}{2}\}\cup (\mathcal{M}^\frac{\delta}{2})^{\mathrm{c}},
\end{equs}
which implies that
\begin{equs}
{} & \big|\int \mu^{(\kappa)}(\dd u) \, \left(1-\varphi(\sqrt{\eps}u)\right) F(\sqrt{\eps}u) \exp\left\{-\frac{1}{\eps} \V(\sqrt{\eps}u) + \frac{1}{2} \|u\|_2^2 \right\}\big|
\\
& \quad \leq \int_{S_\delta} \mu^{(\kappa)}(\dd u) |F(\sqrt{\eps}u)| \exp\left\{-\frac{1}{\eps} \V(\sqrt{\eps}u) + \frac{1}{2} \|u\|_2^2 \right\}. 
\end{equs}
At this point we diverge from the proof of \cite[Theorem 5]{ER82II}, where Varadhan's lemma is used to control 
the integral over $S_\delta$, since we aim to explicitly keep track of the dependency on $\kappa$. We note that for every $\alpha\in(0,\tfrac{1}{2})$ and $M\geq \delta$ the set $S_\delta\cap\{u: \, \|u\|_\alpha\leq M\}$ is compact in
$\C^0(\TT;\RR^n)$ due to the compact embedding $\C^\alpha(\TT;\RR^n)\hookrightarrow \C^0(\TT;\RR^n)$. Since the set
$S_\delta$ is closed and $S_\delta\cap \M=\emptyset$, we find $\delta'>0$, possibly depending on $\delta$ and $M$, such that $S_\delta\cap\{u: \, \|u\|_\alpha\leq M\} \subset \{u: \, \dist_{\|\cdot\|_\infty}(u, \M)> \delta'$. 
Therefore, we obtain that
\begin{equs}
{} & \int_{S_\delta} \mu^{(\kappa)}(\dd u) |F(\sqrt{\eps}u)| \exp\left\{-\frac{1}{\eps} \V(\sqrt{\eps}u) + \frac{1}{2} \|u\|_2^2 \right\}
\\
& \quad \lesssim  \int_{\{\dist_{\|\cdot\|_\infty}(\sqrt{\eps}u, \M)> \delta'\}} \mu^{(\kappa)}(\dd u) F(\sqrt{\eps}u) \exp\left\{-\frac{1}{\eps} \V(\sqrt{\eps}u) + \frac{1}{2} \|u\|_2^2 \right\}
\\
& \quad \quad + \int_{\{\|\sqrt{\eps}u\|_\alpha\geq M\}} \mu^{(\kappa)}(\dd u) F(\sqrt{\eps}u) \exp\left\{-\frac{1}{\eps} \V(\sqrt{\eps}u) + \frac{1}{2} \|u\|_2^2 \right\}
\end{equs}
Using Lemma~\ref{lem:varadhan} we find $c\equiv c(\delta')$ and $b>0$ such that for every $\eps\in(0,\frac{b}{4}]$
\begin{equs}
 \int_{\{\dist_{\|\cdot\|_\infty}(\sqrt{\eps}u, \M)> \delta'\}} \mu^{(\kappa)}(\dd u) \, F(\sqrt{\eps}u) \exp\left\{-\frac{1}{\eps} \V(\sqrt{\eps}u) + \frac{1}{2} \|u\|_2^2 \right\}
 \lesssim \exp\left\{-\frac{c}{2\eps}\right\}
\end{equs}
uniformly in $\kappa\geq \kappa_0$. Using \eqref{eq:coercivity} we estimate the second term via
\begin{equs}
{} & \int_{\{\|\sqrt{\eps}u\|_\alpha\geq M\}} \mu^{(\kappa)}(\dd u) F(\sqrt{\eps}u) \exp\left\{-\frac{1}{\eps} \V(\sqrt{\eps}u) + \frac{1}{2} \|u\|_2^2 \right\}
\\
& \quad \leq \exp\left\{\frac{C}{\eps}\right\} \int_{\{\|\sqrt{\eps}u\|_\alpha\geq M\}} \mu^{(\kappa)}(\dd u) F(\sqrt{\eps}u)
\end{equs}
and due to the polynomial growth of $F$ by Lemma~\ref{lem:fernique_u} we find $b>0$ such that for every $\eps\in(0,\frac{b}{4}]$
\begin{equs}
\int_{\{\|\sqrt{\eps}u\|_\alpha\geq M\}} \mu^{(\kappa)}(\dd u) F(\sqrt{\eps}u) \leq \exp\left\{\frac{C}{\eps}\right\}  \exp\left\{-\frac{bM^2}{16\eps}\right\}
\end{equs}
uniformly in $\kappa\geq\kappa_0$. Choosing $M$ sufficiently large the right hand side of the last inequality is estimated by
$\exp\{-\frac{c}{2\eps}\}$ up to a constant which is uniform in $\kappa\geq \kappa_0$. In total, we are lead to \eqref{eq:fmm_6}.

We now turn back to the first term on the right hand side of \eqref{eq:fmm_decomp}, which as in \cite[Proof of Theorem 5]{ER82II} and since due to the presence of the cut-off function $\varphi$ we can assume that $F$ is bounded, takes the form
\begin{equs}
 {} & \int \mu^{(\kappa)}(\dd u) \, \varphi(\sqrt{\eps}u) F(\sqrt{\eps}u) \exp\left\{-\frac{1}{\eps} \V(\sqrt{\eps}u) + \frac{1}{2} \|u\|_2^2 \right\}
 \\
 & \quad = \eps^{-\frac{\dim \M}{2}}\int_{\M} \dd w \, \theta^{(\kappa)}(w) \int_{\{\|\sqrt{\eps}v\|_{\infty}< \delta\}}
 \mu_w^{(\kappa)}(\dd v) \, \varphi(w+\sqrt{\eps}v) F(w+\sqrt{\eps}v) 
 \\
 & \quad\quad\quad\quad\quad\quad\quad\quad\quad\quad\quad\quad \times \det(\mathrm{Id}_{\T_w} - W_{w,v}) 
 \exp\left\{-\frac{1}{\eps} H(w,\sqrt{\eps}v)\right\}, 
 \label{eq:change_of_var_form}
\end{equs}
where $W_{w,v}$ is the Weingarten map defined through \eqref{eq:weingarten_proper} and $\mu^{(\kappa)}_w$ is the 
centred Gaussian measure with covariance $(-\kappa\Delta \, \mathrm{Id}_{n\times n} + P_{\N_w}\nabla^{(2)}V(w) P_{\N_w})^{-1}$.
The rest of the proof is identical to the proof in the case $\dim \M=0$. 
\end{proof}

In order to estimate the top Lyapunov exponent under Assumption~\ref{ass:V}\ref{it:mmp}, we apply Lemma~\ref{lem:ellis--rosen}
to rotationally invariant $V: \RR^n\to \RR_{\geq 0}$, where the set of minimum points is given by the union of 
spheres $\M=\cup_{i=1}^m\Ss^{n-1}_{r_i}$. In this case the Weingarten map is given by \eqref{eq:weingarten} and the 
$Q_j$'s can be computed explicitly. Due to symmetry, the integral over $w$ does not appear and the expansion simplifies.
Since we are interested in first order expansions, we further restrict to the case $k=3$. We summarise in the next lemma. 

\begin{lemma}\label{lem:ellis--rosen_rot_inv} Assume that $V\in \C^4(\RR^n;\RR_{\geq 0})$, satisfies the coercivity estimate \eqref{eq:coerc_V} for $p=2$, it is rotationally invariant with $\M=\cup_{i=1}^m\Ss^{n-1}_{r_i}$ and 
\eqref{eq:non_degeneracy} holds. For every rotationally invariant $F\in\C^4(\C^0(\TT;\RR^n); \RR)$ 
with derivatives of at most polynomial growth there exists $\eps_0\in(0,1]$ such that for every $\eps\leq \eps_0$ and $\kappa>0$
\begin{equs}
 {} & \eps^{\frac{n-1}{2}} \int \mu_\eps^{(\kappa)}(\dd u) \, F(u) 
 \exp\left\{-\frac{1}{\eps} \V(u) + \frac{1}{2\eps} \|u\|_2^2 \right\}  
 \\
 & \quad = \sum_{i=1}^m |\Ss^{n-1}_{r_i}| \, \theta^{(\kappa)}(r_i\ee_1) \Bigg[ F(r_i\ee_1)
 \\
 & \quad \quad + \eps \int_{\N_{r_i\ee_1}} \mu_{r_i\ee_1}^{(\kappa)}(\dd v) 
 F(r_i\ee_1) \bigg(\frac{(n-1)(n-2)\lng r_i\ee_1,v\rng^2}{2r_i^4} - (n-1) \frac{1}{r_i^2} \lng r_i\ee_1,v\rng 
 \frac{D^{(3)}\V(r_i\ee_1)(v)}{3!} \bigg)
 \\
 & \quad \quad + \eps \int_{\N_{{r_i\ee_1}}} \mu_{r_i\ee_1}^{(\kappa)}(\dd v) \bigg(\frac{D^{(2)}F(r_i\ee_1)(v)}{2}+ D^{(1)}F(r_i\ee_1)(v) (d-1) \frac{1}{r_i^2} \lng r_i\ee_1, v \rng
 \\
 & \quad \quad \quad \quad \quad \quad \quad \quad \quad \quad \quad- 
 D^{(1)}F(r_i\ee_1)(v) \frac{D^{(3)}\mathbf{V}(r_i\ee_1)(v)}{3!}\bigg)\Bigg]
 + \error^{(\kappa)}_\eps, \label{eq:ellis--rosen_rot_inv} 
\end{equs}
for some $\error^{(\kappa)}_\eps\in \RR$ with $|\error^{(\kappa)}_\eps| \leq C \big(\sum_{i=1}^m |\Ss^{n-1}_{r_i}| \, \theta^{(\kappa)}(r_i\ee_1) +1\big) \eps^2$ for a constant $C$ which depends on $F$ and $V$ but is uniform in $\kappa\geq\kappa_0$ for every $\kappa_0>0$. 
Furthermore, for every $\eps\leq \eps_0$ and $\kappa\geq 0$ the partition function $\Z_\eps^{(\kappa)}$ satisfies
\begin{equs}
 \eps^{\frac{n-1}{2}} \Z_\eps^{(\kappa)} = \sum_{i=1}^m |\Ss^{n-1}_{r_i}| \, \theta^{(\kappa)}(r_i\ee_1) + \error^{(\kappa)}_\eps, \label{eq:ellis--rosen_partition} 
\end{equs}
for some $\error^{(\kappa)}_\eps\in \RR$ with $|\error^{(\kappa)}_\eps| \leq C \big(\sum_{i=1}^m |\Ss^{n-1}_{r_i}| \, \theta^{(\kappa)}(r_i\ee_1) +1 \big) \eps$ for a constant $C$ which depends on $V$ but is uniform in $\kappa\geq\kappa_0$
for every $\kappa_0>0$. 
\end{lemma}

\begin{proof} First note that for any $r>0$ and $(w, v)\in \Ss^{n-1}_r\times \N_w$, the Weingarten map $W_{w,v}$ is given in \eqref{eq:weingarten} 
by $W_{w,v}(h) = - \frac{1}{r^2}\lng w,v \rng \, h$, $h\in \T_w$. Hence, we have the identity $\det(\mathrm{Id}_{\T_w}- W_{w,v}) = (1+\frac{1}{r^2}\lng w, v \rng)^{n-1}$ which leads to the expansion
\begin{equs}
 \det(\mathrm{Id}_{\T_w} - W_{w,v}) = \sum_{k=0}^{n-1} \binom{n-1}{k} \frac{1}{r^{2k}}\lng w, v \rng^k 
 = 1 + (n-1)\frac{1}{r^{2}}\lng w, v \rng + \sum_{k=2}^{n-1} \binom{n-1}{k} \frac{1}{r^{2k}} \lng w, v \rng^k. 
\end{equs}
We now apply Lemma~\ref{lem:ellis--rosen} for $k=3$. An explicit calculation shows that 
\begin{equs}
 Q_0(w,v) & = F(w), 
 \\
 Q_2(w,v) & = 2 F(w) \left(\frac{(n-1)(n-2)\lng w,v\rng^2}{2r^4} - (n-1)\frac{1}{r^2} \lng w,v\rng 
 \frac{D^{(3)}\V(w)(v)}{3!} \right)
 \\
 & \quad + 
 2 \left(\frac{D^{(2)}F(w)(v)}{2}+ D^{(1)}F(w)(v) (n-1) \frac{1}{r^2} \lng w, v \rng - 
 D^{(1)}F(w)(v) \frac{D^{(3)}\mathbf{V}(w)(v)}{3!}\right). 
\end{equs}
Note that contributions for $j=1,3$ vanish due parity since the Gaussian measure $\mu_w^{(\kappa)}$ is centred.
Hence, we reduce \eqref{eq:ellis--rosen}
to 
\begin{equs}
 {} &  \eps^{\frac{n-1}{2}} \int \mu_\eps^{(\kappa)}(\dd u) \, F(u) \exp\left\{-\frac{1}{\eps} \V(u) + \frac{1}{2\eps} \|u\|_2^2 \right\} 
 \\
 & \quad = \sum_{i=1}^m \int_{\Ss^{n-1}_{r_i}} \dd w \, \theta^{(\kappa)}(w) \int_{\N_{w}}
 \mu_w^{(\kappa)}(\dd v) \, \sum_{j=0,1} \frac{\eps^{\frac{2j}{2}}}{j!} Q_{2j}(w,v)
 + \error^{(\kappa)}_\eps, \label{eq:ellis--rosen_expl}
\end{equs}
where $\error^{(\kappa)}_\eps \leq C (\sum_{i=1}^m \int_{\Ss^{n-1}_{r_i}} \dd w \, \theta^{(\kappa)}(w) +1)\eps^2$ for a constant 
$C$ which is uniform in $\kappa\geq \kappa_0$. 
As in \cite[proof of Theorem 2]{ER82II}, we further simplify \eqref{eq:ellis--rosen_expl} using the fact that 
$\Ss^{n-1}_{r_i}$ is generated by a group of rotations acting on $\ee_1 = (1,0,\ldots,0)$, 
that is, $\Ss^{n-1}_{r_i}= \{r_iR\ee_1: R\in SO(n)\}$. 
Therefore, for every 
$w\in \Ss^{n-1}_{r_i}$ there exists $R\in SO(n)$ such that $w=r_iR\ee_1$, and it is easy to see that 
the measure $\mu_{r_iR\ee_1}^{(k)}$ has covariance $C_{r_iR\ee_1} = R C_{r_i\ee_1} R^{-1}$.  
Then, using the fact that under $\mu_{r_i\ee_1}^{(\kappa)}$ the law of $Rv$ is given by $\mu_{r_iR\ee_1}^{(\kappa)}$ and the rotation invariance of 
$Q_0(w,v)$ and $Q_2(w,v)$, we can rewrite the integral over $v$ in \eqref{eq:ellis--rosen_expl} as
\begin{equs}
 \int_{\N_{w}} \mu_w^{(\kappa)}(\dd v) \, \sum_{j=0,1} \frac{\eps^{\frac{2j}{2}}}{j!} Q_{2j}(w,v)
 & = \int_{\N_{r_iR\ee_1}}\mu_{r_iR\ee_1}^{(\kappa)}(\dd v) \, \sum_{j=0,1} \frac{\eps^{\frac{2j}{2}}}{j!} Q_{2j}(r_iR\ee_1,v)
 \\
 & = \int_{\N_{r_i\ee_1}}\mu_{r_i\ee_1}^{(\kappa)}(\dd v) \, \sum_{j=0,1} \frac{\eps^{\frac{2j}{2}}}{j!} Q_{2j}(r_iR\ee_1,Rv)
 \\
 & = \int_{\N_{r_i\ee_1}}\mu_{r_i\ee_1}^{(\kappa)}(\dd v) \, \sum_{j=0,1} \frac{\eps^{\frac{2j}{2}}}{j!} Q_{2j}(r_i\ee_1,v).
\end{equs}
Noting in addition that $\theta^{(\kappa)}(w) = \theta^{(\kappa)}(r_iR\ee_1) = \theta^{(\kappa)}(r_i\ee_1)$, \eqref{eq:ellis--rosen_expl}
takes the form
\begin{equs}
 {} & \eps^{\frac{n-1}{2}} \int \mu_\eps^{(\kappa)}(\dd u) \, F(u) \exp\left\{-\frac{1}{\eps} \V(u) + \frac{1}{2\eps} \|u\|_2^2 \right\} 
 \\
 & \quad = \sum_{i=1}^m |\Ss^{n-1}_{r_i}| \, \theta^{(\kappa)}(r_i\ee_1) \int_{\N_{r_i\ee_1}}
 \mu_w^{(\kappa)}(\dd v) \, \sum_{j=0,1} \frac{\eps^{\frac{2j}{2}}}{j!} Q_{2j}(w,v)
 + \error_\eps^{(\kappa)}, 
\end{equs}
which in turn implies \eqref{eq:ellis--rosen_rot_inv}. The proof of \eqref{eq:ellis--rosen_partition} follows similarly, applying Lemma~\ref{lem:ellis--rosen} for $k=1$ and $F\equiv 1$. 
\end{proof}

\section{Estimates on the top Lyapunov exponent}\label{sec:lyap_est}

\subsection{General upper bounds}

In this section we prove the existence of the top Lyapunov exponent and derive a general upper bound based on ergodicity. 

\begin{proposition}\label{prop:lyap_exp_exist} For every $\eps>0$ the top Lyapunov exponents $\lambda_{\mathrm{top}}^{(\eps)}$ and $\tilde \lambda_{\mathrm{top}}^{(\eps)}$
exist.
\end{proposition}

\begin{proof} The proof of the existence of the top Lyapunov exponents $\lambda_{\mathrm{top}}^{(\eps)}$ and $\tilde \lambda_{\mathrm{top}}^{(\eps)}$ is standard in the literature. Let us sketch the argument for  
$\lambda_{\mathrm{top}}^{(\eps)}$. The proof for $\tilde \lambda_{\mathrm{top}}^{(\eps)}$ follows similarly. For simplicity, we let $\kappa=1$. 

We start by proving that for $f\in \supp\mu_\eps^{(\kappa)}$
\begin{equs}
 {} & \int \nu_\eps^{(\kappa)}(\dd f) \, \Exp\sup_{t\leq 1} \log^+\sup_{\|h\|_2\leq 1} \|Du(t;f)h\|_2 < \infty. \label{eq:log_+_1}
\end{equs}
By \eqref{eq:test} and \eqref{eq:hess_V_bnd} we find a deterministic constant $C<\infty$ uniform in $f$, such that 
\begin{equs}
  \frac{1}{2} \partial_t\|Du(t;f)h\|_2^2 + \|\nabla Du(t;f)h\|_2^2 \leq C \| Du(t;f)h\|_2^2.
\end{equs}
Hence, we obtain the uniform in $f$ estimate 
\begin{equs}
 \|Du(t;f)h\|_2^2 + 2 \int_0^t \dd s \, \ee^{2c(t-s)} \|\nabla Du(s;f)h\|_2^2  \leq \ee^{2Ct} \|h\|_2^2, \label{eq:energy_est}
\end{equs}
which implies that
\begin{equs}
  \int \nu_\eps^{(\kappa)}(\dd f) \, \Exp\sup_{t\leq 1} \log^+ \|Du(t;f)\|_{\mathrm{op}} \leq C.
\end{equs}
We define $M:= L^2(\TT;\RR^n)\times \Omega$ and $\Theta:M\to M$ via
$\Theta(f,\omega) = \big(u(1;f, w(\cdot;\omega)), \theta_1 \omega\big)$. We endow $M$ with the product measure $\ \nu_\eps^{(\kappa)}\times \Prob$ and the product $\sigma$-algebra. We consider the mapping $T:M \ni (f,\omega) \mapsto T(f,\omega):=Du(1;f, w(\cdot;\omega))$. Since $h  \mapsto  Du(1;f, w(\cdot;\omega))h$ is a bounded linear operator on $L^2(\TT:\RR^n)$ and by \eqref{eq:log_+_1} we know that $\log^+\|T(\cdot,\cdot)\|_{\mathrm{op}}\in L^1(M; \nu_\eps^{(\kappa)}\times \Prob)$, 
we can apply \cite[Proposition 2.1]{Ru82}
to deduce that there exists a $\Theta$-invariant random variable $\lambda_{\mathrm{top}}^{(\eps)}$ such that 
\begin{equs}
 \lim_{N\nearrow \infty} \frac{1}{N}\log \sup_{\|h\|_2\leq 1} \|Du(N;f, w(\cdot;\omega))h\|_2= \lambda_{\mathrm{top}}^{(\eps)}(f,\omega) \label{eq:to_lyap_discrete} 
\end{equs}
for $\nu_\eps^{(\kappa)}\times \Prob$-almost every $(\omega,f)$. By ergodicity we deduce that $\lambda_{\mathrm{top}}^{(\eps)}(f,\omega)$ is constant for $\nu_\eps^{(\kappa)}\times \Prob$-almost every $(f,\omega)$. As in \cite[Theorem 2.1, Part 3]{Ca85}, we can extend the result to the continuous limit $t\nearrow \infty$, based again on \eqref{eq:log_+_1}.
\end{proof}

\begin{proposition} \label{prop:lyap_exp_formulae} For every $\eps>0$ the following estimates hold
 \begin{equs}
 \lambda_{\mathrm{top}}^{(\eps)} \leq \int \nu_\eps^{(\kappa)}(\dd u) \, \lambda_+(u) \quad \text{and} \quad \tilde \lambda_{\mathrm{top}}^{(\eps)} \leq \frac{1}{\eps}\int \nu_\eps^{(\kappa)}(\dd u) \, \lambda_+(u),
 \end{equs}
 where  $\lambda_+(u)$ is defined in \eqref{eq:lambda_+}. 
\end{proposition}

\begin{proof}  By \eqref{eq:log_linear_bnd} and the definition \eqref{eq:lyap} of $\lambda_{\mathrm{top}}^{(\eps)}$ we obtain the estimate
\begin{equs}
 \lambda_{\mathrm{top}}^{(\eps)} \leq \lim_{t\nearrow \infty} \frac{1}{t} \int_0^t \dd s \, \lambda_+\big(u(s;f, w(\cdot;\omega))\big).
\end{equs}
By \eqref{eq:hess_V_bnd} we know that there exists $C>0$ such that $\lambda_+\big(u(s;f, w(\cdot;\omega))\big)\leq C$. Therefore, by ergodicity
and monotone convergence we get the desired bound for $\lambda_{\mathrm{top}}^{(\eps)}$.
Appealing to \eqref{eq:equal_law}, we also obtain the corresponding bound for $\tilde \lambda_{\mathrm{top}}^{(\eps)}$.
\end{proof}

\subsection{Non-degenerate minima}\label{sec:lyap_est_0}

In this section give an explicit estimate on the top Lyapunov exponent in the case of 
potentials $V$ with finitely many non-degenerate minima as in Assumption~\ref{ass:V}\ref{it:fmm}. We start with an estimate 
which allows to control the infimum over $h$ in the quantity $\lambda_+(u)$ obtained in Proposition~\ref{prop:lyap_exp_formulae} 
due to the non-degeneracy assumption on $V$. The bound can be plugged into Lemma~\ref{lem:ellis--rosen} yielding 
the desired estimate on the top Lyapunov exponent.   

\begin{proposition}\label{prop:h-red_simple} Under Assumptions~\ref{ass:V_general} and \ref{ass:V}\ref{it:fmm}, 
for every $\eta>0$ there exists $\delta_0>0$ such that for every $u\in \C^0(\TT;\RR^n)$, $\kappa>0$, $\delta\leq \delta_0$
and a family $\{\varphi_i\}_{i=1}^m$ of smooth compactly supported functions on $\{\|u-w_i\|_\infty<\delta\}$ with $0\leq \varphi_i\leq 1$ and $\varphi_i(w_i)=1$ we have that
\begin{equs}
 - \lambda_+(u) \geq \sum_{i=1}^m \varphi_i(u) \lambda_{\min}(w_i)
  - C \prod_{i=1}^{m} (1-\varphi_i(u)) (1+\|u\|_\infty^{2(p-1)}) -\eta, 
\end{equs}
where the implicit constant $C$ depends only on $\nabla^2V$.
\end{proposition}

\begin{proof} By the continuity of $\nabla^2V$ there exists $\delta_0>0$ such that for $i=1,\ldots, m$ the sets
\begin{equs} 
 \M_{i,\delta}:=\{u: \,\|u-w_i\|_\infty<\delta\}
\end{equs}
are disjoint and $\|\nabla^2V(u) - \nabla^2V(w_i)\|_{\infty} \leq \eta$ for every $u\in \M_{i,\delta}$, $\delta\leq\delta_0$. Using the definition \eqref{eq:lambda_+} of $\lambda_+(u)$ 
we see that
\begin{equs}
 -\lambda_{+}(u) & = 
 \inf_{\|h\|_{2}=1}\left\{ \kappa\|\nabla h\|_{2}^{2} + \int \dd x \, \nabla^{2} V(u(x)) h(x) \cdot h(x) \right\}
 \\
 & \geq \inf_{\|h\|_{2}=1}\left\{\int \dd x \, \nabla^{2} V(u(x)) h(x) \cdot h(x) \right\},
\end{equs}
where for $h\in L^{2}(\TT)$ and $u\in \M_{i,\delta}$ we have that 
\begin{equs}
  & \int \dd x \, \nabla^{2} V(u(x)) h(x) \cdot h(x)
  \\
  & \quad = 
  \int \dd x \, \nabla^{2} V(w_{i}) h(x) \cdot h(x) + \int \dd x \, \left(\nabla^{2} V(u(x)) - \nabla^{2} V(w_{i})\right) h(x) \cdot h(x) 
  \\
  & \quad \geq (\lambda_{\min}(w_{i}) - \eta) \|h\|_{2}^{2}.
\end{equs}
The final estimate follows since for $u\notin \cup_{i=1}^{m} \M_{i,\delta}$ by the polynomial growth \eqref{eq:pol_growth} of 
$\nabla^2 V$ we have that
\begin{equs}
 \big|\int \dd x \, \nabla^{2} V(u(x)) h(x) \cdot h(x)\big| \leq C (1+\|u\|_{\infty}^{2p-2}). 
\end{equs}
\end{proof}

We now combine Proposition~\ref{prop:lyap_exp_formulae} and Proposition~\ref{prop:h-red_simple} to estimate the top Lyapunov 
exponent $\lambda_{\mathrm{top}}^{(\eps)}$ in the non-degenerate case. 

\begin{theorem}\label{thm:main_1} Under Assumptions~\ref{ass:V_general} and \ref{ass:V}\ref{it:fmm}, for every $n\geq 1$, $\kappa_0>0$ and $\eta>0$ there exists $\eps_0\in [0,1)$ such that for every $\eps\leq \eps_0$ and $\kappa\geq \kappa_0$ we have that
\begin{equs}
 \lambda_{\mathrm{top}}^{\eps} \leq -\frac{ \sum_{i=1}^{m} \lambda_{\min}(w_{i})\theta^{(\kappa)}(w_{i})}{\sum_{i=1}^{m} \theta^{(\kappa)}(w_{i})} +\eta.
\end{equs}
\end{theorem}

\begin{proof} For $\eta>0$ we choose $\delta>0$ and fix a family of smooth compactly supported functions 
$\{\varphi_i\}_{i=1}^m$ supported in $\{\|u-w_i\|_\infty<\delta\}$ as in Proposition~\ref{prop:h-red_simple}. 
By Lemma~\ref{lem:ellis--rosen} for $k=1$ there exists $\eps_0\in(0,1]$ such that for every $\eps\leq \eps_0$,
\begin{equs}
 \int  \mu_{\eps}^{(\kappa)}(\dd u) \, \varphi_i(u) \exp\left\{-\frac{1}{\eps} \V(u) + \frac{1}{2\eps} \|u\|_2^2\right\} 
 =  \theta^{\kappa}(w_i) + \error_\eps^{(\kappa)}, \quad i=1,\ldots,m. \quad \label{eq:ER82I}
\end{equs}
By possibly making $\eps_0$ smaller depending only on $\kappa_0$, by Lemma~\ref{lem:varadhan} we also know that there exists $c\equiv c(\delta)>0$
such that for every $\eps\leq \eps_0$, 
\begin{equs}
 \int  \mu_{\eps}^{(\kappa)}(\dd u) \, \prod_{i=1}^m (1-\varphi_i(u)) (1+\|u\|_\infty^{2(p-1)}) \exp\left\{-\frac{1}{\eps} \V(u) + \frac{1}{2\eps} \|u\|_2^2\right\} 
 = \error_\eps^{(\kappa)}.
\end{equs}
Therefore, we also get that $\Z_\eps^{(\kappa)} = \sum_{i=1}^m \theta^{(\kappa)}(w_i) +\error_\eps^{(\kappa)}$.
Combining Proposition~\ref{prop:lyap_exp_formulae} and Proposition~\ref{prop:h-red_simple}, we obtain the estimate
\begin{equs} 
 \lambda_{\mathrm{top}}^{\eps} \leq \int \nu_\eps^{(\kappa)}(\dd u) \, \lambda_+(u) \leq 
  \frac{- \sum_{i=1}^{m} \lambda_{\min}(w_{i})\theta^{(\kappa)}(w_{i})  
  + \error_\eps^{(\kappa)} + \eta}{\sum_{i=1}^{m} \theta^{(\kappa)}(w_{i})+\error_\eps^{(\kappa)}}.
\end{equs}
Since $\error_\eps^{(\kappa)} \leq C(\sum_{i=1}^m \theta^{(\kappa)}(w_i) +1) \eps$ where $C$ is uniform in $\kappa\geq \kappa_0$
and by Lemma~\ref{lem:det_uniform} we know that $\sup_{\kappa\geq \kappa_0}\big(\sum_{i=1}^m \theta^{(\kappa)}(w_i)\big)^{-1}<\infty$, 
we can choose $\eps_0$ even smaller depending only on $\kappa_0$, such that 
\begin{equs}
\bigg|\frac{\sum_{i=1}^{m} \lambda_{\min}(w_{i})\theta^{(\kappa)}(w_{i})}{\sum_{i=1}^{m} \theta^{(\kappa)}(w_{i})} - \frac{\sum_{i=1}^{m} \lambda_{\min}(w_{i})\theta^{(\kappa)}(w_{i})  
 - \error_\eps^{(\kappa)} - \eta}{\sum_{i=1}^{m} \theta^{(\kappa)}(w_{i})+\error_\eps^{(\kappa)}} \bigg| \leq 4\eta. 
\end{equs}
which implies the desired estimate. 
\end{proof}

\subsection{Degenerate minima}\label{sec:lyap_est_1} 

In this section we estimate the rescaled top Lyapunov exponent for a rotationally-invariant $V$ with degenerate minima as in 
Assumption~\ref{ass:V}\ref{it:mmp}. In analogy to the previous subsection, we estimate the infimum over $h$  
in order to estimate $\lambda_+(u)$ in Proposition~\ref{prop:lyap_exp_formulae} by a suitable expression which 
can be plugged into Lemma~\ref{lem:ellis--rosen_rot_inv}. However, compared to Proposition~\ref{prop:h-red_simple}, 
the estimate is more subtle. The bound should only involve smooth expressions in $u$ and at the same time should 
capture enough information in order for its dominant part to be negative for sufficiently large values of $\kappa$.  

\begin{proposition} \label{prop:h-red} Under Assumptions~\ref{ass:V_general} and \ref{ass:V}\ref{it:mmp}, there exist $C>0$, $p\geq 1$ and $\delta_0>0$ such that for every $u\in \C^0(\TT;\RR^n)$, $\kappa>0$ and $\delta\leq \delta_0$ we have that
\begin{equs}
 -\lambda_+(u) \geq 2\int \dd x \, g'(|u(x)|^2) - \frac{4 C_*^2 \|g'(|u|^2)\|_2^2}{\kappa} - C
 \mathbf{1}_{\left\{\dist_{\|\cdot\|_\infty}(\M,u)\geq \delta\right\}} (1+\|u\|_\infty^{2p}),
\end{equs}
where $C_*$ is the Sobolev constant corresponding to the compact embedding $\dot W^{1,1}(\TT;\RR) \hookrightarrow L^2(\TT;\RR)$ for functions of zero spatial average.   
\end{proposition}

\begin{proof} We first notice that $\nabla^2 V(z) =4 g''(|z|^2) z\otimes z + 2 g'(|z|^2) \mathrm{Id}_{n\times n}$. Using in addition that $g''(|z|^2)> 0$ for any $z\in \M$, by the continuity of $g''$ we find $\delta_0>0$ such that for every $\delta\leq \delta_0$,
\begin{equs}
 -\lambda_+(u) & \geq \min_{\|h\|_2=1} \Big\{\kappa\|\nabla h\|_2^2 + \int \dd x \, 2 g'(|u(x)|^2) |h(x)|^2
 \\
 & \qquad  \qquad+
 \mathbf{1}_{\left\{\dist_{\|\cdot\|_\infty}(\M,u)\geq \delta\right\}} \int \dd x \,
 4 g''(|u(x)|^2) (u(x)\cdot h(x))^2 \Big\}. 
\end{equs}
By the polynomial growth \eqref{eq:pol_growth} we know that 
$g''(|z|^2)\leq C (1+|z|^{2p-2})$, which combined with the Cauchy--Schwarz inequality and the fact that $\|h\|_2=1$ yields the estimate 
\begin{equs}
 \mathbf{1}_{\left\{\dist_{\|\cdot\|_\infty}(u,\M)\geq \delta\right\}} \int \dd x \,
 4 g''(|u(x)|^2) (u(x)\cdot h(x))^2 \geq  
 - C \mathbf{1}_{\left\{\dist_{\|\cdot\|_\infty}(\M,u)\geq \delta\right\}} \left(1+\|u\|_\infty^{2p}\right).
\end{equs}
Hence, we obtain the following lower bound, 
\begin{equs}
 -\lambda_+(u) & \geq \min_{\|h\|_2=1} \left\{\kappa\|\nabla h\|_2^2 + \int \dd x \, 2 g'(|u(x)|^2) |h(x)|^2  \right\}
 \\
 & \quad - C \mathbf{1}_{\left\{\dist_{\|\cdot\|_\infty}(\M,u)\geq \delta\right\}} \left(1+\|u\|_\infty^{2p}\right).
\end{equs}
We now write
\begin{equs}
 {} & \kappa\|\nabla h\|_2^2 + \int \dd x \, 2 g'(|u(x)|^2) |h(x)|^2 
 \\
 & \quad = \left(\kappa\|\nabla h\|_2^2 + \int \dd x \, 2 g'(|u(x)|^2) \right)
 + \int \dd x \, 2 g'(|u(x)|^2) \left(|h(x)|^2-1\right)
\end{equs}
and estimate the error term $\int \dd x \, 2 g'(|u(x)|^2) \left(|h(x)|^2-1\right)$. In doing so, 
we use the embedding $\dot W^{1,1}(\TT;\RR) \hookrightarrow L^2(\TT;\RR)$ for functions of zero spatial average 
and the fact that 
$\int \dd x \, \left(|h(x)|^2-1\right) = 0$, since $\|h\|_2 = 1$, which yield the bound 
\begin{equs}
 \||h|^2-1\|_2 \leq  C_{\dot W^{1,1} \hookrightarrow L^2} \|\nabla|h|^2\|_1 \leq 2 C_{\dot W^{1,1}\hookrightarrow L^2} \|h\|_2 \|\nabla h\|_2 
 \stackrel{\|h\|_2=1}{=} 2 C_{\dot W^{1,1} \hookrightarrow L^2} \|\nabla h\|_2. 
\end{equs}
Together with the Cauchy--Schwarz inequality, this implies the estimate
\begin{equs}
 \left|\int \dd x \, 2 g'(|u(x)|^2)  \left(|h(x)|^2-1\right)\right| & \leq 4 C_{\dot W^{1,1} \hookrightarrow L^2} \|g'(|u(x)|^2)\|_2 \|\nabla h\|_2
 \\
 & \leq \frac{4C_{\dot W^{1,1} \hookrightarrow L^2}^2 \|g'(|u(x)|^2)\|_2^2}{\kappa} + \kappa \|\nabla h\|_2^2,
\end{equs}
which leads to the desired bound. 
\end{proof}

In order to prove the negativity of the rescaled top Lyapunov exponent we apply Lemma~\ref{lem:ellis--rosen_rot_inv}
to the first two terms in the estimate of Proposition~\ref{prop:h-red}. Compared to Section~\ref{sec:lyap_est_0}, the 
final expansions are rather involved, therefore we state them as an independent corollary. In what follows, for $\kappa>0$ 
we let
\begin{equs}
\Z^{(\kappa)} :=\sum_{i=1}^m |\Ss^{n-1}_{r_i}| \theta^{(\kappa)}(r_i e_1). 
\end{equs}

\begin{corollary} \label{cor:asympt} There exists $\eps_0\in(0,1]$ such that for every $\eps\leq \eps_0$ we have that
\begin{equs}
 {} & \frac{1}{\eps} \int \nu_\eps^{(\kappa)}(\dd u) \, 2 \int \dd x \, g'(|u(x)|^2) 
 \\
 & \quad = \frac{1}{\Z^{(\kappa)} +\error_\eps^{(\kappa)}} \sum_{i=1}^m |\Ss^{n-1}_{r_i}| \theta^{(\kappa)}(r_i e_1)
 \int_{\N_{r_i\ee_1}} \mu_{r_i \ee_1}^{(\kappa)}(\dd v) \bigg[4g'''(r_i^2) r_i^2 \| e_1\cdot v\|_2^2 + 2 g''(r_i^2)\|v\|_2^2 
 \\
 & \quad \quad +4g''(r_i^2) (n-1) \lng\ee_1,v\rng^2
 \\
 & \quad \quad - 4g''(r_i^2) r_i \lng e_1,v\rng 
 \left(\frac{4}{3} g'''(r_i^2)r_i^3 \int \dd x \, ( e_1\cdot v(x))^3 +2 g''(r_i^2)r_i \int \dd x \, ( e_1\cdot v(x)) |v(x)|^2 \right)\bigg]
 \\
 & \quad \quad + \frac{\error_\eps^{(\kappa)}}{\Z^{(\kappa)} +\error_\eps^{(\kappa)}}, \label{eq:main_contr}
\end{equs}
and
\begin{equs}
 {} & \frac{1}{\eps} \int \nu_\eps^{(\kappa)}(\dd u) \, 4\|g'(|u|^2)\|_2^2
 \\
 & \quad = \frac{1}{\Z^{(\kappa)} +\error_\eps^{(\kappa)}} 
 \sum_{i=1}^m |\Ss^{n-1}_{r_i}| \theta^{(\kappa)}(r_i e_1)
 16 g''(r_i^2)^2r_i^2 \int_{\N_{r_i\ee_1}} \mu_{r_i \ee_1}^{(\kappa)}(\dd v) \| e_1\cdot v\|_2^2
 \\
 & \quad \quad + \frac{\error_\eps^{(\kappa)}}{\Z^{(\kappa)} +\error_\eps^{(\kappa)}},
 \label{eq:pre-proc_error}
\end{equs}
for some $\error^{(\kappa)}_\eps\in \RR$ with $|\error^{(\kappa)}_\eps| \leq C \big(\Z^{(\kappa)} +1\big) \eps$ for a constant $C$ which is uniform in $\kappa\geq\kappa_0$ for every $\kappa_0>0$. 
\end{corollary}

\begin{proof} We apply Lemma~\ref{lem:ellis--rosen_rot_inv} with $F$ replaced by $F_1(u)= 2\int \dd x g'(|u(x)|^2)$ and 
$F_2(u) = 4\|g'(|u^2|)\|_2^2$. For $w\in \M$ and $v\in \N_w$, in the first case we have that $F_1(w)=0$,
\begin{equs}
 D^{(1)} F_1(w)(v) = 4 g''(|w^2|) \lng w, v \rng, \quad D^{(2)} F_1(w)(v) = 8g'''(|w|^2) \int\dd x \, (w\cdot v(x))^2 + 4 g''(|w|^2) \|v\|_2^2,  
\end{equs}
whereas in the second case we have that $F_2(w) = 0$,
\begin{equs}
 D^{(1)} F_2(w)(v) = 0, \quad D^{(2)} F_2(w)(v)& = 32  g''(|w|^2) \int \dd x \, (w\cdot v(x))^2. 
\end{equs}
Moreover, an explicit calculation shows that 
\begin{equs}
 D^{(3)}\V(w)(v) = 8 g'''(|w|^2) \int \dd x \, (w\cdot v(x))^3 + 12 g''(|w|^2) \int \dd x \, (w\cdot v(x))  |v(x)|^2.
\end{equs}
We write
\begin{equs}
 \frac{1}{\eps} \int \nu_\eps^{(\kappa)}(\dd u) F_j(u) = \frac{1}{\eps\Z_{\eps}^{(\kappa)}} \int \mu_\eps^{(\kappa)}(\dd u) \, F_j(u) 
 \exp\left\{-\frac{1}{\eps} \V(u) + \frac{1}{2\eps} \|u\|_2^2 \right\}.
\end{equs}
Combining with \eqref{eq:ellis--rosen_rot_inv} and \eqref{eq:ellis--rosen_partition}, we obtain \eqref{eq:main_contr} and \eqref{eq:pre-proc_error}. 
\end{proof}

We now combine Proposition~\ref{prop:lyap_exp_formulae}, Proposition~\ref{prop:h-red} and Corollary~\ref{cor:asympt}
to estimate the rescaled top Lyapunov exponent $\tilde \lambda_{\mathrm{top}}^{(\eps)} = \frac{1}{\eps} \lambda_{\mathrm{top}}^{(\eps)}$ in the case of degenerate minima for a rotationally invariant $V$. However, in contrast to the case of 
non-degenerate minima, the next theorem implies negativity for $\kappa$ sufficiently large and $n\geq 3$.

\begin{theorem}\label{thm:main_2} Let
\begin{equs}
 C_{\max} := 2 \left(\frac{1}{4\pi^2}\sum_{k\neq 0}\frac{1}{k^2}\vee C_*^2\right), 
\end{equs}
where $C_*$ is the Sobolev constant of the compact embedding $\dot W^{1,1}(\TT;\RR) \hookrightarrow L^2(\TT;\RR)$ for functions of zero spatial average, $n\geq 2$ and $\kappa_0:=4 \vee_{i=1}^mg''(r_i^2)r_{i}^2 C_{\max}$. Under Assumptions~\ref{ass:V_general} and \ref{ass:V}\ref{it:mmp}, there exists $\eps_0\in(0,1]$ such that for every 
$\eps<\eps_0$ and $\kappa>\kappa_0$ we have
\begin{equs}
 \frac{1}{\eps} \lambda_{\mathrm{top}}^{(\eps)}
 & < -( n-2) \sum_{i=1}^m \frac{|\Ss^{n-1}_{r_i}| \theta^{(\kappa)}(r_i e_1)}{\Z^{(\kappa)}r_i^2}
 + \frac{C_{\max}}{\kappa} \sum_{i=1}^m \frac{|\Ss^{n-1}_{r_i}| \theta^{(\kappa)}(r_i e_1)}{\Z^{(\kappa)}}  4 g''(r_i^2) 
 + \tilde\error_\eps^{(\kappa)}, 
\end{equs}
for some $\tilde\error_\eps^{(\kappa)}\in \RR$ with $|\tilde\error_\eps^{(\kappa)}|\leq C \eps$ for a constant $C$ which is uniform in $\kappa\geq \kappa_0$. 
\end{theorem}

\begin{proof} First note that the Gaussian measure $\mu_{r_i\ee_1}^{(\kappa)}$ on $\N_{r_i\ee_1}$ is centred with covariance matrix 
\begin{equs}
 C_{r_i \ee_1}= \left(-\kappa\Delta \, \mathrm{Id}_{n\times n} + \nabla^{(2)}V(r_i\ee_1)\right)^{-1},
\end{equs}
where 
\begin{equs}
 -\kappa\Delta \, \mathrm{Id}_{n\times n} + \nabla^{(2)}V(r_i\ee_1) 
 & = -\kappa\Delta \, \mathrm{Id}_{n\times n} + 4g''(r_i^2)r_i^2 \ee_1\otimes \ee_1 
 \\
 & = \left( 
 \begin{matrix} 
    -\kappa\Delta + 4g''(r_i^2)r_i^2 & 0 \\
    0 & -\kappa \Delta \, \mathrm{Id}_{(n-1)\times(n-1)}
 \end{matrix}
 \right).
\end{equs}
To simplify the notation below we write $\Exp_{\mu_{r_i\ee_1}^{(\kappa)}}$ for the expectation 
$\int_{\N_{r_i\ee_1}} \mu_{r_i\ee_1}^{(\kappa)}(\dd v)$. Let $\delta>0$ be as in the statement of Proposition~\ref{prop:h-red}.
By Lemma~\ref{lem:varadhan} and \eqref{eq:ellis--rosen_partition} there exists $\eps_0\in(0,1]$ such that for every $\eps\leq \eps_0$, 
\begin{equs}
 \frac{1}{\eps}\int_{{\left\{\dist_{\|\cdot\|_\infty}(\M,u)\geq \delta\right\}}} \nu_\eps^{(\kappa)}(\dd u) (1+\|u\|_\infty^{2p})
 \exp\left\{-\frac{1}{\eps} \V(u) + \frac{1}{2\eps} \|u\|_2^2\right\}
  = \frac{\error^{(\kappa)}_\eps}{\Z^{(\kappa)} + \error^{(\kappa)}_\eps}.
\end{equs}
Combining with Proposition~\ref{prop:lyap_exp_formulae} and Proposition~\ref{prop:h-red}  we get 
\begin{equs}
\frac{1}{\eps} \lambda_{\mathrm{top}}^{(\eps)} \leq \frac{1}{\eps} \int \nu_\eps^{(\kappa)}(\dd u) \, \lambda_+(u)
& \leq 
- \frac{1}{\eps} \int \nu_\eps^{(\kappa)}(\dd u) \, 2 \int \dd x\,  g'(|u(x)|^2) 
\\
& \quad + 
\frac{4 C_*^2 }{\kappa} \int \nu_\eps^{(\kappa)}(\dd u) \|g'(|u|^2)\|_2^2 + \frac{\error^{(\kappa)}_\eps}{\Z^{(\kappa)} + \error^{(\kappa)}_\eps}. 
\end{equs}
Then, by Corollary~\ref{cor:asympt} we infer that
\begin{equs}
 {} & \frac{1}{\eps} \lambda_{\mathrm{top}}^{(\eps)}
 \\
 & \quad \leq 
 - \sum_{i=1}^m \frac{|\Ss^{n-1}_{r_i}| \theta^{(\kappa)}(r_i e_1)}{\Z^{(\kappa)}+\error_\eps^{(\kappa)}}
 \Exp_{\mu_{r_i\ee_1}^{(\kappa)}} \bigg(4g'''(r_i^2)r_i^2 \| e_1\cdot v\|_2^2 + 2 g''(r_i^2)\|v\|_2^2 
 +4g''(r_i^2) (n-1) \lng \ee_1,v\rng^2
 \\
 & \quad \quad - \frac{16}{3} g'''(r_i^2)g''(r_i^2)r_i^4 \lng e_1,v\rng \int \dd x \, (e_1\cdot v(x))^3 
 -8 g''(r_i^2)^2 r_i^2 \lng e_1,v\rng \int \dd x \, (e_1\cdot v(x)) |v(x)|^2
 \\
 & \quad \quad -\frac{C_*^2}{\kappa} 16 g''(r_i^2)^2r_i^2 \| e_1\cdot v\|_2^2\bigg) +\frac{\error^{(\kappa)}_\eps}{\Z^{(\kappa)} + \error^{(\kappa)}_\eps}
\\
 & \quad = : 
 - \sum_{i=1}^m \frac{|\Ss^{n-1}_{r_i}| \theta^{(\kappa)}(r_i e_1)}{\Z^{(\kappa)}+ \error^{(\kappa)}_\eps} E_i + \frac{\error^{(\kappa)}_\eps}{\Z^{(\kappa)} + \error^{(\kappa)}_\eps}. 
 \label{eq:lower_bnd}
\end{equs}
Note that $v$ can be decomposed into $v = \lng \ee_1, v \rng \, \ee_1 + v^{\perp}$, where $\int \dd x \, v^\perp(x) = 0$. 
Using this decomposition, we note that
\begin{equs}
 \|\ee_1\cdot v\|_2^2 & = \lng \ee_1, v \rng^2 + 2\lng \ee_1, v \rng \lng \ee_1, v^{\perp} \rng +  \|v^{\perp}\|_2^2 = \lng \ee_1, v \rng^2 + \|\ee_1 \cdot v^{\perp}\|_2^2, 
 \\
 \|v\|_2^2 & = \lng \ee_1, v \rng^2 + 2 \lng \ee_1, v \rng \lng \ee_1, v^\perp\rng   \|v^\perp\|_2^2 = \lng \ee_1, v \rng^2 + \|v^\perp\|_2^2,
 \\
 \lng e_1,v\rng \int \dd x \, (e_1\cdot v(x))^3 & =  \lng e_1,v\rng \int \dd x \, \big(\lng \ee_1, v \rng^3
 + 3 \lng \ee_1, v \rng^2  (\ee_1\cdot v^\perp(x)) 
 \\
 & \quad \quad \quad \quad \quad \quad + 3 \lng \ee_1, v \rng(\ee_1\cdot v^\perp(x))^2 + (\ee_1 \cdot v^\perp(x))^3\big)
 \\
 & = \lng \ee_1, v \rng^4 + 3 \lng \ee_1, v \rng^2 \|\ee_1\cdot v^\perp\|_2^2
 + \lng \ee_1, v \rng \int \dd x \, (\ee_1\cdot v^\perp(x))^3,
 \\
 \lng \ee_1, v\rng \int \dd x \, (\ee_1\cdot v(x)) \, |v(x)|^2 & =  \lng \ee_1, v\rng \int \dd x \, \big(\lng \ee_1, v \rng + \ee_1\cdot v^\perp(x)\big) 
 \\
 & \quad \quad \quad \quad \times 
 \big(\lng \ee_1, v \rng^2 + 2 \lng \ee_1, v \rng (\ee_1\cdot v^\perp(x)) + (\ee_1\cdot v^\perp(x))^2 \big) 
 \\
 & = \lng \ee_1, v \rng^4 + \lng \ee_1, v \rng^2 \|v^\perp\|_2^2
 + 2 \lng \ee_1, v \rng^2 \|\ee_1\cdot v^\perp\|_2^2 
 \\
 & \quad + \lng \ee_1, v \rng \int \dd x \, (\ee_1\cdot v^\perp(x)) |v^\perp(x)|^2. 
\end{equs}
Using the independence of $\lng \ee_1, v \rng$ and $v^\perp$ and the fact that $\Exp_{\mu_{r_i\ee_1}^{(\kappa)}} \lng \ee_1, v \rng = 0$, 
we obtain that
\begin{equs}
 \Exp_{\mu_{r_i\ee_1}^{(\kappa)}} \left(\lng \ee_1, v \rng^2 \|\ee_1\cdot v^{\perp}\|_2^2\right) 
 & = \Exp_{\mu_{r_i\ee_1}^{(\kappa)}} \lng \ee_1, v \rng^2 
 \Exp_{\mu_{r_i\ee_1}^{(\kappa)}} \|\ee_1\cdot v^{\perp}\|_2^2, 
 \\
 \Exp_{\mu_{r_i\ee_1}^{(\kappa)}}\left(\lng \ee_1, v \rng \int \dd x \, (\ee_1\cdot v^\perp(x))^3\right)
 & = \Exp_{\mu_{r_i\ee_1}^{(\kappa)}}\lng \ee_1, v \rng \Exp_{\mu_{r_i\ee_1}^{(\kappa)}}\int \dd x \, (\ee_1\cdot v^\perp(x))^3
 =0
 \\
 \Exp_{\mu_{r_i\ee_1}^{(\kappa)}} \left(\lng \ee_1, v \rng^2 \|v^{\perp}\|_2^2 \right) &  = 
 \Exp_{\mu_{r_i\ee_1}^{(\kappa)}}\lng \ee_1, v \rng^2 \Exp_{\mu_{r_i\ee_1}^{(\kappa)}} \|v^{\perp}\|_2^2,
 \\
 \Exp_{\mu_{r_i\ee_1}^{(\kappa)}} \left(\lng \ee_1, v \rng \int \dd x \, (\ee_1\cdot v^\perp(x)) |v^\perp(x)|^2\right)
 & = \Exp_{\mu_{r_i\ee_1}^{(\kappa)}} \lng \ee_1, v \rng 
 \Exp_{\mu_{r_i\ee_1}^{(\kappa)}} \int \dd x \, (\ee_1\cdot v^\perp(x)) |v^\perp(x)|^2 = 0.
\end{equs}
Hence, we retrieve the following identity, 
\begin{equs}
 {}  E_i & = 4g'''(r_i^2)r_i^2 \Exp_{\mu_{r_i\ee_1}^{(\kappa)}}  \lng \ee_1,v\rng^2 
 + 4g'''(r_i^2)r_i^2 \Exp_{\mu_{r_i\ee_1}^{(\kappa)}}\| e_1\cdot v^\perp\|_2^2
 + 2 g''(r_i^2) \Exp_{\mu_{r_i\ee_1}^{(\kappa)}} \lng \ee_1,v\rng^2
 \\
 & \quad  
 + 2 g''(r_i^2) \Exp_{\mu_{r_i\ee_1}^{(\kappa)}}\|v^\perp\|_2^2
 + 4g''(r_i^2) (d-1) \Exp_{\mu_{r_i\ee_1}^{(\kappa)}} \lng \ee_1,v\rng^2 
 - \frac{16}{3} g'''(r_i^2)g''(r_i^2)r_i^4 \Exp_{\mu_{r_i\ee_1}^{(\kappa)}} \lng \ee_1,v\rng^4 
 \\
 & \quad
  - 16 g'''(r_i^2)g''(r_i^2)r_i^4 \Exp_{\mu_{r_i\ee_1}^{(\kappa)}} \lng \ee_1,v\rng^2 \Exp_{\mu_{r_i\ee_1}^{(\kappa)}}\| e_1\cdot v^\perp\|_2^2
 - 8 g''(r_i^2)^2 r_i^2 \Exp_{\mu_{r_i\ee_1}^{(\kappa)}} \lng \ee_1,v\rng^4
 \\ 
 & \quad
 - 8 g''(r_i^2)^2 r_i^2 \Exp_{\mu_{r_i\ee_1}^{(\kappa)}} \lng \ee_1,v\rng^2 \Exp_{\mu_{r_i\ee_1}^{(\kappa)}}\|v^\perp\|_2^2
- 16 g''(r_i^2)^2 r_i^2 \Exp_{\mu_{r_i\ee_1}^{(\kappa)}} \lng \ee_1,v\rng^2 \Exp_{\mu_{r_i\ee_1}^{(\kappa)}}\| e_1\cdot v^\perp\|_2^2
  \\
 & \quad 
 -\frac{C_*^2}{\kappa} 16 g''(r_i^2)^2 r_i^2 \Exp_{\mu_{r_i\ee_1}^{(\kappa)}} \lng\ee_1,v\rng^2 - \frac{C_*^2}{\kappa} 16g''(r_i^2)^2r_i^2
 \Exp_{\mu_{r_i\ee_1}^{(\kappa)}} \|\ee_1\cdot v^{\perp}\|_2^2. 
 \label{eq:expl_1}
\end{equs}
Under $\mu_{r_i\ee_1}^{(\kappa)}$, $\lng \ee_1, v \rng$ is a centred Gaussian scalar-valued random 
variable with variance 
\begin{equs}
 \Exp_{\mu_{\ee_1}^{(\kappa)}} \lng \ee_1, v \rng^2 =\frac{1}{4g''(r_i^2)r_i^2}
\end{equs}
and forth moment 
\begin{equs}
 \Exp_{\mu_{\ee_1}^{(\kappa)}} \lng \ee_1, v \rng^4 =\frac{3}{16g''(r_i)^2r_i^4}.
\end{equs}
Therefore, \eqref{eq:expl_1} reduces to
\begin{equs}
 {} & E_i
 = \frac{n-2}{r_i^2} - \left(2g''(r_i^2) + \frac{C_*^2}{\kappa} 16g''(r_i^2)^2r_i^2\right) 
 \Exp_{\mu_{r_i\ee_1}^{(\kappa)}} \|\ee_1\cdot v^{\perp}\|_2^2 - \frac{C_*^2}{\kappa} 4g''(r_i^2). 
 \label{eq:expl_2}
\end{equs}
To compute $\Exp_{\mu_{\ee_1}^{(\kappa)}} \|\ee_1\cdot v^\perp\|_2^2$, we write 
\begin{equs}
 \Exp_{\mu_{r_i\ee_1}^{(\kappa)}} \|\ee_1\cdot v^\perp\|_2^2 = \sum_{k\neq 0 } \Exp_{\mu_{r_i\ee_1}^{(\kappa)}}
 \lng \ee_1 \cdot v, \ee^{2\pi \ii k \cdot} \rng^2 
 = \sum_{k\neq 0} \frac{1}{\kappa (2\pi k)^2 + 4g''(r_i^2) r_i^2} \leq \frac{1}{4\pi^2\kappa} \sum_{k\neq 0} \frac{1}{k^2}.
\end{equs}
We also notice that $\frac{C_*^2}{\kappa} 16g''(r_i^2)^2r_i^2 \leq 2g''(r_i^2)$ for $\kappa \geq \kappa_0\geq 4g''(r_i^2)r_i^2 C_{\max}$ and, 
in particular, $E_i$ is uniformly bounded in $\kappa \geq \kappa_0$. 
Combining with \eqref{eq:lower_bnd}, \eqref{eq:expl_2} and Lemma~\ref{lem:det_uniform}, which implies that 
$(\Z^{(\kappa)})^{-1}$ is uniformly bounded in $\kappa\geq \kappa_0$, 
by possibly making $\eps_0$ smaller depending on $\kappa_0$, for every $\eps\leq \eps_0$ we get that 
\begin{equs}
 {} \frac{1}{\eps} \lambda_{\mathrm{top}}^{(\eps)}
 & < \sum_{i=1}^m \frac{|\Ss^{n-1}_{r_i}| \theta^{(\kappa)}(r_i e_1)}{\Z^{(\kappa)}}
 \left(-\frac{(n-2)}{r_i^2} + \frac{1}{\kappa} \left(\frac{1}{4\pi^2} \sum_{k\neq 0} \frac{1}{k^2} + C_*^2\right) 4g''(r_i^2)\right) 
 + \tilde \error_\eps^{(\kappa)}
 \\
 &\leq \sum_{i=1}^m \frac{|\Ss^{n-1}_{r_i}| \theta^{(\kappa)}(r_i e_1)}{\Z^{(\kappa)}}
 \left(-\frac{(n-2)}{r_i^2} + \frac{1}{\kappa} C_{\max} 4g''(r_i^2)\right) + \tilde \error_\eps^{(\kappa)},
\end{equs}
for some $\tilde \error_\eps^{(\kappa)}$ as in the statement of the theorem which implies the desired estimate. 
\end{proof}

\section{Synchronisation by noise}\label{sec:synchr}  

In this section, we prove synchronisation by noise for \eqref{eq:rds_intro}, see Theorem~\ref{thm:synchr}. The proof is based
on the general framework developed in \cite{FGS14} for synchronisation by noise for white noise RDS, see Definition~\ref{def:rds}. 

As we explain in Appendix~\ref{app:RDS}, the family $\{u(t;\cdot,w(\cdot,\omega))\}_{t\geq 0,\omega\in \Omega}$, where  $u(t;f,w(\cdot,\omega))$ is the solution to \eqref{eq:rds_intro} with initial data $f$, see Section~\ref{sec:ass_framework} for definitions, gives rise to a white noise RDS on $\C^0(\TT;\RR^n)$ 
which extends to $L^2(\TT;\RR^n)$ and we write $(\Omega, \mathcal{F}, \Prob, \{\theta_t\}_{t\in \RR})$ for the corresponding
metric dynamical system. To ease the notation, we write $u(t;\cdot,\omega)$ instead of $u(t;\cdot,w(\cdot,\omega))$. 

As already 
mentioned in the introduction, synchronisation by noise refers to the existence of a weak pull-back attractor which is defined as follows.  

\begin{definition}\label{def:pull_back_attr} A family $\{A(\omega)\}_{\omega\in \Omega}$ of non-empty subsets of $L^2(\TT;\RR^n)$
is called a random compact set in $L^2(\TT;\RR^n)$ if $A(\omega)$ is compact for $\Prob$-almost every $\omega$ and the mapping 
$\omega\mapsto \dist_{\|\cdot\|_2}(f, A(\omega))$ is measurable with respect to $\mathcal{F}$ for every $f\in L^2(\TT;\RR^n)$. 

A random compact set $A$ in $L^2(\TT;\RR^n)$ is called a weak pull-back attractor for a white noise RDS $\{\varphi(t,s,\omega)\}_{t\geq0, \omega\in \Omega}$ on $L^2(\TT;\RR^n)$ if it is $\varphi$-invariant, that is, $\varphi(t,\omega)A(\omega) = A(\theta_t\omega)$ for $\Prob$-almost every $\omega\in \Omega$, and for every compact $K\subset L^2(\TT;\RR^n)$ satisfies 
\begin{equs}
 \lim_{t\nearrow +\infty} \sup_{f\in K} \dist_{\|\cdot\|_2} \big(\varphi(t,\theta_{-t}\omega)f, \mathcal{A}(\omega)\big) = 0 \quad \text{in probability}. \label{eq:weak_attr_def}
\end{equs}
\end{definition}

Note that in the case $\varphi(t,\omega) = u(t;\cdot,w(\cdot,\omega)$, \eqref{eq:weak_attr_def} becomes  
\eqref{eq:weak_attr_intro}. The existence of a weak pull-back attractor $A$ in $L^2(\TT;\RR^n)$ for $\{u(t;\cdot, \omega)\}_{t\geq 0, \omega\in \Omega}$ follows as in the scalar-valued case, see for example \cite[Theorem~4.3]{De98}. In the next proposition,  following an argument similar to 
\cite[Theorem~3.1]{Ge13ii}, we prove that every random compact set for
$\{u(t;\cdot, \omega)\}_{t\geq 0, \omega\in \Omega}$ in $L^2(\TT;\RR^n)$ is compact in $\C^0(\TT;\RR^n)$, which in turn implies
compactness of the weak pull-back attractor $A$ in $\C^0(\TT;\RR^n)$. 

\begin{proposition}[Compactness in $\C^0(\TT;\RR^n)$] \label{prop:weak_attr_comp} Let $A$ be a random bounded subset in $L^2(\TT;\RR^n)$. For every $t>0$ the set $u(t;A)$ is compact in $\C^0(\TT;\RR^n)$. In particular, every weak pull-back attractor $A$ in $L^2(\TT;\RR^n)$ is compact in $\C^0(\TT;\RR^n)$.
\end{proposition}

\begin{proof} Recall that $u(\cdot;f,\omega)=w(\cdot;\omega)+v(\cdot;f,\omega)$ where $w(\cdot;\omega)$ and $v(\cdot;f,\omega)$ solve \eqref{eq:stoch_conv} and \eqref{eq:remainder}. By a classical 
energy estimate, \eqref{eq:coerc_grad_V} and the polynomial growth \eqref{eq:pol_growth} of $b$, we know that there exists $C_1>0$ such that
\begin{equs}
 \frac{1}{2} \partial_t\|v(t;f,\omega)\|_2^2 + \|\nabla v(t;f,\omega)\|_2^2 & = \lng b(v(t;f,\omega)+w(t;\omega)), v(t;f,\omega) \rng 
 \\
 & \leq -C_1 \|v(t;f,\omega)\|_{2p}^{2p} + \|w(t;\omega)\|_{2p}^{2p} + 1 
\end{equs}
which in turn implies that
\begin{equs}
 \int_0^t \dd s \, \left(\|v(s;f,\omega)\|_{2p}^{2p} + \|\nabla v(s;f,\omega)\|_2^2\right) \lesssim \|f\|_2^2 + 
 \int_0^t \dd s \, \|w(s;\omega)\|_{2p}^{2p} + t, \label{eq:energy_estimate}
\end{equs}
hence $v(\cdot;f,\omega)\in L^{2p}\big(0,t;L^2(\TT;\RR^n)\big)\cap L^{2}\big(0,t;H^1(\TT;\RR^n)\big)$ for every $t>0$ and $f\in L^2(\TT;\RR^n)$ which in turn implies that
$v(\cdot;f,\omega)\in L^{2}\big(0,t;\C^0(\TT;\RR^n)\big)$ due to the embedding $H^1(\TT;\RR^n)\hookrightarrow \C^0(\TT;\RR^n)$. 
In addition, using the embedding $L^1(\TT;\RR^n)\hookrightarrow H^{-1}(\TT;\RR^n)$ and the polynomial growth \eqref{eq:pol_growth} of $b$, we have for every $q\geq 1$, 
\begin{equs}
 {} & \int_0^t \dd s \, \|\partial_t v(s;f,\omega)\|_{H^{-1}}^q 
 \\
 & \quad  \lesssim 
 \int_0^t \dd s \, \|\Delta v(s;f,\omega)\|_{H^{-1}}^q + \int_0^t \dd s \, \|b(v(s;f,\omega) + w(s;\omega))\|_{1}^q 
 \\
 & \quad \lesssim \int_0^t \dd s \, \|v(s;f,\omega)\|_{H^1}^q + \int_0^t \dd s \, \|v(s;f,\omega)\|_{2p-1}^{q(2p-1)}  
 + \int_0^t \dd s \, \|w(s;\omega)\|_{2p-1}^{q(2p-1)} +t
 \\
 & \quad \lesssim 
 \int_0^t \dd s \, \|v(s;f,\omega)\|_{H^1}^q + \int_0^t \dd s \, \|v(s;f,\omega)\|_{2p}^{q(2p-1)}  
 + \int_0^t \dd s \, \|w(s;\omega)\|_{2p}^{q(2p-1)} +t.
\end{equs}
Choosing $q=\frac{2p}{2p-1} \in(1,2]$ and using \eqref{eq:energy_estimate} we get 
\begin{equs}
 \int_0^t \dd s \, \|\partial_t v(s;f,\omega)\|_{H^{-1}}^{\frac{2p}{2p-1}} & \lesssim 
 \|f\|_2^2 +  \int_0^t \dd s \, \|w(s;\omega)\|_{2p}^{2p} + t,
\end{equs}
therefore $\partial_t v(\cdot;f,\omega) \in L^{q}\big(0,t;H^{-1}(\TT;\RR^n)\big)$ for every $t>0$ and $f\in L^2(\TT;\RR^n)$. 
Let $A$ be a random bounded set in $L^2(\TT;\RR^n)$. By \cite[Corollary~6]{Si87} applied to the triple 
$L^{2}\big(0,t;\C^0(\TT;\RR^n)\big)$, $L^{2}\big(0,t;H^{1}(\TT;\RR^n)\big)$ and $L^q\big(0,t;H^{-1}(\TT;\RR^n)\big)$, 
the set $\{[0,t]\ni s\mapsto v(s;f,\omega): f\in A(\omega)\}$ is compact in 
$L^{r}(0,t;\C^0(\TT;\RR^n))$ for every $r<2$ and $t>0$. 

We now use the compactness of $\{[0,t]\ni s\mapsto v(s;f,\omega): f\in A(\omega)\}$ in $L^{r}(0,t;\C^0(\TT;\RR^n))$ to deduce
sequential compactness of $u(t;A(\omega),\omega)$ in $\C^0(\TT;\RR^n)$. Given a sequence $\{f_n\}_{n\geq 1}$ in $A(\omega)$ we have that $u(t;f_n,\omega)=v(t;f_n,\omega)+w(t;\omega)$ converges
in $L^r(0,t;\C^0(\TT;\RR^n))$ along a subsequence, again denoted by $f_n$. Hence by choosing a further subsequence
$u(r;f_n,\omega)$ converges in $\C^0(\TT;\RR^n)$ for almost every $s\in [0,t]$. Choosing such an $s\in (0,t)$ and using 
\eqref{eq:flow_property} in the form $u(t;f_n,\omega) = u_s(t-s;u(s;f_n,\omega),w_s(\cdot;\omega))$, by continuity in the
initial data we know that $u_s(t-s;u(r;f_n,\omega),\omega)$, and thus $u(t;f_n,\omega)$, converges in $\C^0(\TT;\RR^n)$.    

Finally, let $A$ be a weak pull-back attractor in $L^2(\TT;\RR^n)$. Since $A$ is a random compact set in $L^2(\TT;\RR^n)$ 
such that $A(\theta_t\omega) = u(t;A,\omega)$, it follows that $A$ is compact in $\C^0(\TT;\RR^n)$ $\Prob$-almost 
surely.  
\end{proof}

In order to prove Theorem~\ref{thm:synchr} we follow \cite{FGS14}. More precisely, under Assumption~\ref{ass:V_general},
in Proposition~\ref{prop:contr_large_sets} we show the contraction on large sets and in Proposition~\ref{prop:swift_trans}
swift transitivity. Then, in Proposition~\ref{prop:local_stable_man}, under the additional assumption that $V$ is a perturbation of 
$u\mapsto \frac{1}{4} |u|^4$ with bounded third derivatives, see Assumption~\ref{ass:V_quartic}, we deduce local asymptotic stability in the case of negative top Lyapunov exponent.
Lastly, we conclude with the proof of Theorem~\ref{thm:synchr}.

Let us point out that the reason we define the distance in \eqref{eq:weak_attr_def}
based on the $L^2$-norm is due to the fact that the local asymptotic stability in Proposition~\ref{prop:local_stable_man}
holds in the $L^2$-framework. However, in order to estimate the non-linearity in Proposition~\ref{prop:contr_large_sets} and Proposition~\ref{prop:swift_trans}, it is more convenient to work in the $\C^0$-framework, since controlling just the 
$L^2$-norm would not be sufficient for the arguments presented there.    

In the rest of this section, to simplify the notation we let $\kappa=\eps=1$. However, all the results appearing below trivially
extend to the case of arbitrary $\kappa>0$ and $\eps\in(0,1]$.

\paragraph{Contraction on large sets and swift transitivity} In what follows we make use of control arguments for 
\eqref{eq:reaction-diffusion_system}. Given a deterministic 
control $\tilde w\in \C^0([0,T]\times \TT;\RR^n)$ and $f\in \C^0(\TT;\RR^n)$ we write $u(\cdot;f,\tilde w)$ for the unique mild
solution to 
\begin{equs}
 & (\partial_t-\Delta) u = b(u) + (\partial_t-\Delta) \tilde w(t), \quad u|_{t=0} = f.
\end{equs}
As we argue in Appendix~\ref{app:RDS} for the solutions to \eqref{eq:reaction-diffusion_system}, the mapping 
$\tilde w \mapsto u(\cdot;f,\tilde w)$ is continuous uniformly for $f$ in a bounded subset of $\C^0(\TT;\RR^n)$, which 
we use in the proof of Proposition~\ref{prop:contr_large_sets} below. 

\begin{proposition}[Contraction on large sets]\label{prop:contr_large_sets} Under Assumption~\ref{ass:V_general}, for every $R>0$ there exist $f\in \RR^n$ and $t_0>0$ such that
\begin{equs}
 \Prob\left(\mathrm{diam}_{\|\cdot\|_\infty}u\big(t_0;\mathrm{B}_{\|\cdot\|_\infty}(f;R),w\big)\leq \frac{R}{4}\right) >0.  
\end{equs}
\end{proposition}

\begin{proof} We first argue that \eqref{eq:hess_V_bnd} implies that 
for every $R>0$ there exists $M>0$ such that for every $|u|\geq 2 R$ and $|h|\leq R$
\begin{equs}
 \big(\nabla V(u+h) - \nabla V(u)\big)\cdot h \geq M |h|^2. \label{eq:dissip_large_scale}
\end{equs}
In order to prove \eqref{eq:dissip_large_scale}, we note that by the mean value theorem
\begin{equs}
\big(\nabla V(u+h) - \nabla V(u)\big)\cdot h = \int_0^1 \dd \lambda \, \nabla^2 V(u + \lambda h) h \cdot h, 
\end{equs}
therefore the desired estimate follows from \eqref{eq:hess_V_bnd} noting that $|u + \lambda h|\geq |u| - |h| \geq R$.  

For $R>0$ let $M$ be as in \eqref{eq:dissip_large_scale} and  
fix a constant $f\in \RR^n$ such that $|f|\geq 2R$. We consider the deterministic control $w_*(t) := - t b(f)$ and
the solution $u(\cdot;f,w_*)$ to the corresponding control problem. By uniqueness, we know that $u(t;f,w_*) = f$ for all $t>0$.
We also consider $g\in \mathrm{B}_{\|\cdot\|_\infty}(f;R)$ and the corresponding solution $u(t;g,w_*)$.
We let $\tilde v(t) := u(t;g,w_*) - u(t;f,w_*) = u(t;g,w_*) - f$. By \cite[Appendix~D, Proposition~D.4]{DPZ92} we know that
\begin{equs}
 \lim_{\delta\nearrow 0} \frac{\|\tilde v(t + \delta)\|_\infty - \|\tilde v(t)\|_\infty}{\delta} =: \partial_t^-\|\tilde v(t)\|_\infty 
 \leq \inf_{\mu\in \partial\|\tilde v(t)\|_\infty} \int \mu(\dd x) \, \partial_t \tilde v(t,x),
\end{equs}
where $\partial\|\tilde v(t)\|_\infty:= \{\mu\in \mathcal{M}(\TT): \, \int \mu(\dd x) \, \tilde v(t,x)=\|\tilde v(t)\|_\infty, \, \|\mu\|_{\mathrm{op}}=1\}$, see \cite[Appendix~D, Eq.~(D.1)]{DPZ92}. Here $\mathcal{M}(\TT)$ stands for the dual of $C^0(\TT;\RR^n)$
and $\|\mu\|_{\mathrm{op}}$ stands for the operator norm of $\mu$.  
Assuming that $\|\tilde v(t)\|_\infty$ does not vanish and choosing $y\in \TT$ such that $\tilde v(t,y)=\|\tilde v(t)\|_\infty$, which is possible due to the continuity 
of $\tilde v(t)$ in space, we 
notice that $\mu_y:=\delta_y \frac{\tilde v(t,y)}{|\tilde v(t;v)|}\in \partial\|\tilde v(t)\|_\infty$ and
\begin{equs}
 \int \mu_y(\dd x) \, \partial_t \tilde v(t,x) = \partial_t \tilde v(t,y) \cdot \frac{\tilde v(t,y)}{|\tilde v(t,y)|}. 
\end{equs}
Therefore, we have shown that 
\begin{equs}
 \partial_t^-\|\tilde v(t)\|_\infty \leq \partial_t v(t,y) \cdot \frac{v(t,y)}{|v(t,y)|}.
\end{equs}
Noting that $(\partial_t - \Delta)\tilde v = b(u(\cdot;g,w_*)) - b(u(\cdot;f,w_*))$, by Lemma~\ref{lem:sup_energy_est} we obtain the estimate
\begin{equs}
 \partial_t^-\|\tilde v(t)\|_\infty \leq \big(b(u(t,y;g,w_*)) - b(u(t,y;f,w_*))\big) \frac{\tilde v(t,y)}{|\tilde v(t,y)|}.
\end{equs}
Letting $\tilde \tau := \inf\{t>0 : \, \|\tilde v(t)\|_\infty >R\}>0$ and using \eqref{eq:dissip_large_scale} we have for $t\leq \tilde \tau$
\begin{equs}
 \partial_t^-\|\tilde v(t)\|_\infty \leq - M |\tilde v(t,y)| = - M \|\tilde v(t)\|_\infty. 
\end{equs}
Therefore by Lemma~\ref{lem:left_gronwall} applied to $\phi(t) := \ee^{Mt}\|\tilde v(t)\|_\infty$ we obtain for $t\leq \tilde\tau$
\begin{equs}
 \|\tilde v(t)\|_\infty \leq \ee^{-Mt} \|\tilde v(0)\|_\infty < R. 
\end{equs}
Hence, $\tilde \tau=\infty$ and there exists $t_0>0$ such that $\|\tilde v(t_0)\|_\infty \leq \frac{R}{16}$. Since 
of $\tilde w \mapsto u(\cdot;g,\tilde w)$ is continuous uniformly for $g\in \mathrm{B}_{\|\cdot\|_\infty}(f;R)$,
there exists $\delta>0$ such that for every $\tilde w\in \mathrm{B}_{\C^0([0,t_0]\times\TT;\RR^n)}(w_*;\delta)$, 
\begin{equs}
 \sup_{g\in \mathrm{B}_{\|\cdot\|_\infty}(f;R)} \|u(t_0;g;\tilde w) - u(t_0;g;w_*)\|_\infty \leq \frac{R}{8},
\end{equs}
which in turn implies that
\begin{equs}
 \Prob\left(\mathrm{diam}_{\|\cdot\|_\infty}u\big(t;\mathrm{B}_{\|\cdot\|_\infty}(f;R),w\big)\leq \frac{R}{4}\right) 
 \geq \Prob\left(w \in \mathrm{B}_{\C^0([0,t_0]\times\TT;\RR^n)}(w_*;\delta)\right) >0. 
\end{equs}
\end{proof}

\begin{proposition}[Swift transitivity] \label{prop:swift_trans} Under Assumption~\ref{ass:V_general}, for every $f,g\in \C^0(\TT;\RR^n)$ and $R>r>0$ there exists $t_0>0$ such that 
\begin{equs}
 \Prob\left(u\big(t_0;B_{\|\cdot\|_\infty}(f;r),w\big)\in B_{\|\cdot\|_\infty}(g;R)\right)>0.
\end{equs}
\end{proposition}

\begin{proof} Since $\C^\infty(\TT;\RR^n)$ is dense in $\C^0(\TT;\RR^n)$, it suffices to prove the statement for $g\in \C^\infty(\TT;\RR^n)$. Fix $f\in \C^0(\TT;\RR^n)$, $g\in \C^\infty(\TT;\RR^n)$ and $t_0>0$ and let
\begin{equs}
 u_*(t;f) := \ee^{t\Delta} f + \frac{t}{t_0} (g-\ee^{t_0 \Delta}f). 
\end{equs}
Consider the control $w$ given by 
\begin{equs}
 (\partial_t - \Delta) w_*(t) :=-b(u_*(t;f)) - \frac{t}{t_0} \Delta(g-\ee^{t_0 \Delta}f) +\frac{1}{t_0}(g-\ee^{t_0 \Delta}f). 
\end{equs}
Then, it is easy to check that $u_*(\cdot;f) = u(\cdot;f,w)$, $u_*(t_0;f) = g$ and $\|u_*(t;f)\|_\infty \leq 2\|f\|_\infty + \|g\|_\infty$ 
for every $t\leq t_0$. Moreover due to the polynomial growth \eqref{eq:pol_growth}
of $b$, we have that for every $t\leq t_0$
\begin{equs}
 \|w_*(t)\|_\infty \lesssim_{t_0} (1+\|f\|_\infty + \|g\|_\infty)^{2p-1} + \|\Delta g\|_\infty + \|\Delta \ee^{t_0 \Delta} f\|_\infty<\infty. 
\end{equs}

Let $\theta :=r+\frac{R-r}{2}$, $\delta>0$, $\tilde w\in \mathrm{B}_{\C^0([0,t_0];\times\TT;\RR^n)}(w_*;\delta)$, 
$h\in B_{\|\cdot\|_\infty}(f;r)$ and define 
\begin{equs}
 \tau := \inf\{t>0: \, \|u(t;h,\tilde w) - u(t;f, w_*)\|_\infty >\theta\}>0.
\end{equs}
For $t\leq \tau\wedge t_0$ and $\tilde v(t):=u(t;h,\tilde w) - u(t;f,w_*)$, due to the polynomial growth \eqref{eq:pol_growth} of 
$\nabla_u b(u)$ and the fact that $\|u(t;f,w_*)\|_\infty \leq 2\|f\|_\infty +\|g\|_\infty$ for every $t\leq t_0$, we have that 
\begin{equs}
 \|\tilde v(t)\|_\infty & \leq \|h-f\|_\infty + \|\int_0^t \dd s \, \ee^{(t-s)\Delta}\big[b\big(u(s;h,\tilde w)\big) - b\big(u(s;f, w_*)\big)\big]\|_\infty + \|\tilde w(t) - w_*(t)\|_\infty
 \\
 & \leq \|h-f\|_\infty + \int_0^t \dd s \, \|b\big(u(s;h,\tilde w)\big) - b\big(u(s;f, w_*)\big)\|_\infty + \|\tilde w(t) - w_*(t)\|_\infty
 \\
 & \leq r + C t_0 (1+ \|f\|_\infty + \|g\|_\infty + \theta)^{2p-2} \theta + \delta < r+ \frac{R-r}{2} 
\end{equs}
provided we choose $t_0\equiv t_0(f,g, r, R)>0$ and $\delta\equiv \delta(R,r) >0$ sufficiently small. Hence, we must have $\tau\wedge t_0=t_0$ and
since $u(t_0;f,w_*) = g$ we know that $u(t_0;h,\tilde w)\in B_{\|\cdot\|_\infty}(g;R)$, which in turn implies that
\begin{equs}
 \Prob\left(u\big(t_0;B_{\|\cdot\|_\infty}(f;r),w\big)\in B_{\|\cdot\|_\infty}(g;R)\right)\geq 
 \Prob\left(w \in \mathrm{B}_{\C^0([0,t_0]\times\TT;\RR^n)}(w_*;\delta)\right) >0.
\end{equs}
\end{proof}

\paragraph{Local asymptotic stability} We make the following additional assumption on $V$,

\begin{assumption}\label{ass:V_quartic} The potential $V$ is of the form $V(u) = \frac{1}{4} |u|^4 + W(u)$ for some $W\in\C^4(\RR^n;\RR_{>0})$ with bounded third derivatives. 
\end{assumption}

This assumption allows to estimate the norm of the bilinear form $D^{(2)}u(t;f)$ on $L^2(\TT;\RR^n)\times L^2(\TT;\RR^n)$
and thereby to validate the assumptions of the local stable manifold theorem \cite[Proposition 1]{SV18}.

\begin{proposition}[Local asymptotic stability]\label{prop:local_stable_man} Under Assumptions~\ref{ass:V_general} and
\ref{ass:V_quartic}, assume that there exists $\lambda <0$ such that $\lambda_{\mathrm{top}}^{(\eps)}<\lambda$. Then, there exists measurable functions $0<\alpha(f, \omega) <\beta(f, \omega)$ such that for $\nu_\eps^{(\kappa)}\times \Prob$-almost every $(\omega,f)$
\begin{equs}
 \left\{g\in \bar B_{\|\cdot\|_2}(f;\alpha(f, \omega)) : \|u(t;g,\omega) - u(t;f,\omega)\|_2 \leq \beta(f, \omega) \ee^{\lambda t} \ \text{for all} \ t\geq0\right\} \label{eq:lsm}
\end{equs}
is a measurable neighbourhood of $f$.
\end{proposition}

\begin{proof} We first prove the claim for discrete times, that is, we replace \eqref{eq:lsm} by 
\begin{equs}
 \left\{g\in \bar B_{\|\cdot\|_2}(f;\alpha(f, \omega)) : \|u(k;g, \omega) - u(k;f,\omega)\|_2 \leq \beta(f, \omega) \ee^{\lambda k} \ \text{for all} \ k\in \mathbb{N}\right\}. 
\end{equs}
Using \cite[Proposition 1]{SV18}, the claim for discrete times follows if together with \eqref{eq:log_+_1} we verify the following estimate 
\begin{equs}
 \int \nu_\eps^{(\kappa)}(\dd f) \, \Exp \sup_{t\leq 1} \log^+\big(\|u(t;\cdot+f) - u(t;f)\|_{\C^2(\bar B_{\|\cdot\|_2}(0;1))}\big) \label{eq:log_+_2}
 < \infty. 
\end{equs}
In order to prove \eqref{eq:log_+_2}, we proceed similarly to \eqref{eq:log_+_1}, now working with $D^{(2)}u(t;f)(h_1,h_2)$ for $h_1,h_2\in L^2(\TT;\RR^n)$.
We note that\footnote{Here, we use the notation $(A_{ij})_{ij}$ to denote the matrix $A\in \RR^n\times \RR^n$ with elements $A_{ij}$.}
\begin{equs}
 (\partial_t - \Delta) D^{(2)}u(t;f)(h_1,h_2) & = -\left(\nabla\partial_{ij}V(u(t;f)) \cdot  Du(t;f)h_2\right)_{ij} Du(t;f)h_1
 \\
 & \quad - \nabla^2V(u(t;f)) D^{(2)}u(t;f)(h_1,h_2). 
\end{equs}
Testing the equation with $D^{(2)}u(t;f)(h_1,h_2)$ we are lead to 
\begin{equs} 
 & \frac{1}{2}\partial_t \|D^{(2)}u(t;f)(h_1,h_2)\|_2^2 + \|\nabla D^{(2)}u(t;f)(h_1,h_2)\|_2^2 
 \\
 & \quad = - \lng  \left(\nabla\partial_{ij}V(u(t;f)) \cdot  Du(t;f)h_2\right)_{ij} Du(t;f)h_1, D^{(2)}u(t;f)(h_1,h_2)\rng 
 \\
 &\quad \quad - \lng \nabla^2V(u(t;f)) D^{(2)}u(t;f)(h_1,h_2), D^{(2)}u(t;f)(h_1,h_2) \rng. \label{eq:2nd_der_test}
\end{equs}
In the special case where $V(u) = \frac{1}{4} |u|^4 + W(u)$, we rewrite \eqref{eq:2nd_der_test} as 
\begin{equs}
& \frac{1}{2}\partial_t  \|D^{(2)}u(t;f)(h_1,h_2)\|_2^2 + \|\nabla D^{(2)}u(t;f)(h_1,h_2)\|_2^2
\\
& \quad  = - \lng  \left(\nabla\partial_{ij}V(u(t;f)) \cdot  Du(t;f)h_2\right)_{ij} Du(t;f)h_1, D^{(2)}u(t;f)(h_1,h_2)\rng
\\
& \quad \quad  
- 2\|u(t;f)\cdot D^{(2)}u(t;f)(h_1,h_2)\|_2^2 - \||u(t;f)| |D^{(2)}u(t;f)(h_1,h_2)|\|_2^2
\\
&\quad \quad - \lng \nabla^2W(u(t;f)) D^{(2)}u(t;f)(h_1,h_2), D^{(2)}u(t;f)(h_1,h_2) \rng.
\end{equs}
By Assumption~\ref{ass:V_quartic} we know that $|\nabla\partial_{ij}V(u)| + |\nabla\partial_{ij}W(u)|\lesssim 1+|u|$. Therefore, we have the estimates 
\begin{equs}
 & |\lng  \left(\nabla\partial_{ij}V(u(t;f)) \cdot  Du(t;f)h_2\right)_{ij} Du(t;f)h_1, D^{(2)}u(t;f)(h_1,h_2)\rng| 
 \\
 & \quad \lesssim \int \dd x \, (1+|u(t;f)|) |Du(t;f)h_2| |Du(t;f)h_1| |D^{(2)}u(t;f)(h_1,h_2)| 
 \\
 & \quad \leq \frac{1}{2} \|(1+|u(t;f)|) |D^{(2)}u(t;f)(h_1,h_2)|\|_2^2 + C \||Du(t;f)h_1| |Du(t;f)h_2|\|_2^2 
\end{equs}
and
\begin{equs}
 & |\lng \nabla^2W(u(t;f)) D^{(2)}u(t;f)(h_1,h_2), D^{(2)}u(t;f)(h_1,h_2) \rng| 
 \\
 & \quad \lesssim \int \dd x \, (1+|u(t;f)|)  |D^{(2)}u(t;f)(h_1,h_2)|^2 
 \\
 & \quad \leq \frac{1}{2} \|(1+|u(t;f)|) |D^{(2)}u(t;f)(h_1,h_2)|\|_2^2 + C \|D^{(2)}u(t;f)(h_1,h_2)\|_2^2.
\end{equs}
Hence, we have proved that
\begin{equs}
 & \frac{1}{2} \partial_t\|D^{(2)}u(t;f)(h_1,h_2)\|_2^2 + \|\nabla D^{(2)}u(t;f)(h_1,h_2)\|_2^2
 \\
 & \quad \lesssim \|D^{(2)}u(t;f)(h_1,h_2)\|_2^2 + \||Du(t;f)h_1| |Du(t;f)h_2|\|_2^2, 
\end{equs}
which together with the Sobolev embedding $W^{1,1}(\TT;\RR) \hookrightarrow L^2(\TT;\RR)$ in the form
\begin{equs}
 \||Du(t;f)h_1| |Du(t;f)h_2|\|_2^2 l\lesssim \sum_{i\neq j} \|Du(s;f)h_i\|_2^2 \big(\|Du(s;f)h_j\|_2^2+ \|\nabla Du(s;f)h_j\|_2^2\big)
\end{equs}
yields
\begin{equs}
 \|D^{(2)}u(t;f)(h_1,h_2)\|_2^2 \lesssim \sum_{i\neq j} \int_0^t \dd s \, \ee^{2C(t-s)}  \|Du(s;f)h_i\|_2^2 \left(\|Du(s;f)h_j\|_2^2+ \|\nabla Du(s;f)h_j\|_2^2\right).
\end{equs}
Plugging in \eqref{eq:energy_est} we obtain that
\begin{equs}
 \int_0^t \dd s \, \ee^{2C(t-s)} \,  \|Du(t;f)h_i\|_2^2  \left(\|Du(s;f)h_j\|_2^2+ \|\nabla Du(s;f)h_j\|_2^2\right)
 \leq \ee^{4Ct} \|h_i\|_2^2 \|h_j\|_2^2, 
\end{equs}
which gives
\begin{equs}
  \|D^{(2)}u(t;f)(h_1,h_2)\|_2 \lesssim \ee^{2Ct} \|h_1\|_2 \|h_2\|_2. 
\end{equs}

We now extend the result to continuous times. For $t\in [k,k+1)$ and $h_k=u(k;g)- u(k;f)$ using
\eqref{eq:flow_property} we write
\begin{equs}
 \|u(t;g) - u(t;g)\|_2 & = \|u_k(t-k; u(k;g)) - u_k(t-k; u(k;f))\|_2
 \\
 & \leq \int_0^1 \dd \lambda \|Du_k(t-k; u(k;f)+\lambda h_k)h_k\|_2.  
\end{equs}
By \eqref{eq:energy_est} with $u$ replaced by $u_k$ we know that
\begin{equs}
 \int_0^1 \dd \lambda \|Du_k(t-k; u(k;f)+\lambda h_k)h_k\|_2 \leq \ee^{C(t-k)} \|h_k\|_2 \leq \ee^{C} \beta(f,\omega) \ee^{\lambda k} \leq \ee^{C-\lambda} \beta(f,\omega) \ee^{\lambda t}
\end{equs}
and the claim follows since $C$ and $\lambda$ are deterministic. 
\end{proof}

\paragraph{Proof of Theorem~\ref{thm:synchr}} By Proposition~\ref{prop:weak_attr_comp} there exists a weak pull-back attractor $A$ which is compact in $\C^0(\TT;\RR^n)$.
By Propositions~\ref{prop:contr_large_sets} and \ref{prop:swift_trans} in combination with \cite[Theorem 2.14-1]{FGS14} we
know that $A$ has small diameter, namely for every $\delta>0$
\begin{equs}
\Prob\left(\mathrm{diam}_{\|\cdot\|_\infty}(A) <\delta\right)>0. 
\end{equs}
By Proposition~\ref{prop:lyap_exp_formulae}, Theorem~\ref{thm:main_1}/Theorem~\ref{thm:main_2} and \eqref{eq:equal_law} there exist $\lambda <0$ such that $\lambda_{\mathrm{top}}^{(\eps)}\leq \lambda$. Therefore, by Proposition~\ref{prop:local_stable_man} and 
\cite[Corollary 1]{SV18} $u$ is asymptotically stable in $L^2(\TT;\RR^n)$, that is, there exist a deterministic open set
$U\subset L^2(\TT;\RR^n)$ and $t_n\nearrow \infty$ such that
\begin{equs}
\Prob\big(\lim_{n\nearrow\infty}\mathrm{diam}_{\|\cdot\|_2}\big(u(t_n;U)\big) = 0 \big)>0.
\end{equs}
Finally, \cite[Theorem 2.14-2]{FGS14} implies that $\mathrm{diam}_{\|\cdot\|_2}A =0$, 
that is, $A=\{a(\omega)\}$ for $a(\omega)\in\C^0(\TT;\RR^n)$ and synchronisation occurs. \qed

\begin{appendices} 

\addtocontents{toc}{\protect\setcounter{tocdepth}{0}}

\section{White noise RDS and invariant measures} \label{app:RDS}

We recall the following definitions and notations for stochastic flows and random dynamical systems (RDS). In what follows
$(\Omega, \mathcal{F}, \Prob, \{\theta_t\}_{t\in \RR})$ denotes a metric dynamical system, that is, $(t,\omega)\mapsto \theta_t(\omega)$ is measurable, $\theta_0\omega=\omega$, $\theta_{t+s}=\theta_t\theta_s$ for every $t,s\in \RR$ and 
$\theta_t$ is $\Prob$-preserving for every $t\in \RR$.  

\begin{definition}[Stochastic flow and (white noise) RDS]\label{def:rds} A stochastic flow on a separable
Banach space $\mathcal{B}$ is a family $\{\varphi(t,s,\omega)\}_{t\geq s, \omega\in \Omega}$ of maps 
$\mathcal{B} \ni f\mapsto \varphi(t,s,\omega) f\in \mathcal{B}$ such that every $\omega\in \Omega$,
\begin{enumerate}[i.] 
\item $\varphi(s,s,\omega) =\mathrm{id}$ for every $s\in\RR$,
\item $\varphi(t,s,\omega)f = \varphi(t,r,\omega)\varphi(r,s,\omega)f$ for every $t\geq r\geq s$, $f\in \mathcal{B}$.
\end{enumerate}

A stochastic flow $\{\varphi(t,s,\omega)\}_{s\leq t, \omega\in \Omega}$ is called:
\begin{enumerate}[i.] \setcounter{enumi}{2}
\item Measurable if the map $(t,\omega,f) \mapsto \varphi(t,s,\omega) f$ is measurable for every $s\leq t$.
\item Continuous if $(t,f) \mapsto \varphi(t,s,\omega) f$ is continuous on $\{t: t\geq s\}\times \mathcal{B}$ for every 
$\omega\in \Omega$. 
\end{enumerate}

A measurable and continuous stochastic flow $\{\varphi(t,s,\omega)\}_{t\geq s,\omega\in \Omega}$ on 
$(\Omega, \mathcal{F}, \Prob, \{\theta_t\}_{t\in \RR})$ which satisfies the cocycle property 
\begin{equs}
\varphi(t,s,\omega)f = \varphi(t-s,0,\theta_s\omega)f
\end{equs}
for every $f\in\mathcal{B}$, $\omega\in \Omega$ and $t\geq s$ is called a random dynamical system (RDS). If, in addition,  
$\sigma(\{\varphi(t,0,\cdot)f: \, f\in \mathcal{B}, \, t\geq 0\})$ and $\sigma(\{\varphi(0,r,\cdot)f: \, f\in \mathcal{B}, \, r\leq 0\})$ are independent we say that $\{\varphi(t,s,\omega)\}_{t\geq s,\omega\in \Omega}$ is a white noise RDS. 

In the case of a (white noise) RDS, we restrict to the case $s=0$ and for simplicity write $\{\varphi(t,\omega)\}_{t\geq 0, \omega\in \Omega}$, dropping the dependence on $s=0$.  
\end{definition} 

In what follows, we argue that \eqref{eq:reaction-diffusion_system} gives rise to a white noise RDS on $\C^0(\TT;\RR^n)$ which can be extended to $L^2(\TT;\RR^n)$. Recall that for $s\in \RR$,  $u_s = w_s + v_s$, where $w_s\in \C^0([s,T];\C^\alpha(\TT;\RR^n))$ and $v_s$ solves \eqref{eq:remainder}.  

We first argue that \eqref{eq:remainder} has a unique global solution. For simplicity, let us restrict to the case $s=0$.
Given $f\in \C^0(\TT;\RR^n)$ and $w\in \C^0([0,T]\times \TT;\RR^n)$, the existence and uniqueness of mild solutions in 
$\C^0([0,T_*]\times\TT;\RR^n)$ for some $T_*\equiv T_*(\|f\|_\infty, \|w\|_{\C^0([0,T]\times \TT;\RR^n)})\leq T$ follows by a standard fixed-point argument. In order 
to prove global in time existence we first note that by classical energy estimates we have for every $q\geq 2$, 
\begin{equs}
\partial_{t}\frac{1}{q}\|v\|_q^q & = \int \dd x \, v(x)^{[q-1]} \cdot \big(\kappa\Delta v(x) + b(w+v)(x)\big)
\\
 & =-(q-1)\kappa \int \dd x \, |v(x)|^{q-2}|\nabla v(x)|^{2} + \int \dd x \, v(x)^{[q-1]} \cdot b(v)(x)
 \\
 & \quad +\int \dd x \, v(x)^{[q-1]} \cdot \big(b(w+v)-b(v)\big)(x). 
\end{equs}
Therefore, by \eqref{eq:coerc_grad_V}, the mean value theorem and the polynomial growth \eqref{eq:pol_growth} we obtain that
\begin{equs}
 \partial_{t}\frac{1}{q}\|v\|_q^q & \leq \int \dd x \, |v(x)|^{q-2}v(x) \cdot b(v(x) )+\int \dd x \, v(x)^{[q-1]}\cdot 
 \big(b(w+v)-b(v)\big)(x)\\
 & \leq \int \dd x \, |v(x)|^{q-2}(C-|v(x)|^{2p})+ \int \dd x \, |v(x)|^{q-1} \int_0^1 \dd \lambda |\nabla b(v +\lambda w)(x) w(x)|
 \\
 & \le C\|v\|_{q-2}^{q-2} - \|v\|_{2p+q-2}^{2p+q-2} 
 + C \int \dd x \, |v(x)|^{q-1} (1+|w(x)|^{2p-2}+|v(x)|^{2p-2}) |w(x)|\\
 & \le C\|v\|_{q-2}^{q-2} - \|v\|_{2p+q-2}^{2p+q-2} 
 +C\|w\|_{\infty} \|v\|_{q-1}^{q-1} + C \|w\|_{\infty}^{2p-1} \|v\|_{q-1}^{q-1} + C\|w\|_{\infty} \|v\|_{q+2p-3}^{q+2p-3}. 
\end{equs}
Combining with Young's inequality we find $r_q\geq 1$ such that
\begin{equs}
 \partial_{t}\frac{1}{q}\|v\|_q^q \leq C - \frac{1}{2} \int|v|{}^{q-2+2p} + \|w\|_{\infty}^{r_q}.
\end{equs}
Hence, we have shown that
\begin{equs}
\|v\|_{L^\infty([0,T];L^q(\TT;\RR^n))} \lesssim \|f\|_{\infty} + \|w\|_{\infty}^{r_q} + 1, 
\end{equs}
which implies that for every $q\geq 1$, $g:=b(w+v)\in \C^0([0,T];L^q(\TT;\RR^n))$.  So we can consider $v$ as a solution to 
\begin{equs}
(\partial_{t}-\kappa\Delta) v = g, \quad v|_{t=0} = f,
\end{equs}
with $g\in L^\infty([0,T];L^q(\TT;\RR^n))$ for $q$ large enough to obtain via heat kernel estimates
\begin{equs}
 \|v\|_{\C^0([0,T]\times\TT;\RR^n)} \lesssim \|f\|_\infty + \|w\|_{\infty}^{r_q} + 1.  
\end{equs}

In order to prove stability in $f$ and $w$, we consider two solutions $v_1$ and $v_2$ to
\begin{equs}
(\partial_{t} - \kappa\Delta) v_{i} = b(w_{i}+v_{i}), \quad v_{i}|_{t=0} = f_{i}. 
\end{equs}
Writing $\tilde v:=v_1-v_2$ and $\tilde w := w_1  - w_2$, $\tilde v$ satisfies
\begin{equs}
(\partial_{t} - \kappa\Delta)\tilde{v} = b(w_1+v_2)-b(w_1+v_2), \quad \tilde{v}|_{t=0} =f_1-f_2.
\end{equs}
Let $t>0$ and assume that $\|\tilde{v}(t)\|_{\infty}$ does not vanish. We can choose $y\in \TT$ such that 
$|\tilde v(t,y)| = \|\tilde{v}(t)\|_{\infty}$ and arguing as in the proof of Proposition~\ref{prop:contr_large_sets} 
based on Lemma~\ref{lem:sup_energy_est} and using in addition the mean value theorem and \eqref{eq:pol_growth}
we obtain that
\begin{equs}
 \partial_{t}^{-}\|\tilde{v}(t)\|_{\infty} & \leq (b(w_{1}+v_{1})-b(w_{2}+v_{2}))(t,y) \cdot \frac{\tilde{v}(t,y)}{|\tilde{v}(t,y)|}
 \\
 & \leq \int_0^1 \dd \lambda \, |\nabla b(w_{2}+v_{2} + \lambda (\tilde w +\tilde v))(t,y)(\tilde w(t,y)+\tilde v(t,y)) \cdot \frac{\tilde{v}(t,y)}{|\tilde{v}(t,y)|}|
 \\
 & \leq C (\|\tilde w(t)\|_{\infty} + \|\tilde{v}(t)\|_{\infty})
\end{equs}
for a constant $C$ which depends on $\|w_{i}\|_{\C^0([0,T]\times \TT;\RR)}$ and $\|v_{i}(t)\|_{\C^0([0,T]\times \TT;\RR)}$. 
Applying Lemma~\ref{lem:left_gronwall} to
$\phi(t) := \ee^{-Ct} \|\tilde{v}(t)\|_{\infty} + C \int_0^t \dd s \, \ee^{-Cs} \|\tilde w(s)\|_{\infty}$, for every $t\leq T$ we get that
\begin{equs}
 \|\tilde{v}\|_{\infty}(t) & \leq  \ee^{Ct}  \|f_1 - f_2\|_{\infty} +  C \int_0^t \dd s \, \ee^{C (t-s)}  \|\tilde w_1(s) - \tilde w_2(s)\|_{\infty}. 
\end{equs}

We now argue that $(t,s,\omega,f) \mapsto u_s(t;f,w_s(\cdot;\omega))$ gives rise to a stochastic flow on 
$\mathcal{B}=\C^0(\TT;\RR^n)$. The flow property $\varphi(t,s,\omega)f = \varphi(t,r,\omega)\varphi(r,s,\omega)f$
follows from the identity
\begin{equs}
u_s(t;f,w_s(\cdot;\omega)) = w_r(t;\omega) + \ee^{(t-r)\kappa \Delta} w_s(t;\omega) + v_s(t;f;w_s(\cdot;\omega)),
\end{equs}
noting that $\tilde v(t) = \ee^{(t-r)\kappa \Delta} w_s(t;\omega) + v_s(t;f;w_s(\cdot;\omega))$, $t\geq r$, solves
\begin{equs} 
 \begin{cases}
  & (\partial_t - \kappa\Delta) \tilde v = b(w_r(\cdot;\omega)+\tilde v) \quad \text{on} \ (r,\infty)\times \TT, \\
  & \tilde v|_{t=r} = u_s(r;f;w_s(\cdot;\omega)).
 \end{cases}
\end{equs}
Therefore by uniqueness of solutions to \eqref{eq:remainder} we deduce that $\tilde v(t) = v_r(t;u_s(r;f,w_s(\cdot;\omega)),w_r(\cdot;\omega))$, which implies that 
\begin{equs}
u_r(t;u_s(r;f,w_s(\cdot;\omega)),w_r(\cdot;\omega)) = u_s(t;f,w_s(\cdot;\omega)). \label{eq:flow_property}
\end{equs}
Due to the continuity of $(f,w_s)\mapsto v_s(\cdot;f,w_s)$ and the fact that $w_s(t;f,w_s)$ and $v_s(t;f,w_s)$ are continuous in $t$,
we see that $(t,s,\omega,f) \mapsto u_s(t;f,w_s(\cdot;\omega))$ is measurable and continuous in $(t,f)$. 

For $t\in \RR$ we define $\theta_t:\Omega \to \Omega$ via
\begin{equs}
(\theta_t\omega)(\psi): = \omega\big(\psi(\cdot-t, \cdot)\big), \quad \text{for every} \ \psi\in \mathscr{S}(\RR\times\TT;\RR^n).  
\end{equs}
Note that $\theta_0\omega = \omega$, $\theta_t$ leaves $\Prob$ invariant since space-time white noise $\xi$ is stationaty
and $\theta_{t+s} = \theta_t \theta_s$. Therefore $(\Omega, \mathcal{F}, \Prob, \{\theta_t\}_{t\in \RR})$ is a metric dynamical system with the natural filtration $\{\mathcal{F}_{s,t}\}_{s\leq t}$ given by
\begin{equs}
\mathcal{F}_{s,t} := \sigma\big(\{\omega(\psi) : \psi\in\mathscr{S}(\RR\times\TT;\RR^n), \ \psi(r,x) = 0 \ \text{for every} \ r\notin [s,t],x\in \TT\}\big)  
\end{equs}
In order to prove the cocycle property, we notice that
\begin{equs}
 w_s(t,x;\omega) & = \omega\big(\mathbf{1}_{[s,t]} H_{t-\cdot}^{(\kappa)}(x-\cdot)\big)
 =\omega\big(\mathbf{1}_{[0,t-s]}(\cdot-s) H_{t-s-(\cdot -s)}^{(\kappa)}(x-\cdot)\big)
 \\
 & =(\theta_s\omega)\big(\mathbf{1}_{[0,t-s]} H_{t-s-\cdot}^{(\kappa)}(x-\cdot)\big) = w(t-s; \theta_s\omega).
\end{equs}
Therefore, using the last identity we have that
\begin{equs}
  u_s(t;f,w_s(\cdot;\omega)) & = w_s(t;\omega) + v_s(t;f, w_s(\cdot;\omega))
  \\
  &  = w(t-s;\theta_s\omega) + v_s(t;f, w(\cdot-s;\theta_s\omega))
  \\
  & = w(t-s;\theta_s\omega) + v(t-s;f, w(\cdot-s;\theta_s\omega)) = u(t-s;f,w_s(\cdot;\theta_s\omega)),
\end{equs}
where the equality $v_s(t;f, w(\cdot-s;\theta_s\omega)) = v(t-s;f, w(\cdot-s;\theta_s\omega))$ follows from uniqueness 
since $\tilde v(t) = v_s(t;f, w(\cdot-s;\theta_s\omega))$ for $t\geq s$ solves
\begin{equs} 
 \begin{cases}
  & (\partial_t - \kappa\Delta) \tilde v = b(w(\cdot-s;\omega)+\tilde v) \quad \text{on} \ (s,\infty)\times \TT, \\
  & \tilde v|_{t=s} = f.
 \end{cases}
\end{equs}

To deduce that we have a a white noise RDS note that $u_s(\cdot;f,w_s(\cdot;\omega))$ is adapted to 
$\{\mathcal{F}_{s,t}\}_{t\geq s}$ and that $\mathcal{F}_{r,s}$
is independent of $\mathcal{F}_{s,r}$ for every $s\leq r \leq t$. Therefore, independence of the $\sigma$-algebras 
$\sigma(\{u_s(t;f,w_s): \, f\in \mathcal{B}, \, t\geq 0\})$ and 
$\sigma(\{u_s(r;f,w_s): \, f\in \mathcal{B}, \, r\leq 0\})$ follows from 
the independence of $\mathcal{F}_{0,t}$ and $\mathcal{F}_{r,0}$ for every .

Finally, we argue that $(f,w)\mapsto v(\cdot;f,w)$ is continuous from $L^2(\TT;\RR^n)\times\C^0\left([0,T];\C^0(\TT;\RR^n)\right)$ to $\C^0\left([0,T];L^2(\TT;\RR^n)\right)$, which allows us to consider the white noise RDS on $L^2(\TT;\RR^n)$.
We consider a sequence $\{f_N\}_{N\geq 1}\subset \C^0(\TT;\RR^n)\cap L^2(\TT;\RR^n)$ which converges 
to $f$ in $L^2(\TT;\RR^n)$. Given $T>0$ and $M, N\geq 1$ by  \eqref{eq:energy_est} we know that for every $\omega\in \Omega$
\begin{equs}
\sup_{t\leq T} \|u_s(t;f_M,w_s(\cdot;\omega)) - u(t;f_N,w_s(\cdot;\omega))\|_2 \leq \ee^{CT} \|f_M - f_N\|_2,
\end{equs}
therefore $\{u_s(\cdot;f_N,w_s(\cdot;\omega))\}_{N\geq 1}$ is Cauchy in $\C^0\big([0,T];L^2(\TT;\RR^n)\big)$, 
and converges to a limit 
$u(\cdot;f,w_s(\cdot;\omega))\in \C^0\big([0,T];L^2(\TT;\RR^n)\big)$ which does not depend on
the specific choice of $\{f_N\}_{N\geq 1}$. 

Similarly, using again \eqref{eq:energy_est}, solutions to the linear equations \eqref{eq:linearised_system} and \eqref{eq:linearised_system_rescaled} 
can be extended to $\C^0\big([0,\infty);L^2(\TT;\RR^n)\big)$ for initial conditions $h\in L^2(\TT;\RR^n)$.

Following \cite[Theorem~9.15]{DPZ92} the solution $(t,f) \mapsto u(t;f) = w(t) + v(t;f,w)$ of \eqref{eq:reaction-diffusion_system} generates a strong Markov process. We now argue that \eqref{eq:gibbs} is the unique invariant measure for the associated Markov semigroup. The fact that \eqref{eq:gibbs} is invariant for the generated Markov process follows from \cite[Section~8.6]{DPZ96}. One possible way to prove the uniqueness of invariant measures is via Doob's Theorem \cite[Theorem~4.2.1]{DPZ96}, relying on the strong Feller property and irreducibility \cite[Section~4.1]{DPZ96}. In the present framework the strong Feller property
holds for functions $F:L^2(\TT;\RR^n)\to \RR$ due to the non-degeneracy of $\xi$. It follows from \cite[Theorem~7.1.1]{DPZ96} based on \eqref{eq:energy_est},
which provides an estimate on the $L^2$-operator norm of the linearisation with respect to the initial data $f$. For 
irreducibility we refer the reader to \cite[Section~7.4]{DPZ96}. 
 
\section{The Weingarten map} \label{app:weingarten}

In \eqref{eq:G_fmm} the Weingarten map or shape operator $W:(w, v)\in \N\M \mapsto W_{w,v}$ appears. Here 
$\N\M:=\{(w,v): w\in \M, \, v\in \N_w\}$ is the normal bundle and $W_{w,v}$ is the linear transformation on $\T_w$
defined through
\begin{equs}
 \lng W_{w,v}(h), g \rng = \lng D_h \tilde g(w), v \rng, \quad \text{for every } h,g\in \T_w, \label{eq:weingarten_proper}
\end{equs}
where $\tilde g$ is any tangent vector field with $\tilde g(w) = g$, $D_h$ stands for the directional derivative in the direction $h$ and 
$\lng \cdot,\cdot \rng$ stands for the inner product on $\T_w \oplus \N_w$ inherited by $L^2(\TT;\RR^n)$. By Leibniz rule and the fact that $\T_w \perp \N_w$, we have that 
\begin{equs}
 \lng W_{w,v}(h), g \rng = - \lng g, D_h\tilde v(w) \rng, \label{eq:weingarten_leibniz}
\end{equs}
where $\tilde v$ is any normal vector field with $\tilde v(w) = v$. 

\begin{example} In the case of the sphere $\Ss^{n-1}_r$ of radius $r>0$, seen as a manifold embedded in $L^2(\TT;\RR^n)$, the Weingarten map 
$W:(w, v)\in \N\Ss^{n-1}_r \mapsto W_{w,v}$, is the linear transformation on $\T_w$ given by 
\begin{equs}
 W_{w,v}(h) := - \frac{1}{r^2}\lng w, v \rng \, h, \quad \text{for every } h\in \T_w. \label{eq:weingarten}
\end{equs}
Indeed, let $w\in \Ss^{n-1}_r$ and $v \in \N_{w}$. Then $v= \frac{1}{r^2}\lng w, v \rng w + v^{\perp}$, 
where $v^{\perp}\in L^2(\TT;\RR^n)$ with $\int \dd x \, v^{\perp}(x)=0$. We now define the normal vector field $\tilde v: \tilde w\mapsto \frac{1}{r^2}\lng w, v \rng \tilde w + v^{\perp}$
which satisfies $\tilde v(w) = v$. For $h\in \T_{w}$, let $\gamma:[0,1]\to \Ss^{n-1}$ such that $\gamma(0) = w$ and $\gamma'(0) = h$. 
A simple calculation shows that
\begin{equs}
 D_h\tilde v(w) = \frac{\dd}{\dd t} \tilde v \circ \gamma(t)\big|_{t=0} = \frac{\dd}{\dd t} \left(\frac{1}{r^2}\lng w, v\rng \gamma(t) + v^{\perp}\right)\big|_{t=0}
 = \frac{1}{r^2}\lng w, v \rng h
\end{equs}
and combining with \eqref{eq:weingarten_leibniz} we obtain \eqref{eq:weingarten}. 
\end{example}

\section{Estimates in the supremum norm}

\begin{lemma} \label{lem:sup_energy_est} Let $g\in\C^0([0,T]\times \TT;\RR^n)$ and suppose we are given solution $v$ to 
\begin{equs}
 (\partial_t - \Delta) v = g. 
\end{equs}
Assume furthermore that there exists $t\in(0,T]$ and $y\in \TT$ such that $\|v(t)\|_\infty\neq 0$, $|v(t,y)| = \|v(t)\|_\infty$, 
and 
\begin{equs}
 \partial_t^-\|v(t)\|_\infty \leq \partial_t v(t,y)\cdot \frac{v(t,y)}{|v(t,y)|}.
\end{equs}
Then, we have that 
\begin{equs}
 \partial_t^- \|v(t)\|_\infty \leq g(t,y)\cdot \frac{v(t,y)}{|v(t,y)|}.
\end{equs}
\end{lemma}

\begin{proof} We fist notice that by assumptions we have that
\begin{equs}
 \partial_t^- \|v(t)\|_\infty \leq \Delta v(t,y) \cdot \frac{v(t,y)}{|v(t,y)|} 
 + g(t,y)\cdot \frac{v(t,y)}{|v(t,y)|}.
\end{equs}
Hence, it suffices to show that $\Delta v(t,y) \cdot \frac{v(t,y)}{|v(t,y)|}\leq0$. Let $h(y) := |v(t,y)|^2$. 
Then $h(y) \geq h(x)$ for any $x\in \TT$. Therefore 
\begin{equs}
 0 \geq \Delta h(y) & = 2 \sum_{m=1}^n |\partial_y v_m(t,y)|^2 + 2 \sum_{m=1}^n v_m(t,y) \Delta v_m(t,y) 
 \\
 & \geq 2 \sum_{m=1}^n v_m(t,y) \Delta v_m(t,y)  = 2 v(t,y) \cdot \Delta v(t,y),
\end{equs}
which proves the desired estimate.  
\end{proof}

\begin{lemma} \label{lem:left_gronwall} For every continuous function $\phi:[0,\infty)\to \RR$ satisfying 
$\partial_t^- \phi(t) \leq 0$ for all $t>0$, we have that $\phi(t) \leq \phi(0)$ for all $t>0$. 
\end{lemma}

\begin{proof} For $r>0$, let $f(t):=\phi(t) - rt$. Then $\partial_t^-f(t) \leq -r <0$ for all $t>0$. 
Assume that there exists $t_*>0$ such that $f(t_*) > f(0)$ and set $\delta = \frac{f(t_*) - f(0)}{2}$. 
Let $t_0:= \inf\{t>0: \, f(t) >\delta\}$. Then $f(t_0) = \delta$ and for all $t\leq t_0$, $f(t)\leq f(t_0)$.  
Hence, we get that 
\begin{equs}
 \frac{f(t_0) - f(t)}{t_0 - t} \geq 0,
\end{equs}
for all $t\leq t_0$, which in turn implies that $\partial_t^- f(t_0) \geq 0$, yielding a contradiction 
since $\partial_t^- \phi(t)<0$ for all $t>0$. Thus, $f(t) = \phi(t) - rt \leq f(0) = \phi(0)$ 
for all $t>0$ and since $r$ is arbitrary we conclude that $\phi(t)\leq \phi(0)$. 
\end{proof}

\end{appendices}

\addtocontents{toc}{\protect\setcounter{tocdepth}{1}}

\pdfbookmark{References}{references}
\addtocontents{toc}{\protect\contentsline{section}{References}{\thepage}{references.0}}

\bibliographystyle{plainurl}
\bibliography{synchr_bibliography}{}

\begin{flushleft}
\footnotesize \normalfont
\textsc{Benjamin Gess\\
Fakult\"at f\"ur Mathematik, Universit\"at Bielefeld\\
Bielefeld, D-33615\\
Max--Planck--Institut f\"ur Mathematik in den Naturwissenschaften\\ 
Leipzig, D-04103} \\
\texttt{\textbf{bgess@math.uni-bielefeld.de}}
\end{flushleft}

\begin{flushleft}
\footnotesize \normalfont
\textsc{Pavlos Tsatsoulis\\
Fakult\"at f\"ur Mathematik, Universit\"at Bielefeld \\
Bielefeld, D-33615}\\
\texttt{\textbf{ptsatsoulis@math.uni-bielefeld.de}}
\end{flushleft}

\end{document}